\documentclass[letterpaper, 10 pt, conference]{ieeeconf}  

\IEEEoverridecommandlockouts                              
\usepackage{amsmath,amsfonts}
\usepackage{graphics,graphicx}
\usepackage{subcaption}
\usepackage{units}
\usepackage{tikz}
\usepackage{ifthen}
\usetikzlibrary{shapes}
\usetikzlibrary{positioning}

\usepackage[english]{babel}

\usepackage[hidelinks]{hyperref}

\usepackage[amsmath]{ntheorem}
\theoremstyle{plain}

\newtheorem{hyp}{Assumption}
\newtheorem{prop}{Proposition}
\newtheorem{lemma}{Lemma}

\newtheorem{rem}{Remark}

\newtheorem{corollary}{Corollary}
\newtheorem{theorem}{Theorem}

\newcommand{\NM}[1]{{\color{red} #1}}

\newboolean{LONG} 
\newboolean{CONF} 

\title{\LARGE \bf
Uniform \ifthenelse{\boolean{CONF}}{\NM{quasi}}{non}-convex optimisation via Extremum Seeking}

\author{Nicola Mimmo, Lorenzo Marconi and Giuseppe Notarstefano
\thanks{This research was partially supported by the European Project "AerIal RoBotic technologies for professiOnal seaRch aNd rescuE" (AirBorne), Call: H2020, ICT-25-2016/17, Grant Agreement no: 780960.}
\thanks{The authors are with the Department of Electrical and Information Engineering,  "Guglielmo Marconi", University of Bologna, 40126 Bologna, Italy, e-mail: 
        {\tt\small \{nicola.mimmo2, lorenzo.marconi,  giuseppe.notarstefano\}@unibo.it}}%
}

\begin{document}

\setboolean{LONG}{true}
\setboolean{CONF}{false}

\tikzstyle{triangle} = [draw, regular polygon, isosceles triangle]

\maketitle
\thispagestyle{empty}
\pagestyle{empty}

\begin{abstract}
The paper deals with a well-known extremum seeking scheme by proving uniformity properties with respect to the amplitudes of the dither signal and of the cost function. Those properties are then used to show that the scheme guarantees the global minimiser to be semi-global practically stable despite the presence of local \ifthenelse{\boolean{CONF}}{saddle points}{minima}. \ifthenelse{\boolean{CONF}}{}{Under the assumption of a   globally Lipschitz cost function,  it is shown that  the scheme, {improved through a high-pass filter}, makes the global minimiser practically stable with a global domain of attraction.} \ifthenelse{\boolean{CONF}}{To achieve these results, we analyse the average system associated with the extremum seeking scheme via arguments based on the Fourier series.}{}
\end{abstract}
\section{Introduction}

The early research on the Extremum Seeking (ES) dates back to the 1920s \cite{DOCHAIN2011369} and since then this strategy has been extensively exploited to solve several optimisation problems in electronics \cite{Toloue2017Multivariable}, mechatronics \cite{Malek2016Fractional}, mechanics \cite{Zhang2007Numerical}, aerodynamics \cite{Wang2000Experimental}, thermohydraulics \cite{Burns2020Proportional}, and thermoacoustic \cite{Moase2010Newton}.  Some of the most popular ES schemes are those proposed in \cite{tan2006non, KRSTIC2000595}, which represent the subject of the proposed analysis, although a remarkable variety of schemes were proposed, such as the adoption of integral action in \cite{guay2016perturbation}, the use of a cost function's parameter estimator \cite{Guay2014Minmax,Nesic2013Framework}, the introduction of an observer \cite{HAZELEGER2020109068}, the extension to fractional derivatives in \cite{Malek2016Fractional}, the use of a predictor to compensate output delays in \cite{Oliveira2017Extremum}, the implementation of a Newton-based algorithm avoiding the Hessian matrix inversion \cite{LABAR2019356, Oliveira2020Multivariable}, and the concurrent use of a simplex-method to find the global minimiser \cite{Zhang2016Simplex}.

All the methods, in a way or another, share the common philosophy of perturbing the system  subject to optimisation using the so-called \textit{dither} signal, which is periodic in most of the proposed solutions,  to unveil in which direction the associated cost function decreases (in the case of a minimisation problem).  The analysis tool that is typically adopted to prove {\em practical} convergence to the minimiser exploits the \textit{averaging theory}  \cite{Sanders1985Averaging}{. It is shown} that the \textit{average} system {\em asymptotically} converges to the minimiser and that the trajectories of the original system {remain} closed to the \textit{average} ones if a design parameter{,   namely $\gamma$,} is kept sufficiently small  \cite{Teel2003unified,TAN2005550,KRSTIC2000595,TEEL2000317,TEEL1999329}. 

\ifthenelse{\boolean{CONF}}{}{A further fundamental parameter of most of the ES schemes is represented by the dither amplitude, named $\delta$ from now on. Roughly, it represents the local range in which the cost function is evaluated to estimate the local cost variation and so the decreasing direction. Differently from $\gamma$, $\delta$ is subject to a design compromise \cite{Nesic2006Choice}. From one hand, the smaller $\delta$ the more the estimation of the local cost variation approximates the local cost gradient \cite{TAN2005550, DURR20131538}. Moreover, assuming the {current optimisation variable} sufficiently close to the actual minimiser, the smaller $\delta$ the better is the cost optimisation. On the other hand, as far as local minima are conceived, small $\delta$ could trap the scheme at a local minimiser due to a local inversion of the cost gradient. So, to overcome this issue, the dither amplitude should be 	sufficiently large to make the estimation of the cost variation robust in respect of local cost fluctuations \cite{SUTTNER2022Robustness}. For this reason, some authors suggested an adaptation policy for the dither amplitude which shrinks $\delta$ over time so that at the beginning a large dither lets the ES overcome local minima while asymptotically a small dither guarantees good optimisation performances \cite{TAN2009245,BHATTACHARJEE2021Extremum}.}

For sake of completeness, it is worth mentioning the results in \cite{scheinker2014extremum,ZHU2022109965} in which the value of the cost function is directly used to perturb the phase of the dither signal rather than to estimate the local gradient. 
Lie-derivative arguments, instead of averaging techniques, are then adopted in the analysis.

Within the previous research context, this paper proposes \ifthenelse{\boolean{CONF}}{{two}}{three} contributions.

First, {We propose to study the average systems presented in \cite{tan2006non} via Fourier series arguments. This alternative approach allows handling a class of non-convex cost functions with a unique global minimiser.}
{We show that the average system trajectories converge to a neighbourhood of the minimiser. The size of this neighbourhood is proportional to $\delta$, which is not required to be small.} 
As for the second contribution, it is shown that the addition of a high-pass filter makes $\gamma$ independent (uniform) of the value of the cost function. We prove that semi-global and practical convergence to the minimiser is achieved with the parameter $\gamma$ only depending on the Lipschitz constant of the cost function in the domain of interest and not on its value. This allows for tuning of the algorithm that preserves good convergence speed even in presence of a large domain of attractions. \ifthenelse{\boolean{CONF}}{}{Third and final, we show that if the cost function is globally Lipschitz and under certain regularity conditions on the average system, the investigated ES scheme makes the global minimiser  practically stable with a global domain of attraction.}

	The rest of this paper is organised as follows. Section \ref{sec:Prob} provides the formulation of the problem and reviews the basic ES scheme proposed by \cite{tan2006non}. {In this section, we show that the ES scheme applicability can be extended to non strictly convex cost functions via averaging analyses, which are not based on Taylor's arguments.} Section \ref{sec:Global} describes the ES scheme improved with the high-pass filter and states the \ifthenelse{\boolean{CONF}}{result about the uniformity of $\gamma$ with respect to the cost amplitude.}{two main results: the first, about the uniformity of $\gamma$ with respect to the cost amplitude; the second, on the global domain of attraction of the minimiser.} The performance of the investigated ES schemes is tested through the simulations detailed in Section \ref{sec:Application}. Finally, Section \ref{sec:Concl} closes this paper with some concluding comments. \ifthenelse{\boolean{CONF}}{For lack of space, all the proofs of Lemmas, Propositions, and Theorems claimed in this paper are reported in \cite{Mimmo2022ESarXiv}.}{}

\section{Problem Formulation }
\label{sec:Prob}

The ES problem consists of the optimisation of a{n unknown} cost function $h\,:\,\mathbb{R}\to \mathbb{R}$ {satisfying} the following  two assumptions. 
\begin{hyp}
The function $h$ is smooth and there exists a  $x^\star \in \mathbb{R}$ such that 
	\label{hyp:Existence}
	\[ h(x)-h(x^\star) > 0\quad \forall\, x \in \mathbb{R}\,:\, x \ne x^\star.
	\]
\end{hyp} 
\begin{hyp}	
\label{hyp:Unicity}
		There exist a locally Lipschitz {and \textit{strictly quasi-convex}} function $m\,:\,\mathbb{R}\to \mathbb{R}$, a class-${\cal K}_\infty$ function $\alpha(\cdot)$, and a $A\geq 0$ such that
\ifthenelse{\boolean{CONF}}{
for all $x_1, x_2 \in \mathbb{R}\,:\, (x_1-x^\star)(x_2-x^\star)\ge 0$ \[|m(x_2)-m(x_1)| \ge \alpha(|x_2-x_1|).\]}{
	\begin{enumerate}
	\item $m(x)-A \le h(x) \le m(x)+A$ for all $x \in \mathbb{R}$
	\item for all $x_1, x_2 \in \mathbb{R}\,:\, (x_1-x^\star)(x_2-x^\star)\ge 0$ \[|m(x_2)-m(x_1)| \ge \alpha(|x_2-x_1|)\]
	\end{enumerate}
}
\end{hyp}
	
	\begin{rem}
As for Assumption \ref{hyp:Unicity}, \ifthenelse{\boolean{CONF}}{it}{we observe that the function $h(\cdot)$ might have local minima whose ''depth'' is bounded by the number $A$. When $A=0$ (which  implies $m(\cdot) \equiv h(\cdot)$)  Assumption \ref{hyp:Unicity}}
asks for a (not necessarily strict) monotone behaviour of $h$, where the latter  could have isolated saddle points. {In other words, \ifthenelse{\boolean{CONF}}{$h(\cdot)$}{$m(\cdot)$}
	belongs to the class of the so-called \textit{strictly quasi-convex} functions \cite{karamardian1967strictly}.} Assumption \ref{hyp:Unicity} is weaker than the common assumption $(\partial h(x)/\partial x)x> 0$ for any $x \ne x^\star$ typically present in literature, (see, among the others, [\cite{tan2006non}, Assumptions 3 and 4])  ruling out the existence of local \ifthenelse{\boolean{CONF}}{}{minima or even} saddle points. 
	\end{rem}

\ifthenelse{\boolean{CONF}}{}{
\begin{figure}
	\centering
\begin{tikzpicture}
	\draw[->] (-4,1) -- (3,1) node[right] {$x$};
	\draw[->] (-3,-0.5) -- (-3,3.5) node[right] {$h(x),\,m(x)$};
	\draw[scale=1, domain=-4:0, smooth, variable=\x, dashed] plot ({\x}, {0.5*(0.5*sin(deg(2*\x))-\x)});
	\draw[scale=1, domain= 0:1, smooth, variable=\x, dashed] plot ({\x}, {\x*\x});
	\draw[scale=1, domain= 1:2, smooth, variable=\x, dashed] plot ({\x}, {1+2*(\x-1)-(\x-1)*(\x-1)});
	\draw[scale=1, domain= 2:3, smooth, variable=\x, dashed] plot ({\x}, {2+0*(\x-2)+(\x-2)*(\x-2)}) node[right, xshift = -2, yshift = -5] {$m(x)$};
	\draw[scale=1, domain=-4:0, smooth, variable=\x] plot ({\x}, {0.4*sin(deg(5*\x))+0.5*(0.5*sin(deg(2*\x))-\x)});
	\draw[scale=1, domain= 0:1, smooth, variable=\x] plot ({\x}, {0.4*sin(deg(5*\x))+\x*\x});
	\draw[scale=1, domain= 1:2, smooth, variable=\x] plot ({\x}, {0.4*sin(deg(5*\x))+1+2*(\x-1)-(\x-1)*(\x-1)});
	\draw[scale=1, domain= 2:3, smooth, variable=\x] plot ({\x}, {0.4*sin(deg(5*\x))+2+0*(\x-2)+(\x-2)*(\x-2)}) node[left, yshift = 5, xshift = 2] {$h(x)$};
	\draw[scale=1, domain=-4:0, smooth, variable=\x, dotted] plot ({\x}, {0.5+0.5*(0.5*sin(deg(2*\x))-\x)});
	\draw[scale=1, domain= 0:1, smooth, variable=\x, dotted] plot ({\x}, {0.5+\x*\x});
	\draw[scale=1, domain= 1:2, smooth, variable=\x, dotted] plot ({\x}, {0.5+1+2*(\x-1)-(\x-1)*(\x-1)});
	\draw[scale=1, domain= 2:2.5, smooth, variable=\x, dotted] plot ({\x}, {0.5+2+0*(\x-2)+(\x-2)*(\x-2)});
	\draw[scale=1, domain=-4:0, smooth, variable=\x, dotted] plot ({\x}, {-0.5+0.5*(0.5*sin(deg(2*\x))-\x)});
	\draw[scale=1, domain= 0:1, smooth, variable=\x, dotted] plot ({\x}, {-0.5+\x*\x});
	\draw[scale=1, domain= 1:2, smooth, variable=\x, dotted] plot ({\x}, {-0.5+1+2*(\x-1)-(\x-1)*(\x-1)});
	\draw[scale=1, domain= 2:3, smooth, variable=\x, dotted] plot ({\x}, {-0.5+2+0*(\x-2)+(\x-2)*(\x-2)});
	\draw[<->] ({(-pi/(2*7))}, {0.5*(0.5*sin(deg(2*(-pi/(2*7))))-(-pi/(2*7)))}) -- ({(-pi/(2*7))}, {0.5+0.5*(0.5*sin(deg(2*(-pi/(2*7))))-(-pi/(2*7)))});
	\node (A) at (-0.4, 0.25) {$A$};
	\node at ({(-pi/(2*5))},{0.4*sin(deg(5*(-pi/(2*5))))+0.5*(0.5*sin(deg(2*(-pi/(2*5))))-(-pi/(2*5)))}) {$\bullet$};
	\node at ({(-pi/(2*5))},{-0.3+0.4*sin(deg(5*(-pi/(2*5))))+0.5*(0.5*sin(deg(2*(-pi/(2*5))))-(-pi/(2*5)))}) {$h(x^\star)$};
	\node at ({(-pi/(2*5))},1) {$\bullet$};
	\node at ({(-pi/(2*5))},{1+0.3}) {$x^\star$};
\end{tikzpicture}
\caption{Graphical representation of $h(\cdot)$, $m(\cdot)$, $x^\star$, and $A$ declared in Assumption  
	 \ref{hyp:Unicity}.}
\label{fig:h}	
\end{figure}
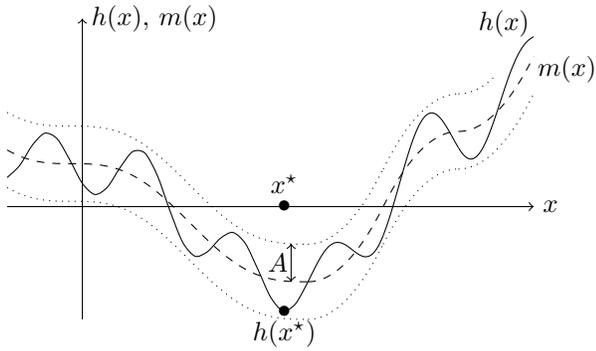
}

{T}he problem of {\em semi-global extremum seeking} can be formulated in the following way. For {any  $\epsilon >0$ and $r_0>0$,} design a system of the form 
\[
 \dot x = \varphi_{\epsilon, r_0}(x, h(x),t)  \quad x(0)=x_0, 
\]
so that for all $x_0$ satisfying $|x_0 - x^\star| \leq r_0$ the resulting trajectories $x(t)$ {are bounded and satisfy} $\lim_{t \to \infty}\sup |x(t) - x^\star| \leq \epsilon$.

\ifthenelse{\boolean{CONF}}{}{If the {convergence} property {is sought  $\forall x_0 \in \mathbb{R}$} with  $\varphi$ {independent of $\epsilon$ and $x_0$, then} the problem is referred to as {\em global extremum seeking}. }

Among the different ES schemes proposed in the literature to solve the previous problem, a common {one} is given by   (see  \cite{tan2006non})
\begin{equation}
	\label{eq:BasicES}
	\dot{x}=  -\gamma \, y_\delta(x,t) \, u(t)\qquad x(0) = x_0
\end{equation}
in which $y_\delta(x,t):=h(x + \delta u(t))$, $u(t):= \sin (2\pi t)$ is the dither signal and $\gamma, \delta > 0$ are tunable parameters\footnote{
The general case of a dither takes the form  $\tau\mapsto\sin(\omega \tau)$ with $\omega >0$, $\tau \in \mathbb{R}$ (as considered in \cite{tan2006non}) can be always obtained by  rescaling the time as $t = \tau \,2\pi/\omega$.}.
 The block diagram of this algorithm is represented in Figure \ref{fig:GloablES}. In the next part, we briefly comment on the main properties, as available in {the} literature, of this algorithm and strengthen them.  

   Since the right-hand side of \eqref{eq:BasicES} is 1-periodic, the average system linked to (\ref{eq:BasicES}) is given by (\cite{khalil2002nonlinear}, \S 10.4) 
\begin{equation}
	\label{eq:average}
	\dot{{x}}_a =  -\gamma \int_0^1 y_\delta({x}_a, t) u(t) dt \,.
\end{equation}

In the remaining part of the section, we present a different route for the analysis of \eqref{eq:average} based on a Fourier series expansion rather than on the Taylor one used in the available literature. 
This analysis allows one to claim stronger results on \eqref{eq:BasicES}.

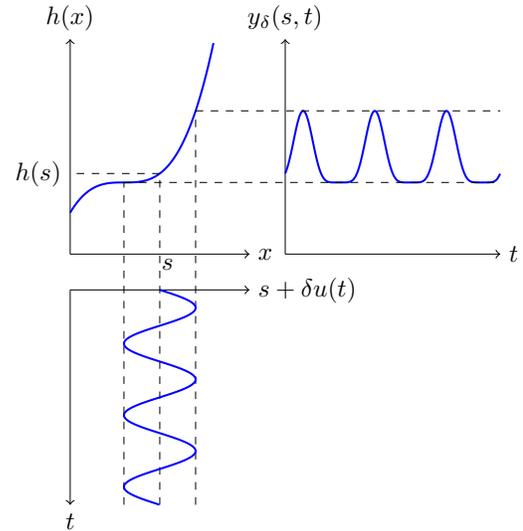
\begin{figure}
	\centering
	\resizebox{0.80\columnwidth}{!}{$
	\begin{tikzpicture}
		\draw[->](0,0) -- (2.5,0) node[right]{$x$};
		\draw[->](0,0) -- (0,3) node[above]{$h(x)$};
		\draw[samples = 100, scale=1, domain=0:2,smooth,variable=\s, blue, thick] plot({\s},{1+(\s-0.75)^3});
		\draw[dashed] ({1.25},{1+(1.25-0.75)^3}) -- (0,{1+(1.25-0.75)^3}) node[left]{$h(s)$};
		\def\xs{3}
		\draw[->](\xs,0) -- (\xs+3,0) node[right]{$t$};
		\draw[->](\xs,0) -- (\xs,3) node[above]{$y_\delta(s,t)$};
		\draw[samples = 100, scale=1, domain=0:3,smooth,variable=\s, blue, thick] plot({\xs + \s},{1+(1.25+0.5*sin(\s*360)-0.75)^3});
		\draw[dashed] ({1.25-0.5},{1+(1.25-0.5-0.75)^3}) -- ({\xs + 3},{1+(1.25-0.5-0.75)^3});
		\draw[dashed] ({1.25+0.5},{1+(1.25+0.5-0.75)^3}) -- ({\xs + 3},{1+(1.25+0.5-0.75)^3});
		\def\ys{-0.5}
		\draw[->](0,\ys) -- (2.5,\ys) node[right]{$s + \delta u(t)$};
		\draw[->](0,\ys) -- (0,-3+\ys) node[below]{$t$};
		\draw[dashed, -] (1.25,-3+\ys) -- (1.25,{1+(1.25-0.75)^3}) node[right, pos = 0.725, xshift = -0.1cm]{$s$};
		\draw[dashed, -] (1.25-0.5,-3+\ys) -- (1.25-0.5,{1+(1.25-0.5-0.75)^3});
		\draw[dashed, -] (1.25+0.5,-3+\ys) -- (1.25+0.5,{1+(1.25+0.5-0.75)^3});
		\draw[samples = 100, scale=1, domain=0:3,smooth,variable=\s, blue, thick] plot({1.25+0.5*sin(\s*360)},{-\s+\ys});
	\end{tikzpicture}
$}
	\caption{Graphical representation of the periodic behaviour of $y_\delta(s,t)$. Conceiving $s$ as a parameter, we see that the oscillation induced by $u(t)$ is elaborated by the non-linear map $h(\cdot)$ and results in a periodic function of time (top-right). This latter can be described through its Fourier coefficients $a_k(s)$ and $b_k(s)$.}
	\label{fig:Fourier}
\end{figure}

{The smoothness of $h(\cdot)$ in Assumption 1 is essentially asked to guarantee the existence  of the Fourier series of the function $y_\delta$\footnote{Milder regularity properties guaranteeing the existence of the series could be assumed.}\ifthenelse{\boolean{CONF}}{}{(plus other regularity properties used in the proof of the forthcoming Lemma \ref{lemma:Reduced1}  and Proposition \ref{prop:BasicES})}. Then, s}ince $y_\delta(x_a, t)$ and its time derivatives are continuous and periodic, $y_\delta(x_a, t)$ can be expressed in terms of {its} Fourier series  as 
	\begin{equation}
		\label{eq:FourierSum}
		\begin{aligned}
			&y_\delta(x_a, t)= \dfrac{a_{0,\delta}(x_a)}{2} + \\
			&\sum_{k=1}^{\infty}  a_{k,\delta}(x_a) \cos\left(k 2\pi t\right) + b_{k,\delta}(x_a) \sin\left(k 2\pi  t\right)
		\end{aligned}
	\end{equation}
	where
	\begin{equation}
		\label{eq:FourierCoeff}
		\begin{aligned}
			a_{k, \delta}(x_a) &:= 2 \int_{0}^{1} y_\delta(x_a,t)\cos(k 2\pi  t)\,dt\\
			b_{k,\delta}(x_a) &:= 2 \int_{0}^{1} y_\delta(x_a,t)\sin(k 2\pi  t)\,dt\,.
		\end{aligned}
	\end{equation}
     Embedding (\ref{eq:FourierSum}) in (\ref{eq:average}) it is immediately seen that the average system linked to (\ref{eq:BasicES}) reads as 
     \begin{equation} \label{eq:average-Fourier}
      \dot x_a = -{\gamma \over 2} b_{1,\delta}(x_a)\,.
     \end{equation}
  For this system the following result holds. In the result, we refer to the class-${\cal K}_\infty$ function $\underline{\delta}^\star(\cdot)$  defined as
  \begin{equation}
  	\label{eq:deltastar}
  \underline{\delta}^\star(s) := 2\int_{0}^{1/2} \alpha(s \sin(2\pi t))\,dt  
  \end{equation}

\begin{lemma} \label{lemma:Reduced1}
Let $h(\cdot)$ be such that Assumptions \ref{hyp:Existence}-\ref{hyp:Unicity} are satisfied. Then:
\begin{itemize}
\item[a)] for all positive $\delta$ \ifthenelse{\boolean{CONF}}{}{and $A$ such that $\underline{\delta}^\star(\delta) \geq A+ \bar b$ for some $\bar b>0$,} there exists a compact set ${{\cal A}_\delta} \subseteq [ x^\star-\delta, x^\star+\delta]$ that is globally asymptotically and locally exponentially stable for (\ref{eq:average-Fourier}).
\item[b)] There exists a $\bar \delta^\star>0$ such that, for all positive $\delta$ \ifthenelse{\boolean{CONF}}{{such that  $\delta \leq \bar \delta^\star$}}{and $A$ such that  $\delta \leq \bar \delta^\star$ and $\underline{\delta}^\star(\delta) \geq A + \bar b$ for some $\bar b>0$}, there exists an equilibrium point $x_{a \delta}^\star \in \mbox{\em int } {{\cal A}_\delta}$ that is  
 locally exponentially stable for system (\ref{eq:average-Fourier}). If, in addition,  the function $\bar{h}(s) := h(x^\star+s)-h(x^\star)$ is even, then $x_{a \delta}^\star = x^\star$.
\end{itemize}
\end{lemma}
\ifthenelse{\boolean{CONF}}{}{The Lemma is proved in Appendix \ref{sec:ProofLemmaReduced1}.}
	
Item a) of the previous lemma {states} that the trajectories of the average system reach a compact set ${\cal A}_\delta$ that is contained in a $\delta$ neighbourhood of $x^\star$ for all possible $\delta$\ifthenelse{\boolean{CONF}}{}{ that are sufficiently  large with respect to the ''depth'' $A$ of the local minima in a global way}. This, in particular, implies that there exists a class $\cal  K L$ function $\beta(\cdot, \cdot): \mathbb{R} \times \mathbb{R}_+ 	\to \mathbb{R}_+$ such that 
\[
|x_a(t)|_{{\cal A}_\delta} \leq \beta(|x_a(0)|_{{\cal A}_\delta}, t)\,
\]
where $|\cdot|_{\mathcal{A}_\delta}$ denotes the distance to the set $\mathcal{A}_\delta$. 
\ifthenelse{\boolean{CONF}}{}{In case $A=0$, namely only saddle points {can be} present, then $\delta$ is only required to be positive. Otherwise, $\delta$ must be taken sufficiently large to not get stuck in local minima.  }      
Item $b)$ claims that the set ${\cal A}_\delta$ collapses to an equilibrium point if $\delta$ is also taken sufficiently small\ifthenelse{\boolean{CONF}}{}{, besides being, as before,  sufficiently large according to $A$. which could require Assumption 2 to hold with a sufficiently small $A$}. Moreover, the last point of item b) shows that $x^\star$ represents the equilibrium point only for  cost functions that are  locally symmetric around the optimum. 

Standard averaging results can be then used to show that {the} same property is preserved also for the trajectories of the original system (\ref{eq:BasicES}) for sufficiently small $\gamma$ but in a semi-global and practical way. This is detailed in the next Proposition \ref{prop:BasicES} where we refer to the class ${\cal K}_\infty$ function $\chi(s)$ defined as  
\[
\chi(s):= \beta^{-1}(s, 0). 
\]
In the following analysis we  denote by $L_r>0$ and $M_r>0$, {respectively} the local Lipschitz constant and the upper bound  of the function $h(\cdot)$ on a closed interval of length  $r$. In particular, regularity of $h$ implies that   
\begin{itemize}
	\item for all $r>0$ there exists  $L_r > 0$ such that for all $x_1,x_2 \in [x^\star-r,\,x^\star + r]$
	\begin{equation}
		\label{eq:Lr}
		|h(x_1)-h(x_2)|\le L_r |x_1-x_2|\,.
	\end{equation}
	\item for all $r>0$, there exists $M_r > 0$ such that for all $x \in [x^\star-r,\,x^\star + r]$ 
	\begin{equation}
		\label{eq:Mr}
	|h(x)| \leq M_r \,.  
	\end{equation}
\end{itemize}

\begin{prop}
	\label{prop:BasicES}
Let $h(\cdot)$ be such that Assumptions \ref{hyp:Existence}-\ref{hyp:Unicity} hold and let $r,\delta, d$ be arbitrary positive numbers such that $r - d- 2\delta >0$\ifthenelse{\boolean{CONF}}{}{ and $\underline{\delta}^\star(\delta)\geq A + \bar b$ for some $\bar b>0$}. Let $r_0 :=\chi(r-d-2\delta)$. There exist   $\bar t(r_0,d)$  and  $\gamma^\star(M_r,L_r,\delta, d) > 0$ such that for any $\gamma \in (0,\,\gamma^\star) $, any $x_0 \in \mathbb{R}\,:\,|x_0-x^\star| \le r_0$,  the trajectories of \eqref{eq:BasicES} are  bounded and
	\[
	 |x(t)|_{{\cal A}_\delta} \le d \qquad \forall \, t \geq {\bar t \over \gamma}\,.
	\]
 \end{prop} 

Appendix \ref{app:Detailed_BasicESNotGlobal_R} details the proof of this Proposition.
	
 An immediate consequence of the previous result is the next corollary showing that under Assumptions \ref{hyp:Existence}-\ref{hyp:Unicity} system (\ref{eq:BasicES}) solves the problem of semi-global extremum seeking formulated before\ifthenelse{\boolean{CONF}}{}{ provided that Assumption \ref{hyp:Unicity} is fulfilled with a sufficiently small $A$}.
\begin{corollary}
	\label{prop:CorollaryBasicES}
Let $h(\cdot)$ be such that Assumptions \ref{hyp:Existence}-\ref{hyp:Unicity} are fulfilled and let $r_0$ and $\epsilon$ be arbitrary positive numbers. Then, there exist $\bar t(r_0, \epsilon)>0$,  $\bar \delta^\star(\epsilon)> 0$ and  $\gamma^\star(r_0,\bar{\delta}^\star, \epsilon) > 0$ such that for any  $\delta \in (0, \bar \delta^\star)$\ifthenelse{\boolean{CONF}}{}{ and $\underline{\delta}^\star(\delta) \geq A + \bar b$ for some $\bar b>0$}, any $\gamma \in (0,\,\gamma^\star) $ and any $x_0 \in \mathbb{R}\,:\,|x_0-x^\star| \le r_0$,  the trajectories of \eqref{eq:BasicES} are  bounded and
	\[
	 |x(t) - x^\star|  \le \epsilon \qquad \forall \, t \geq {\bar t \over \gamma}\,.
	\]
 \end{corollary} 
 
By going through the proof of Proposition \ref{prop:BasicES},  it is immediately seen that $\gamma^\star$ is  inversely proportional to $M_r$. As a consequence, the higher the cost function is within the set where $x(t)$ ranges,  the lower the value of $\gamma$ and, in turn, the slower the convergence rate of $x$ to the neighbourhood of the optimum. 
 Section \ref{sec:Global} presents an improvement of \eqref{eq:BasicES} overtaking this limitation\ifthenelse{\boolean{CONF}}{}{ and, in turn, paving the way for a global result}. 

\subsection{{Comments on Taylor expansion-based averaging analyses}}

The {analysis} of \eqref{eq:average} is typically approached, see \cite{tan2006non}, by using a Taylor expansion of $h(\cdot)$ {to obtain} a system of the form  
\begin{equation}
	\label{eq:av_complete}
	\dot{{x}}_a = \, -\gamma 
	c_1 \delta  \left.\dfrac{\partial h}{\partial x}\right|_{{x}_a} -\gamma 
	\sum_{k=2}^{\infty} c_k \,{\delta^{2k-1}}\,
	\left.\dfrac{\partial^{2k-1}h}{\partial x^{2k-1}} \right|_{{x}_a} 
\end{equation}
where $c_k$ are suitably defined  positive coefficients. A key role in the study of this system is  played by the first-order approximation  
\begin{equation}
	\label{eq:First_Order}
	\dot{{x}}_a = \, -\gamma 
	c_1 \delta  \left.\dfrac{\partial h}{\partial x}\right|_{{x}_a}\,. 
\end{equation}
In fact, if {our} Assumption \ref{hyp:Unicity} is strengthen  by asking that{, for any $x \ne x^\star$, it holds} $(\partial h(x)/\partial x)(x-x^\star)> 0$ (respectively $>\alpha(|x-x^\star|)$ with $\alpha(\cdot)$ a class-${\cal K}$ function, see  \cite{tan2006non}, Assumptions 3 and 4){, then the Lyapunov arguments of [\cite{tan2006non}, eq. (45)]} demonstrate $x^\star$ to be a stable (respectively globally asymptotically stable) equilibrium point for (\ref{eq:First_Order}). These stability properties are transferred to \eqref{eq:average} for sufficiently small $\delta$. Then, averaging techniques \cite{khalil2002nonlinear} can be used to prove that{, for sufficiently small $\gamma$,} the trajectories of \eqref{eq:BasicES} {and \eqref{eq:average}} 
 remain arbitrarily close. 
  Namely, the semi-global extremum seeking problem is solved.   


The fact that the asymptotic properties of the average system \eqref{eq:average} are just ensured by the first-order term \eqref{eq:First_Order}, and thus by the gradient of $h$, implies that isolated local \ifthenelse{\boolean{CONF}}{}{minima, or even }saddle points, of $h$, {cannot be handled by that proof technique}. This justifies why [\cite{tan2006non}, Assumption 3], which is stronger than {our} Assumption \ref{hyp:Unicity}, is needed.  Furthermore, we observe that the previous analysis requires that the dither amplitude is kept sufficiently small {for} the higher-order terms of the average dynamics to be negligible.  

\section{{The High-Pass Filter (HPF)}-ES Algorithm}
\label{sec:Global}
\begin{figure}[t]
	\centering
	\begin{tikzpicture}
		\node[draw](map){$h(s)$};
		\node[draw, right of = map, circle, xshift = 1.5cm](sumy){$+$};
		\node[dashed, draw, above of = map, xshift = 1.5cm](Mean){Eq. \eqref{eq:bary}};
		\draw[dashed, ->] (map) -| (Mean); 
		\draw[dashed, ->] (Mean) -| (sumy) node[pos = 0.5, right]{$-\bar{y}$};
		\node[draw, below of = map, xshift = 0cm, yshift = -0.0cm](ES){$\dot{{x}} = -\gamma v$};
		\node[draw, right of = ES, circle, xshift = 1.5cm](times){$\times$};
		\node[draw, left of = ES, circle, xshift = -1.5cm](sum){$+$};
		\draw[->](map) -- (sumy);
		\draw[->](times) --  node[pos= 0.5, above]{$v$} (ES);  
		\draw[->](ES) -- node[pos= 0.5, above]{${x}$} (sum);
		\draw[->](sum) |- node[pos= 0.5, above]{$s$} (map);
		\node[below of = ES](dither){$u$};
		\draw[->] (sumy) -- (times);
		\node[triangle, left of = dither, shape border rotate=180, xshift = -0.25cm](delta){$\delta$};
		\draw[->](dither) -|  (times);
		\draw[->](dither) -- (delta);
		\draw[->](delta) -| (sum);
	\end{tikzpicture}
	\caption{The HPF-ES scheme consists of an extension of the classic ES algorithm with the addition of the dashed block.}
	\label{fig:GloablES}
\end{figure}
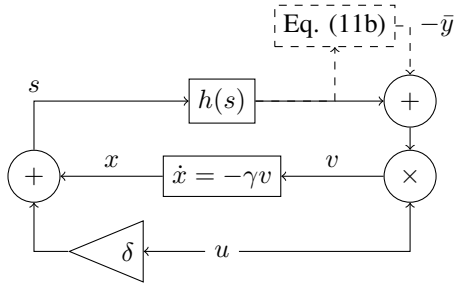

In \cite{tan2006non} the authors proposed the {next} modification of \eqref{eq:BasicES} 
\begin{subequations} 
	\label{eq:GlobalESGeneric}
	\begin{align}
		\label{eq:GlobalESGeneric_dotx}
		\dot x =&\,- \gamma  \left ( y_\delta({x},t)-\bar{y} \right ) \,u(t) &&   {x}(0)={x}_0\\
		\dot{\bar{y}}  = &\, \gamma\left (y_\delta({x},t)-\bar{y}\right ) &&   \bar{y}(0)=\bar{y}_0 \label{eq:bary}
	\end{align}
\end{subequations}
where $u(t)$ is the dither signal defined before, and $(x,\bar{y}) \in \mathbb{R} \times \mathbb{R}$. {A block representation of \eqref{eq:GlobalESGeneric} is depicted in Figure \ref{fig:GloablES}.}

The intuition behind the previous scheme is to interpret $y_\delta({x},t)-\bar y$ as the output of a high pass filter of  $y_\delta$. Moreover, the difference $y_\delta({x},t)-\bar{y}$ represents a proxy  of the local \textit{mean} variation of $h(x)$, directly proportional to the \textit{mean} local Lipschitz constant. In the following, we show how this feature guarantees that the upper bound for the value of $\gamma$ is not dependent on $M_r$ {(eq. \eqref{eq:Mr})} but rather only on $L_r$ {(\eqref{eq:Lr})}. As for $\gamma$, similarly to \eqref{eq:BasicES}, it must be small to let $x$ and $\bar{y}$ be sufficiently slow to preserve the correlation between the oscillations of $y_\delta({x},t)-\bar{y}$ and those of $u(t)$.
The average system of \eqref{eq:GlobalESGeneric} is defined as 
\begin{subequations}
	\label{eq:TildeAV}
	\begin{align}
		\label{eq:DotTildeXAv}
		\dot{{x}}_a &=-\gamma  \int_{0}^{1} \left ( y_\delta({x_a},\tau)-\bar{y}_a \right ) \, u(\tau)\,d\tau\\
		\label{eq:baryAV}
		\dot{\bar{y}}_a & = -\gamma\,\bar{y}_a +  \gamma\int_{0}^{1}y_\delta ({x}_a,\tau)\,d\tau.
	\end{align}
\end{subequations}

By expanding  $y_\delta(x_a, t)$ in terms of the Fourier series as in the previous section, it turns out that   
\begin{subequations} 
	\label{rho}
	\begin{align}
		\displaystyle \int_{0}^{1}y_\delta(x_a,\tau)\,d\tau &= \dfrac{a_{0,\delta}(x_a)}{2}\\
		\displaystyle \int_{0}^{1}\left ( y_\delta({x_a},\tau)-\bar{y}_a \right )  u(\tau)\,d\tau  &= \dfrac{b_{1,\delta}(x_a)}{2}
	\end{align}
	where in the latter we exploited $\int_{0}^{1}\bar{y}_a u(\tau) d\tau = 0$ and \eqref{eq:FourierCoeff}. 
\end{subequations}

Hence, the average system reads as 
\begin{subequations}
	\label{eq:TildeAV_2}
		\begin{align}
			\label{eq:DotTildeXAv_2} 
		\dot{{x}}_a &=-\gamma\dfrac{ b_{1,\delta}(x_a)}{2 }\\
		\label{eq:DotBarYAV}		
		\dot{\bar{y}}_a & = -\gamma\, \bar{y}_a + \gamma\, \dfrac{a_{0, \delta}({x}_a)}{2}
	\end{align}
\end{subequations}
which is a cascade where the first subsystem coincides with \eqref{eq:average-Fourier} and the second subsystem is linear and asymptotically stable. From this, the next result follows from  Lemma  \ref{lemma:Reduced1}.

\begin{lemma} \label{lemma:Reduced2}
	Let $h(\cdot)$ be such that Assumptions \ref{hyp:Existence}-\ref{hyp:Unicity} are satisfied. Then:
	\begin{itemize}
		\item[a)] for {any} positive $\delta> 0$\ifthenelse{\boolean{CONF}}{}{ and $A> 0$ such that $\underline{\delta}^\star(\delta) \geq A+ \bar b$ for some $\bar b>0$,} there exist a compact set ${\cal A}_\delta \subseteq [ x^\star-\delta, x^\star+\delta]$ and a continuous function $\tau: \mathbb{R} 	\to \mathbb{R}$ such that the set
		\[
		 \mbox{\rm graph} \left. \tau \right |_{{\cal A}_\delta} = \{ (x_a, y_a) \in {\cal A}_ \delta \times \mathbb{R} \; : \; y_a = \tau(x_a)\}
		\]
 is globally asymptotically and locally exponentially stable for \eqref{eq:TildeAV_2}.
		\item[b)] There exists $\bar \delta^\star>0$ such that, for {any} $\delta \in (0, \bar \delta^\star )$\ifthenelse{\boolean{CONF}}{}{ and any $A> 0$ such that $\underline{\delta}^\star(\delta) \geq A+ \bar b$ for some $\bar b>0$}, there exists an equilibrium point $(x_{a \delta}^\star, \bar y_{a \delta}^\star ) \in  \mbox{\rm graph} \left. \tau \right |_{{\cal A}_\delta}$ that is  globally asymptotically and locally exponentially stable for system (\ref{eq:average-Fourier}). If, in addition,  the function $\bar{h}(s) := h(x^\star+s)-h(x^\star)$ is even, then $x_{a \delta}^\star = x^\star$.
	\end{itemize}
\end{lemma}
\ifthenelse{\boolean{CONF}}{}{
A sketch of the proof of this Lemma, with special regard to the definition of the function $\tau(\cdot)$, is presented in Appendix \ref{sec:ProofLemmaReduced2}.}
 
From this, the following result  mimics the one of Proposition \ref{prop:BasicES} with the remarkable difference that the upper bound $\gamma^\star$ of $\gamma$ is uniform with respect to $M_r$.

\begin{theorem}
	\label{prop:HPFES}
	Let $h(\cdot)$ be such that Assumptions \ref{hyp:Existence}-\ref{hyp:Unicity} are fulfilled and let $r,\delta, d$ be arbitrary positive numbers such that $r - d- 2\delta >0$\ifthenelse{\boolean{CONF}}{{}}{ and $\underline{\delta}^\star(\delta)\geq A + \bar b$ for some $\bar b>0$}. Let $r_0 :=\chi(r-d-2\delta)$. There exist   $\bar t(r_0,d)$ and  $\gamma^\star(L_r,\delta, d) > 0$  such that for any $\gamma \in (0,\,\gamma^\star) $ and any $(x_{0},\,\bar{y}_{0})$ fulfilling
	$|x_{0}-x^\star|\le r_0$ and $|\bar{y}_{0} - a_{0,\delta}(x_{0})/2| \le \gamma^\star$,
then	the trajectories of \eqref{eq:GlobalESGeneric} are  bounded and
	\[
	\|(x(t), \bar y(t))\|_{\mbox{\rm graph} \left. \tau \right |_{{\cal A}_\delta} } \le d \qquad \forall \, t \geq {\bar t \over \gamma}\,.
	\]
\end{theorem} 
\ifthenelse{\boolean{CONF}}{}{This Theorem is proved in Appendix \ref{sec:DetailedProofHPFE}. The fact that $\gamma^\star$ does not depend anymore on $M_r$ but only on $L_r$ suggests that in presence of a globally Lipschitz cost function the HPF extremum seeking scheme in (\ref{eq:GlobalESGeneric}) is global. This intuition is confirmed in the following theorem.

\begin{theorem}
	\label{cor:Global}
	Let $h(\cdot)$ be such that Assumptions \ref{hyp:Existence}-\ref{hyp:Unicity} are fulfilled. Moreover, assume $h(\cdot)$ be globally Lipschitz. Then, for any $\epsilon>0$ there exist $\gamma^\star(\epsilon) > 0$ and $\overline{\delta}(\epsilon)$ such that for any $\gamma \in (0, \gamma^\star)$, for any $\delta$ and $A>0$ fulfilling $\delta \le \overline{\delta}(\epsilon)$ and $\underline{\delta}^\star(\delta) \ge A+\overline{b}$ for some $\bar{b} > 0$, and for any $x_0 \in \mathbb{R}$ and $\bar{y}_0 \in \mathbb{R}\,:\,|\bar{y}_0-a_{0,\delta}(x_0)/2| \le \gamma$, then the trajectories of \eqref{eq:GlobalESGeneric} are bounded and
	\[
	\limsup_{t \to \infty} |x(t)-x^\star| \le \epsilon.
	\]
\end{theorem}
This theorem is proved in Appendix \ref{sec:ProofCorollary}.}

\section{Numerical Results}
\label{sec:Application}
This section presents numerical results obtained adopting 
 the following cost function $h(\cdot)\,:\,\mathbb{R}\to \mathbb{R}$ 
\ifthenelse{\boolean{CONF}}{
\begin{equation}
	\label{eq:h2ndCase}
		h(x) = h_0+\left\{\begin{array}{cc}
			(x-\pi)^2-1 & x < \pi\\
			\cos(x-\pi)-2 & x \in [\pi,\,2\pi)\\
			(x-2\pi)^2-3 & x \ge 2\pi
		\end{array}\right.
\end{equation}
where
}
{\begin{equation}
	\label{eq:h2ndCase}
	\begin{aligned}
	h(x) =&\, h_0+A\sin(10x)\\
	&\,+\left\{\begin{array}{cc}
		(x-\pi)^2-1 & x < \pi\\
		\cos(x-\pi)-2 & x \in [\pi,\,2\pi)\\
		(x-2\pi)^2-3 & x \ge 2\pi
	\end{array}\right.
\end{aligned}
\end{equation}
where } $h_0 \in \mathbb{R}$. This function, which verifies Assumptions \ref{hyp:Existence} and \ref{hyp:Unicity}, is depicted in Figure \ref{fig:2ndMap_a} for $A= 0$ and $h_0 = 10$, and in Figure \ifthenelse{\boolean{CONF}}{\ref{fig:2ndMap_a}}{\ref{fig:LocalMinima_a}} for $A = 1/4$ and $h_0 = 10$, {on $[-\pi,\,3\pi]$}. 

{The tests are grouped into two categories, the first of which highlights the performance of \eqref{eq:BasicES} whereas the second deals with the behaviour of \eqref{eq:GlobalESGeneric}.

As for the performance of \eqref{eq:BasicES}, the results are presented in agreement with an incremental complexity policy. First, for the case of a local strongly convex cost, we show that the Fourier-based averaging \eqref{eq:average-Fourier} predicts the asymptotic equilibrium of \eqref{eq:BasicES} more accurately than the first-order Taylor expansion-based averaging \eqref{eq:First_Order}.  Then, assuming a quasi-strongly convex cost, we test that the trajectory of \eqref{eq:First_Order} does not track the trajectory of \eqref{eq:BasicES} as good as done by \eqref{eq:average-Fourier}. Moreover, we introduce a non-convex cost function and we show that increasing $\delta$ lets \eqref{eq:BasicES} to pass over local minima.}

\ifthenelse{\boolean{CONF}}{
	\begin{figure}
		\centering
		\includegraphics[width=0.55\columnwidth]{CostFunction_Log}
		\caption{Cost function \eqref{eq:h2ndCase} for $h_0 = 0$. It presents multiple points with null first derivative and it is also non-symmetric around the minimiser. The saddle point is located at $x = \pi$ and the minimiser is at $x^\star = 2\pi$.}
		\label{fig:2ndMap_a}
	\end{figure}	
	\begin{figure}
		\centering
		\includegraphics[width=0.55\columnwidth]{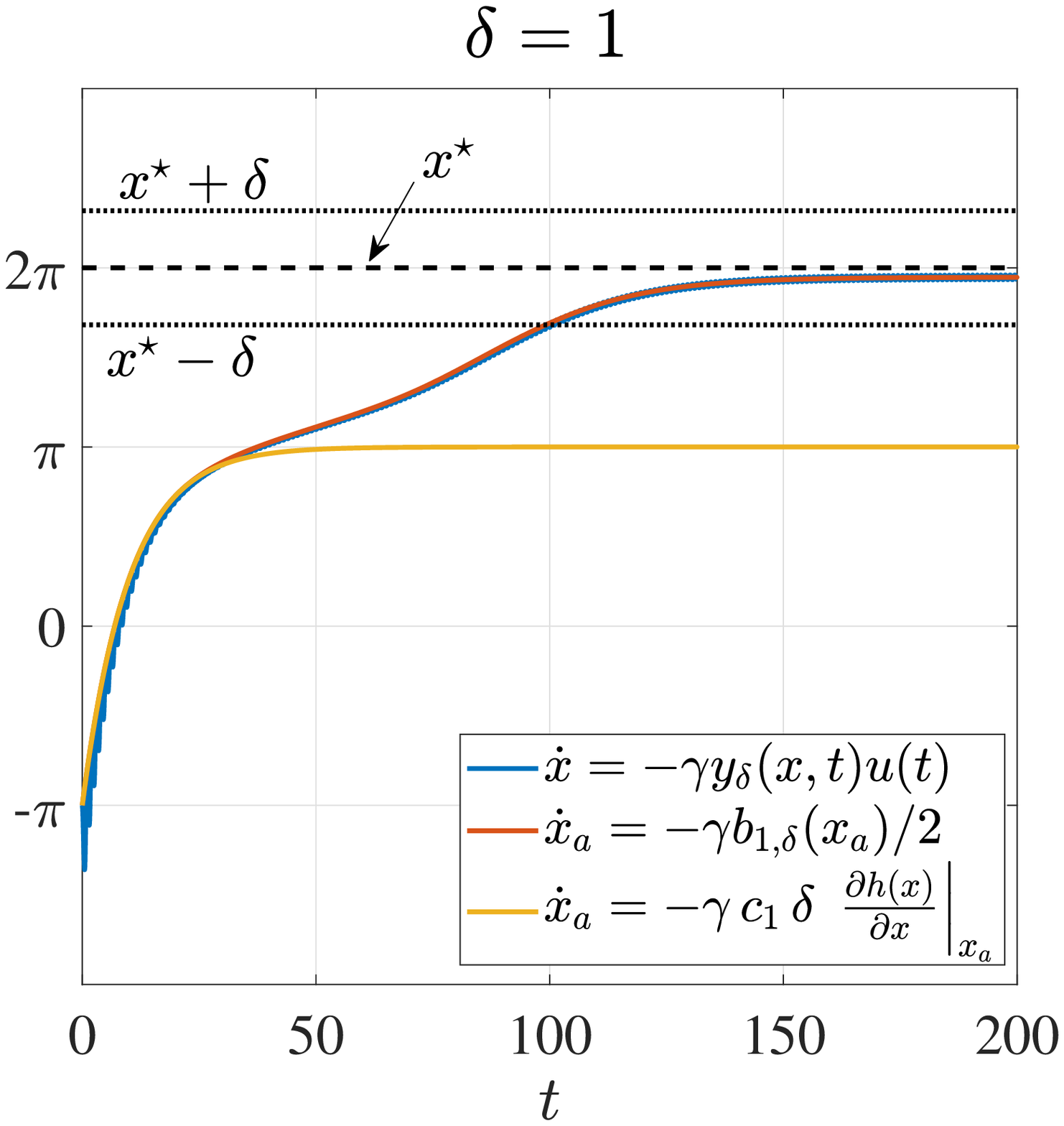}
		\caption{The presence of saddle points stacks the Taylor-based averaging (yellow) of the classic ES (blue). Vice versa, the Fourier-based averaging (red) better represents the actual behaviour of the classic ES (blue). These results are obtained for $h_0=0$ and $\gamma = 0.1$.}
		\label{fig:2ndMap_b}
	\end{figure}
	\begin{figure}
		\centering
		\includegraphics[width=0.55\columnwidth]{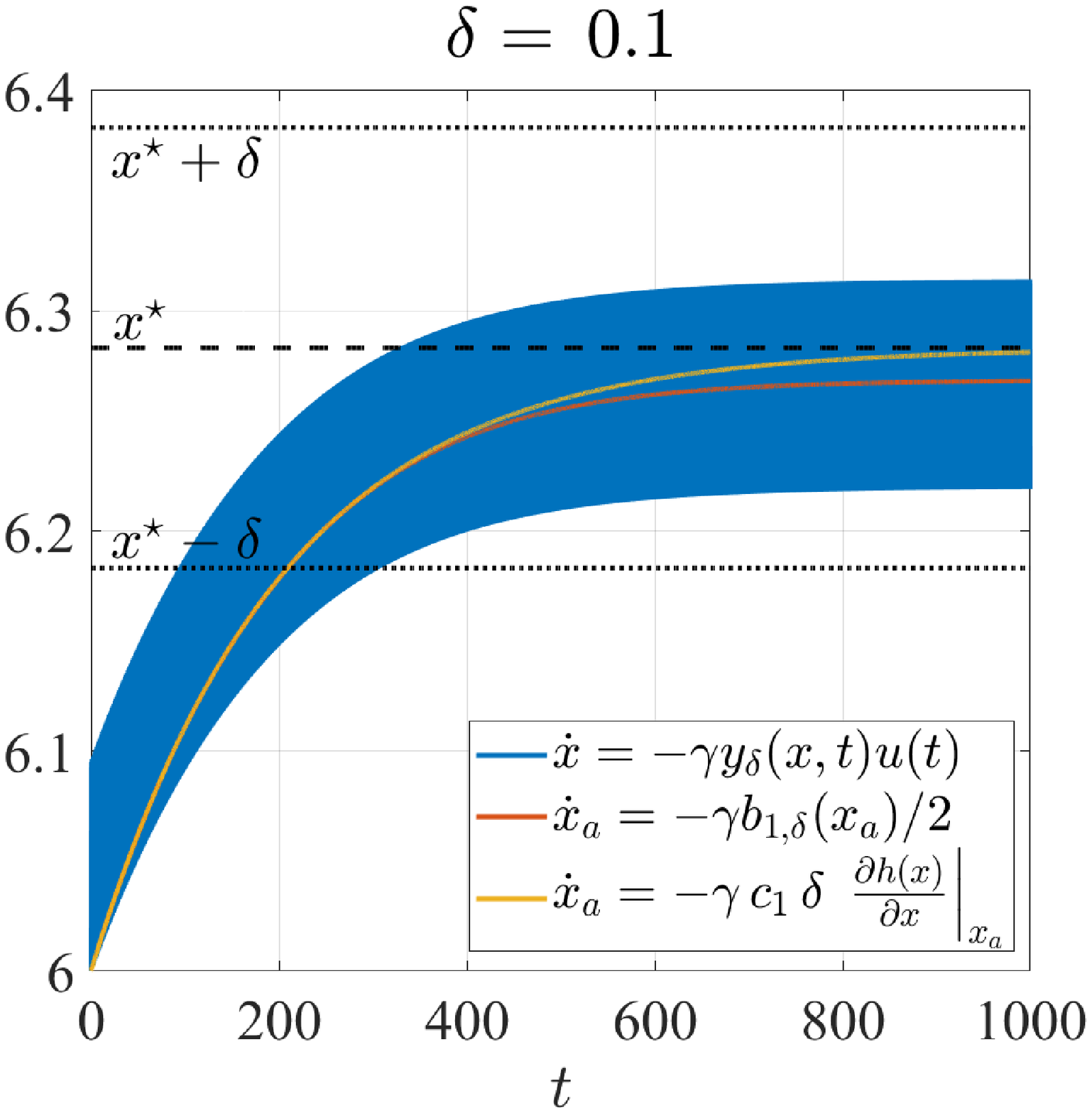}
		\caption{The average based on the Taylor expansion (yellow) of the classic ES wrongly assess $x^\star$ as an equilibrium point. Indeed, due to the asymmetry of $h(\cdot)$ around $x^\star$, the classic ES (blue) converges to 
			an equilibrium point different from $x^\star$. Moreover, the average based on the Fourier series (red) tracks the classic ES more accurately than the average based on the Taylor expansion. These results are obtained for $h_0=0$ and $\gamma = 0.1$.}
		\label{fig:2ndMap_c}
	\end{figure}
	\begin{figure}
		\centering
		\includegraphics[width=0.55\columnwidth]{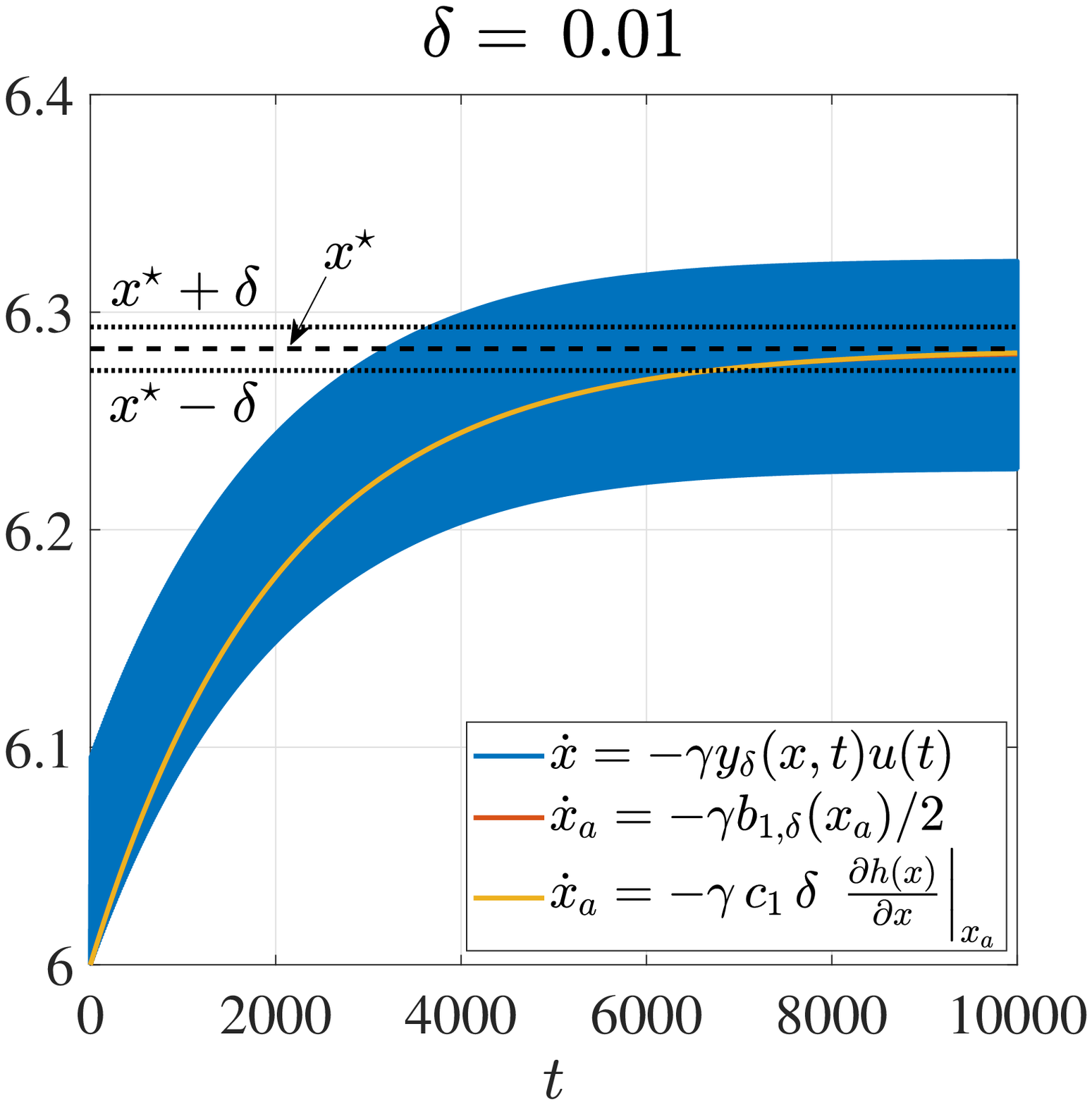}
		\caption{The average of the classic ES, based on the Fourier series (red), converges to an equilibrium point within the set $[x^\star-\delta,\, x^\star+\delta]$ accordingly to what is foreseen in Lemma \ref{lemma:Reduced1}. Thus, the smaller is $\delta$ the closer the equilibrium point is to $x^\star$. These results are obtained for $h_0=0$ and $\gamma = 0.1$.}
		\label{fig:2ndMap_d}
	\end{figure}
}{
	\begin{figure}[h!]
		\centering
		\begin{subfigure}{.48\columnwidth}
			\includegraphics[width=\textwidth]{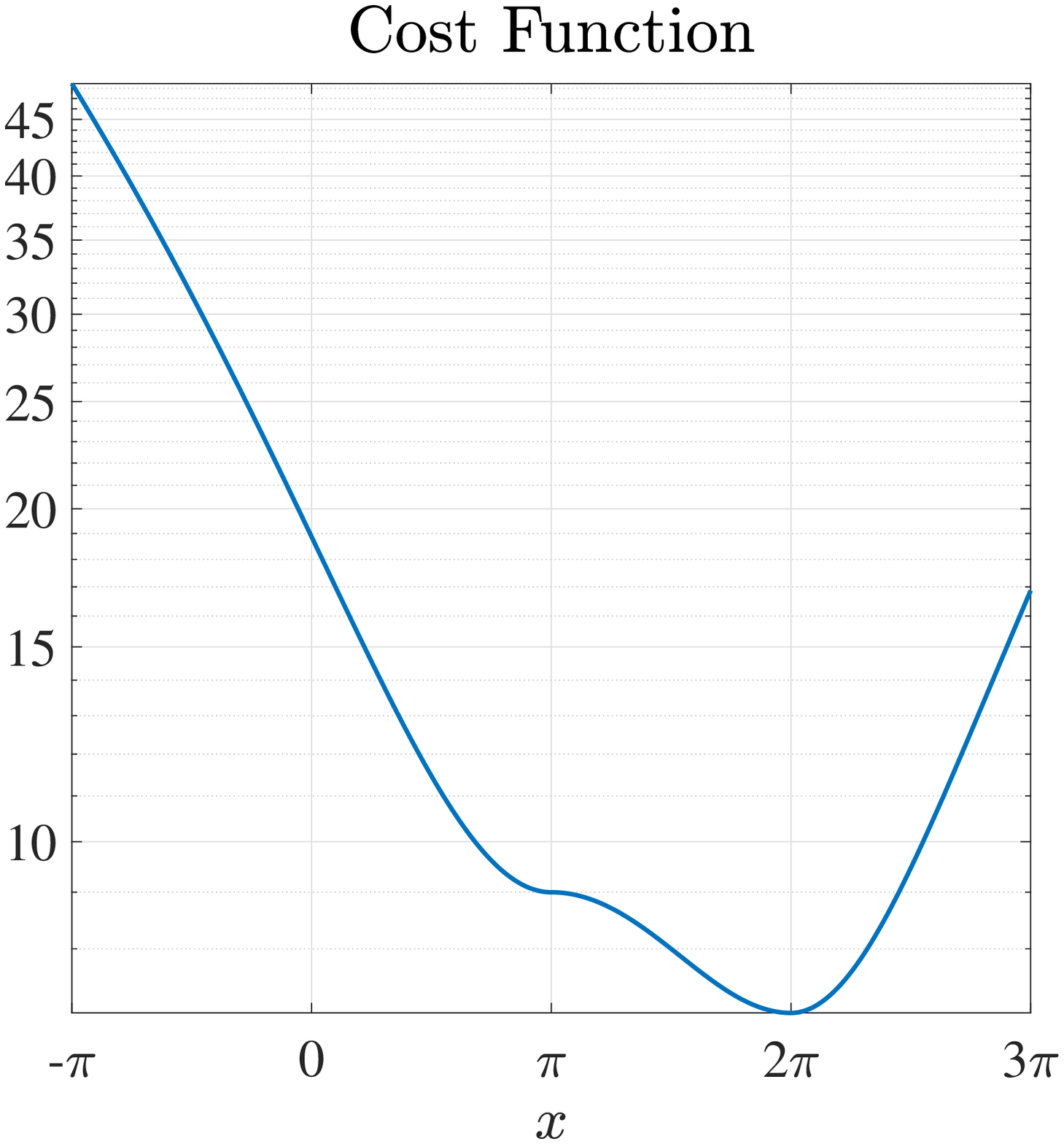}
			\caption{} 
			\label{fig:2ndMap_a}
		\end{subfigure}
			\begin{subfigure}{.48\columnwidth}
		\includegraphics[width=\textwidth]{EquilibriumDelta001}\llap{\includegraphics[height=2.5cm]{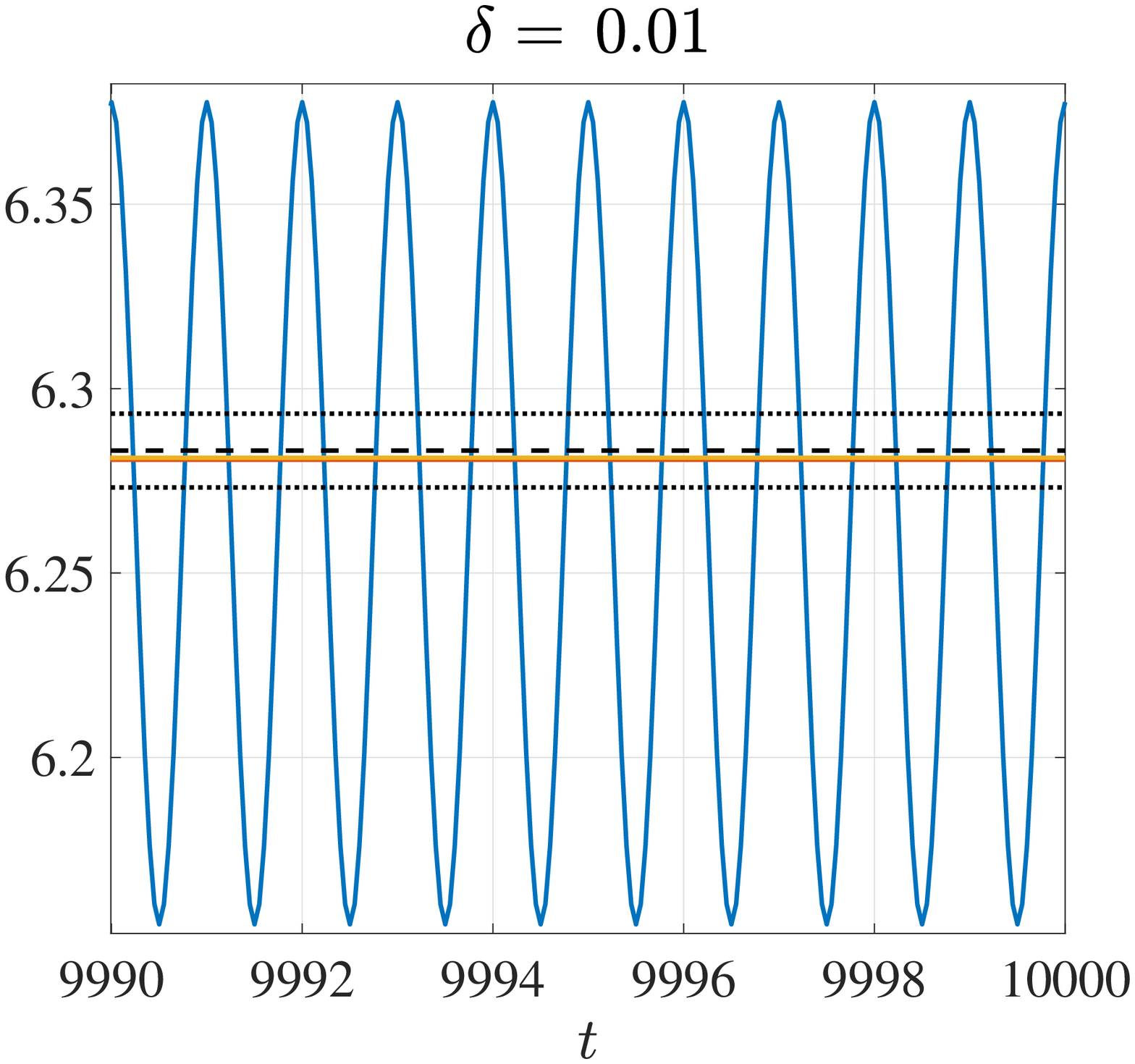}}\llap{\includegraphics[height=1.3cm]{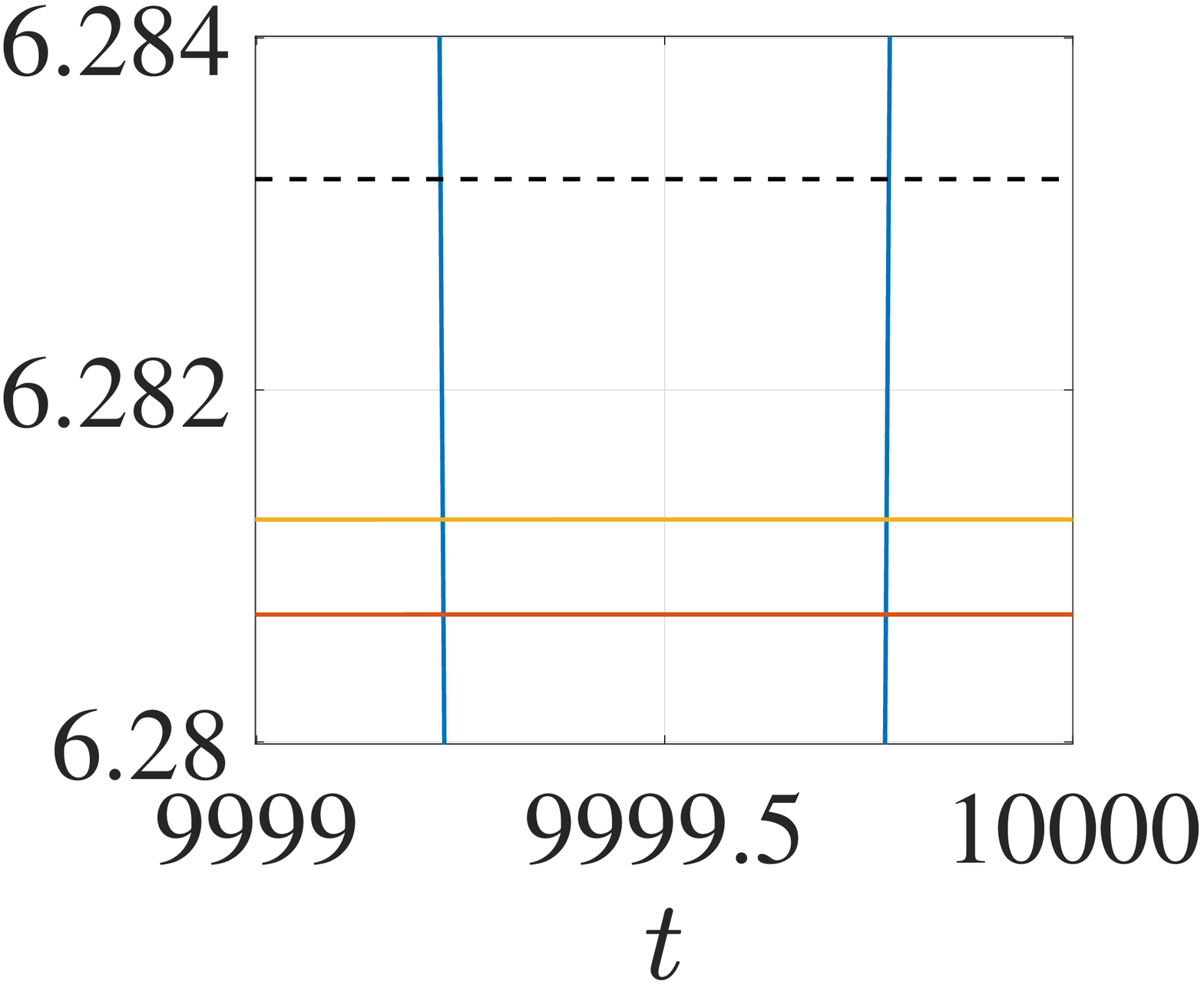}}
		\caption{} 
		\label{fig:2ndMap_d}
	\end{subfigure}
		\begin{subfigure}{.48\columnwidth}
	\includegraphics[width=\textwidth]{EquilibriumDelta01}\llap{\includegraphics[height=2.5cm]{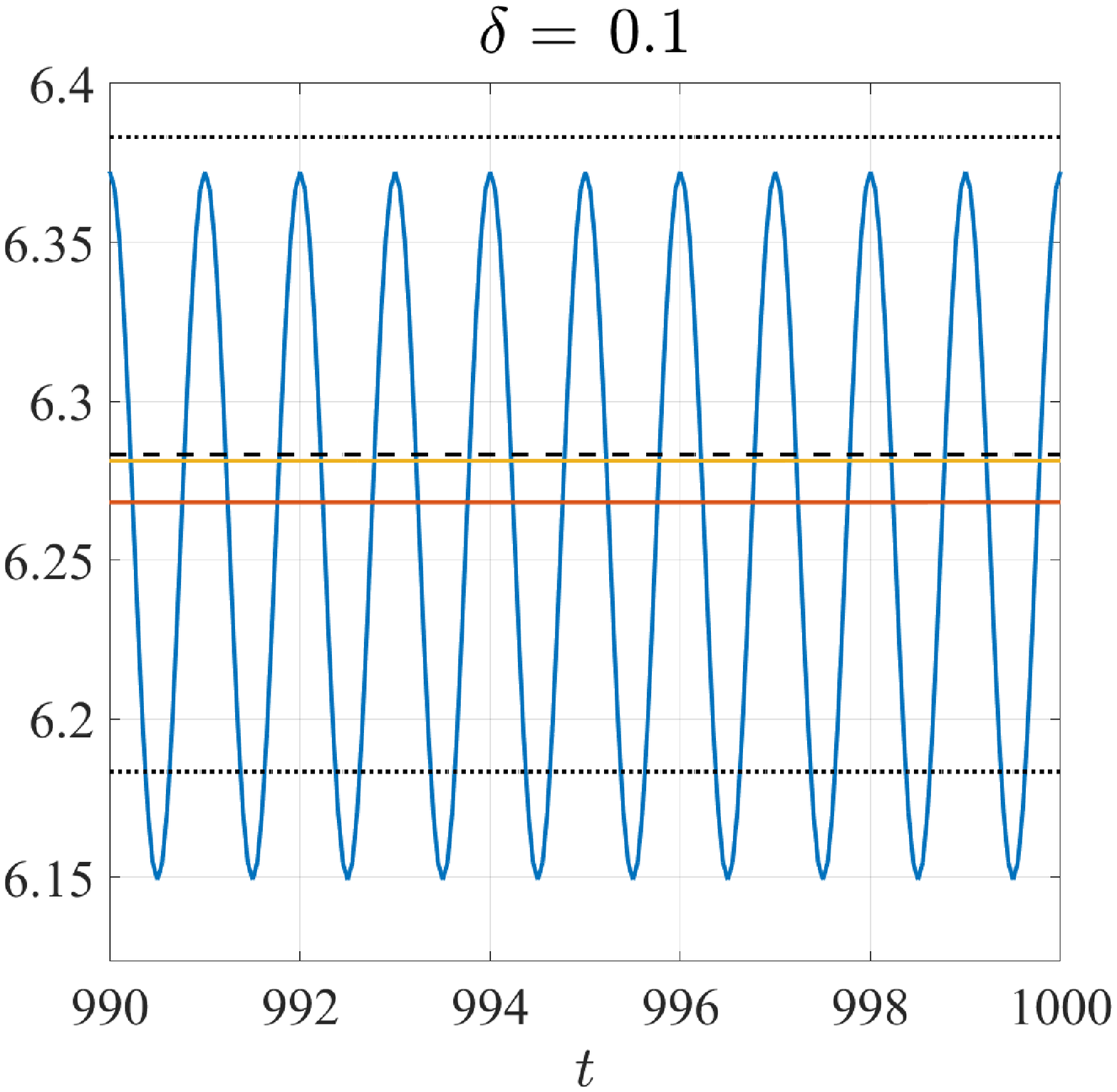}}
	\caption{} 
	\label{fig:2ndMap_c}
\end{subfigure}
		\begin{subfigure}{.48\columnwidth}
			\includegraphics[width=\textwidth]{Saddle}\llap{\includegraphics[height=2.3cm]{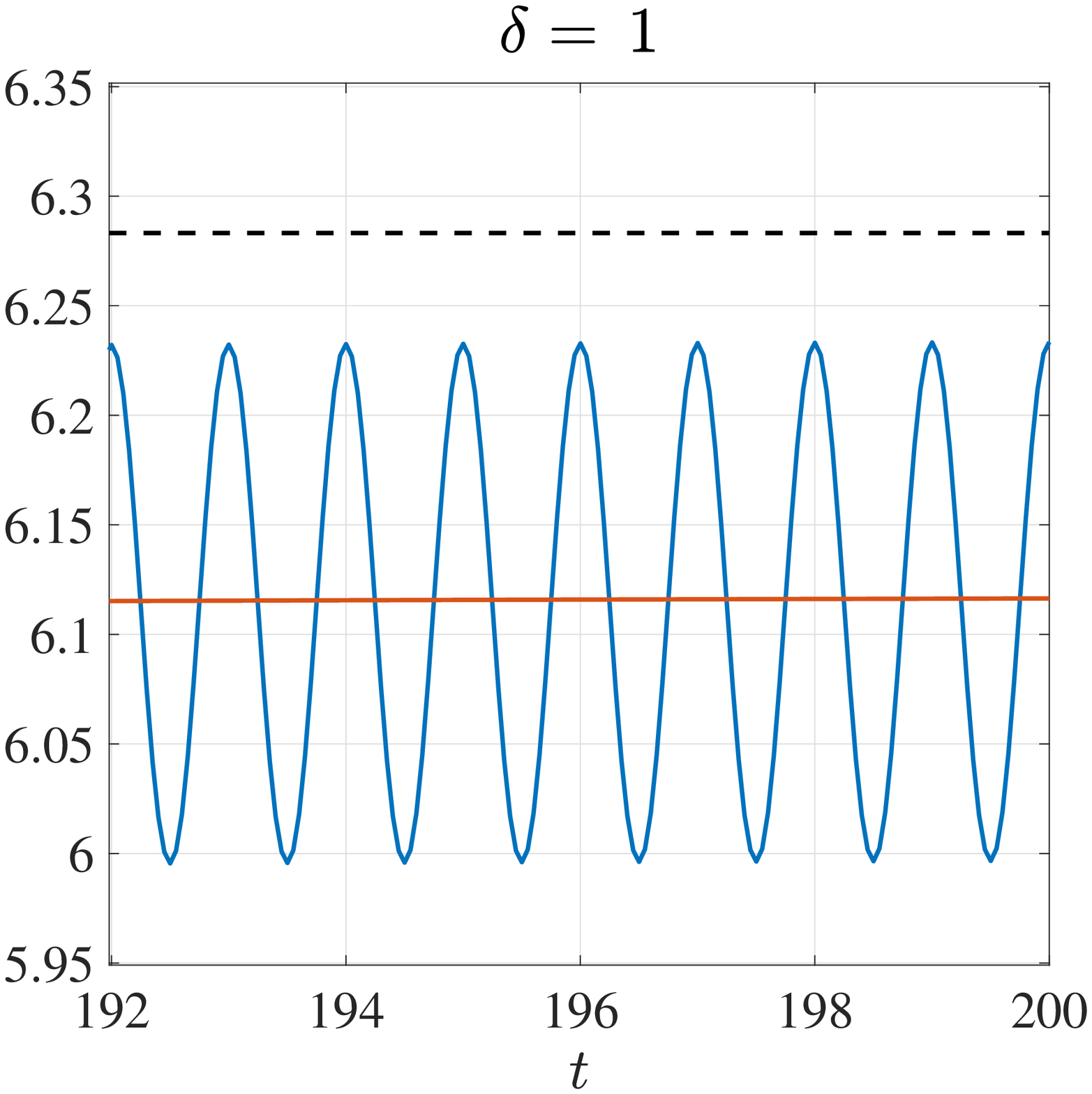}}
			\caption{} 
			\label{fig:2ndMap_b}
		\end{subfigure}
		\caption{(a) Cost function \eqref{eq:h2ndCase} for $A=0, \, h_0 = 10$. It presents multiple points with null first derivative and it is also non-symmetric around the minimiser. The saddle point is located at $x = \pi$ and the minimiser is at $x^\star = 2\pi$. (b)-(c) The average of the classic ES converges to an equilibrium point within the set $[x^\star-\delta,\, x^\star+\delta]$ accordingly to what foreseen in Lemma \ref{lemma:Reduced1}. (d) The presence of saddle points stacks the Taylor-based averaging (yellow) of the classic ES (blue). Vice versa, the Fourier-based averaging (red) better represents the actual behaviour of the classic ES (blue). In general, the average based on the Fourier series tracks the classic ES more accurately than the average based on the Taylor expansion. These results are obtained for $\gamma = 0.1$.}
		\label{fig:2ndMap}
	\end{figure}
}

\ifthenelse{\boolean{CONF}}{}{{To simulate the case of a strictly convex function, $x_0=6$ is picked sufficiently close to the minimiser of the cost function in Figure \ref{fig:2ndMap_a}. Then, Figures \ref{fig:2ndMap_d}-\ref{fig:2ndMap_c} show the convergence performance of the basic ES together with the average computed by \eqref{eq:First_Order} (yellow) and by \eqref{eq:average-Fourier} (red). 
 These simulations confirm that, in agreement with Lemma \ref{lemma:Reduced1} and Proposition \ref{prop:BasicES}, the optimisation variable $x$ converges to $\mathcal{A}_\delta \subset [x^\star-\delta,\,x^\star+\delta]$.  In more detail, the asymmetry of $h(\cdot)$ in the neighbourhood of $x^\star$ implies that the equilibrium point of \eqref{eq:average} does not correspond to $x^\star$, see yellow lines in Figures \ref{fig:2ndMap_c} and \ref{fig:2ndMap_d}. 
 
To simulate a quasi-strictly convex function, $x_0=-\pi$ is picked sufficiently large to make \eqref{eq:BasicES} passing through the saddle point of the cost function depicted in Figure \ref{fig:2ndMap_a}. In this scenario,  Figure \ref{fig:2ndMap_b} shows that the trajectory of  \eqref{eq:First_Order} (yellow) gets stacked in presence of local extrema whereas the trajectories of the proposed Fourier-based \eqref{eq:average-Fourier} (red) remain close to those of \eqref{eq:BasicES}, as foreseen by Proposition \ref{prop:BasicES}.}
\begin{figure}[h!]
	\centering
	\begin{subfigure}{.48\columnwidth}
		\includegraphics[width=\textwidth]{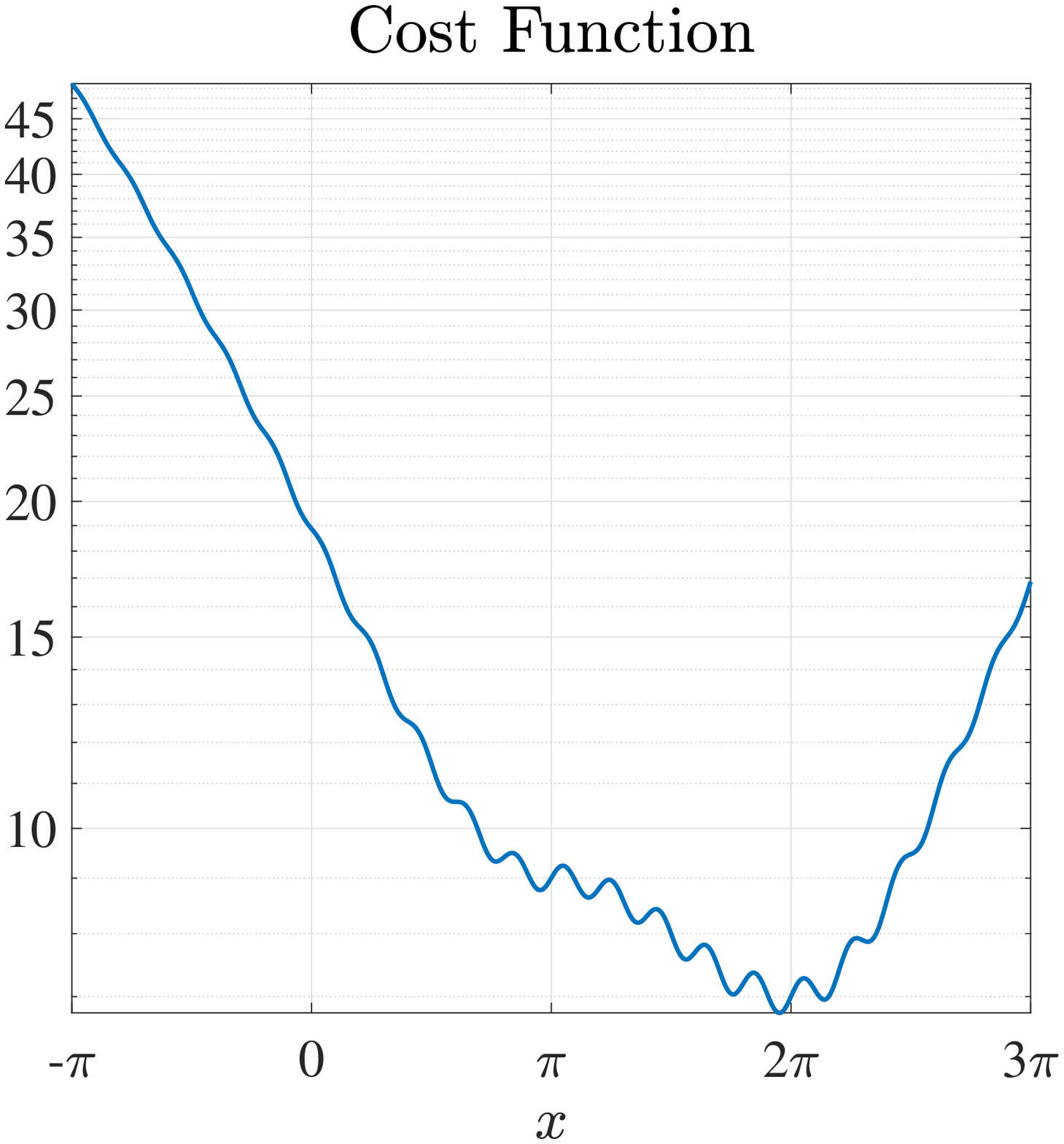}
		\caption{} 
		\label{fig:LocalMinima_a}
	\end{subfigure}
	\begin{subfigure}{.48\columnwidth}
		\includegraphics[width=\textwidth]{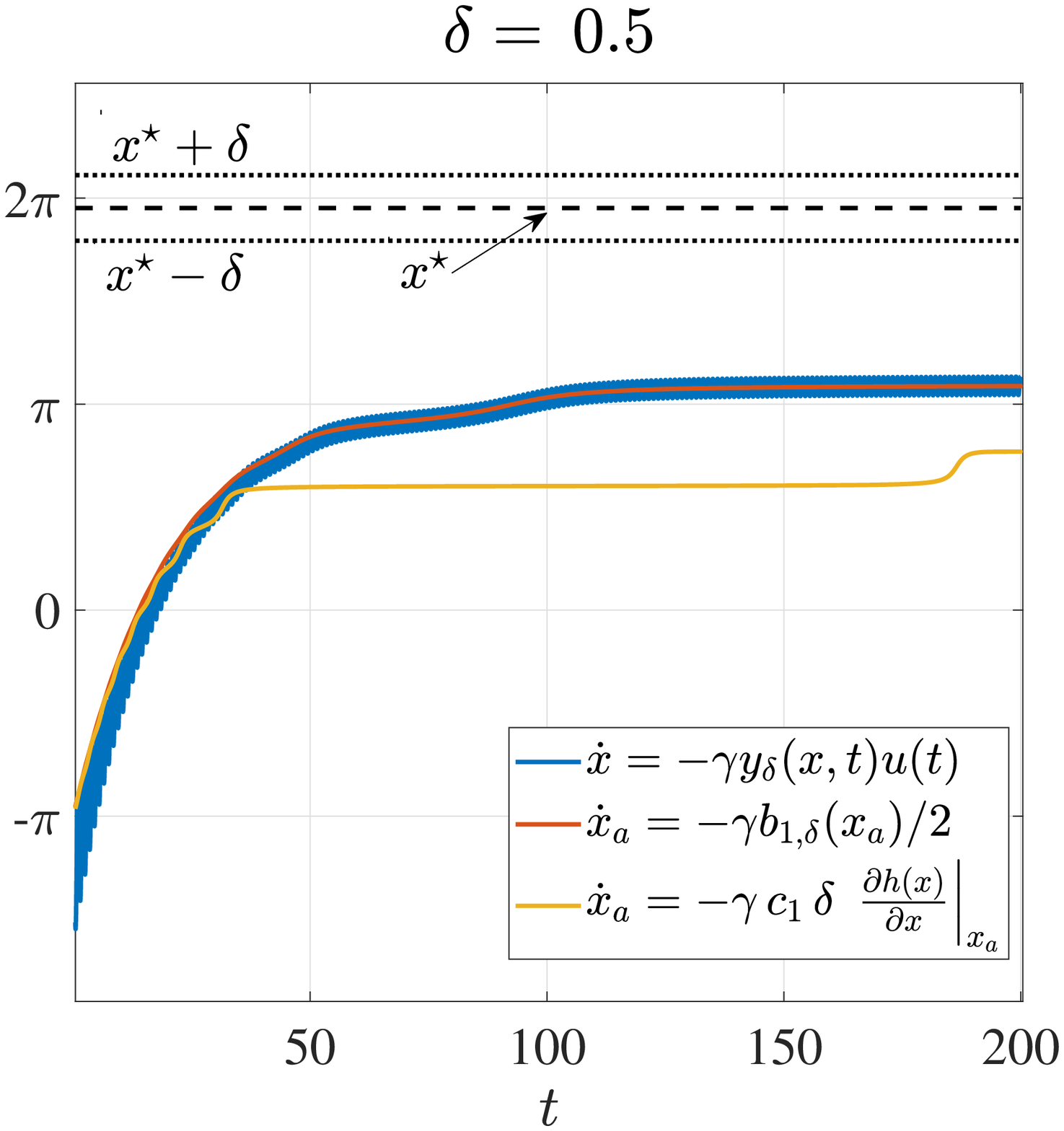}\llap{\includegraphics[height=2.2cm]{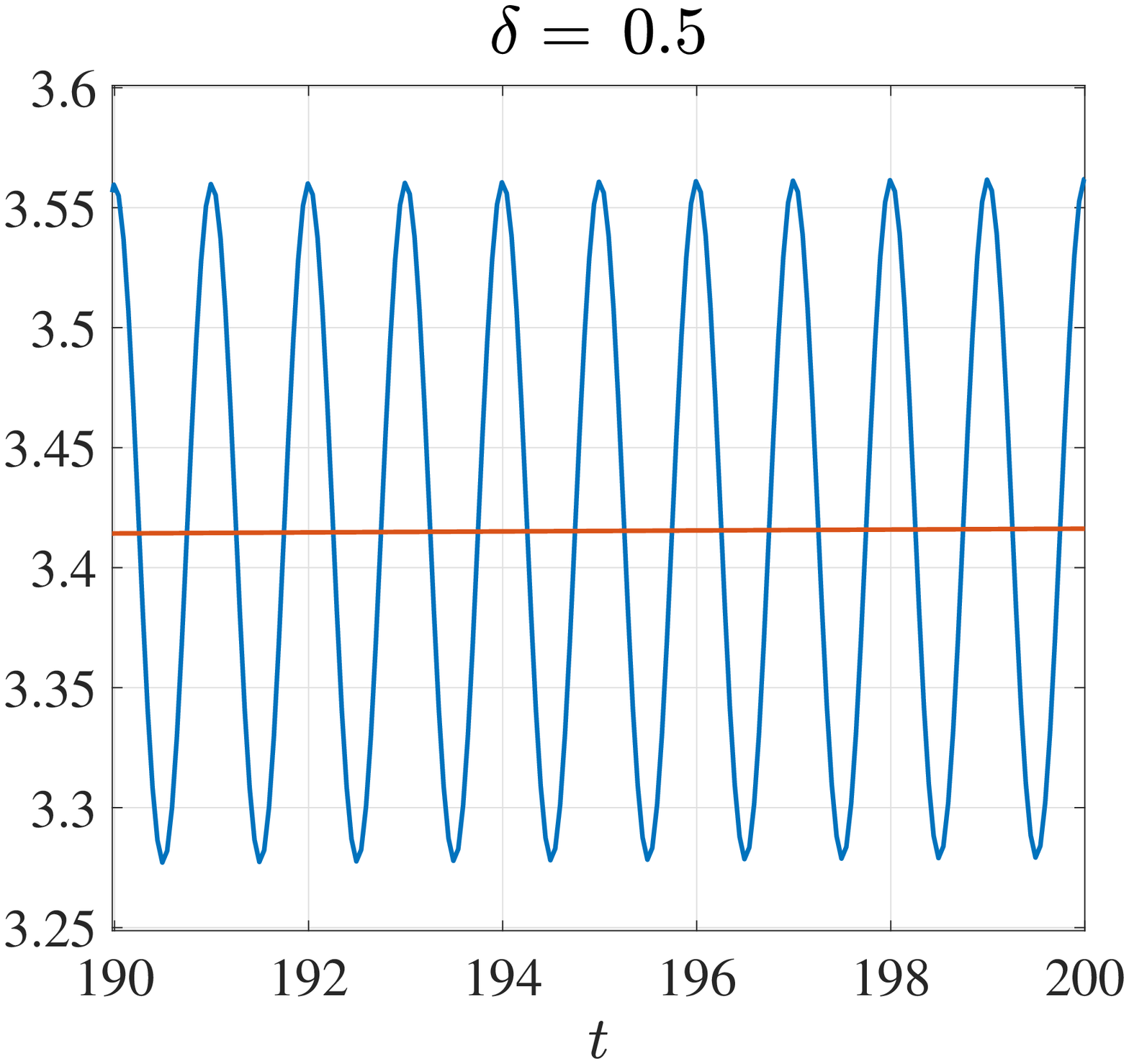}}
		\caption{} 
		\label{fig:LocalMinima_b}
	\end{subfigure}
	\begin{subfigure}{.48\columnwidth}
		\includegraphics[width=\textwidth]{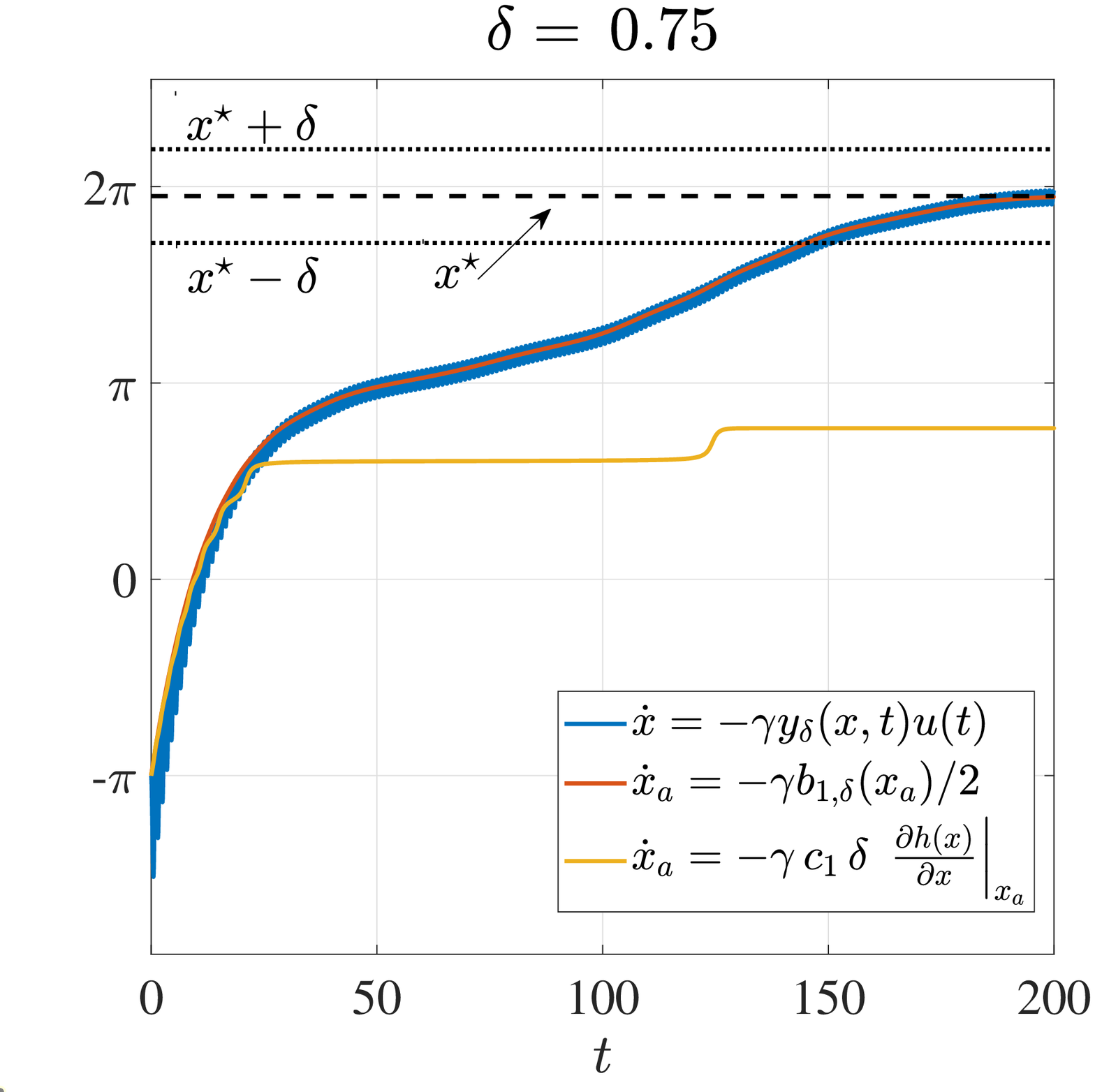}\llap{\includegraphics[height=2.2cm]{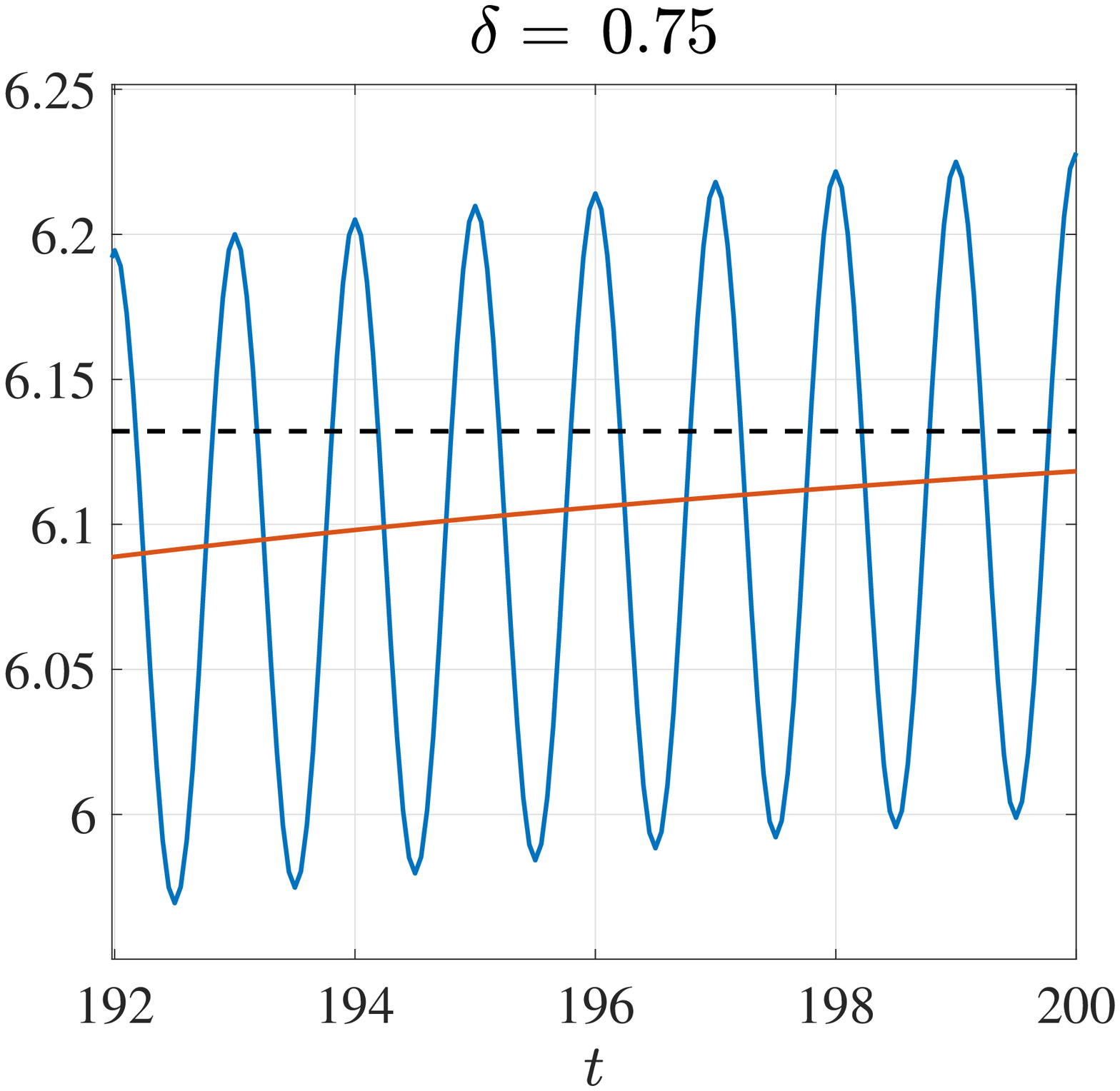}}
		\caption{} 
		\label{fig:LocalMinima_c}
	\end{subfigure}
	\begin{subfigure}{.48\columnwidth}
		\includegraphics[width=\textwidth]{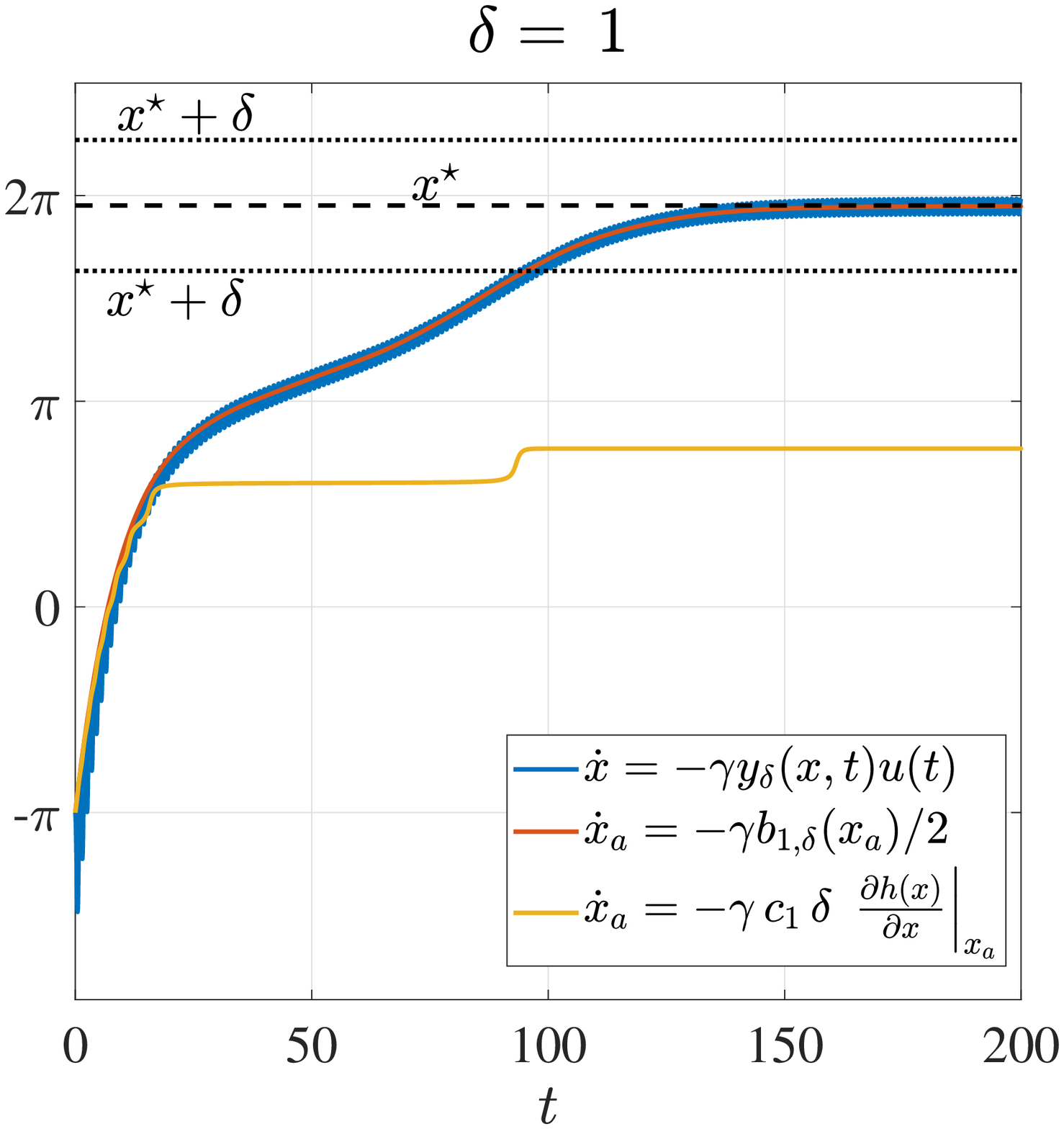}\llap{\includegraphics[height=2.2cm]{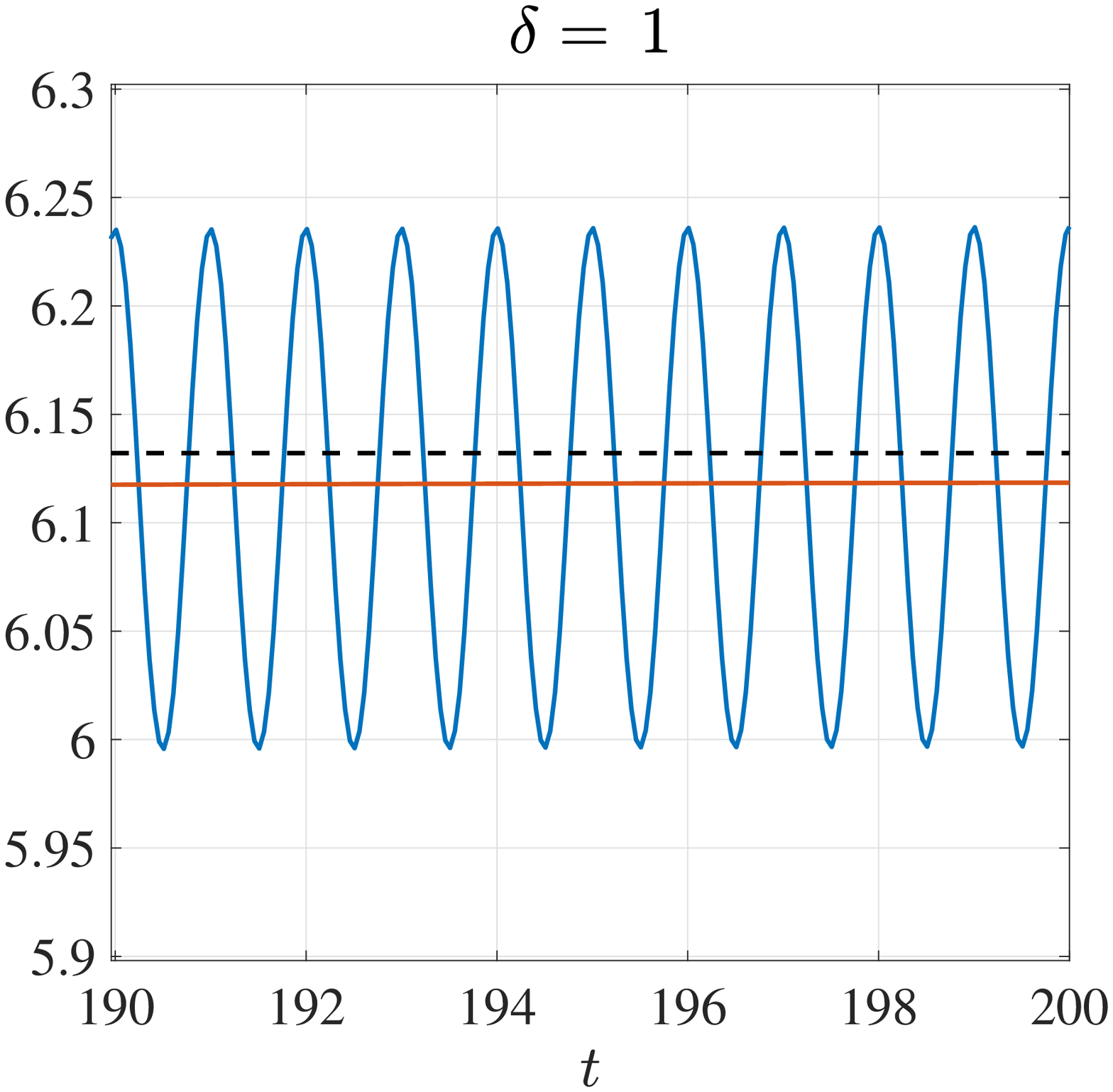}}
		\caption{} 
		\label{fig:LocalMinima_d}
	\end{subfigure}
	\caption{Behaviour of $x(t)$ and $x_a(t)$ computed through \eqref{eq:First_Order} (yellow) and \eqref{eq:average-Fourier} (red) with  $\gamma = 0.1$. A dither amplitude of $\delta = 0.5$ is not sufficient to let \eqref{eq:BasicES} (blue) to overcome local extrema whereas, for $\delta \in \{0.75,1\}$ the plot demonstrate that the set $\Delta = [x^\star-\delta,\, x^\star+\delta]$ is reached in finite time. In all these simulations, the averaging trajectories of \eqref{eq:First_Order} (yellow) get stuck at $x \approx 2.42$ where $\partial h(x)/\partial x = 0$.}
\label{fig:LocalMinima}
\end{figure}
}

{The behaviour of \eqref{eq:BasicES} in the case of local minima is shown in Figure \ref{fig:LocalMinima}, for the cost function depicted in Figure \ref{fig:LocalMinima_a}. A selection of a too-small $\delta$ could trap the ES \eqref{eq:BasicES} (blue) in the neighbourhood of local minima, see Figure \ref{fig:LocalMinima_b}. Then, following Lemma \ref{lemma:Reduced1}, for  sufficiently large $\delta$, the scheme \eqref{eq:BasicES} (blue) can pass through local minima, see Figures \ref{fig:LocalMinima_c} and \ref{fig:LocalMinima_d}. It is worth noting that, also in this case, the averaging based on the first-order Taylor expansion \eqref{eq:First_Order} (yellow) is less accurate than \eqref{eq:average-Fourier} (red) in describing the trajectories of \eqref{eq:BasicES} (blue). A comparison of Figures \ref{fig:LocalMinima_c} and \ref{fig:LocalMinima_d} confirms the correctness of the bounds provided in Lemma \ref{lemma:Reduced1}.}

Before investigating the performance of \eqref{eq:GlobalESGeneric}, the following test is performed to show that the classic ES suffers of large values of $|h(\cdot)|$. Indeed, as depicted in Figure \ref{fig:Mr} {(blue lines in subplots (a) and (c))}, while keeping $\gamma$ fixed, larger $M_r$ lead to more oscillatory behaviours. To mathematically support this result we observe that 
\begin{equation}
	\label{eq:Taylor}
	h(x+\delta u) = h(x) + R(x, \delta u)
\end{equation}
where $R(\cdot,\cdot)$ represents the remainder of the Taylor expansion around $x$ of $h(\cdot)$. Exploit the definition of the Lipschitz constant of $h(\cdot)$ and $|u(t)|_\infty \le 1$ to bound the remainder from above as
\begin{equation}
	\label{eq:BoundRemainder}
	|h(x+\delta u) - h(x)| = |R(x, \delta u)| \le L_r\delta.
\end{equation}
Substitute \eqref{eq:Taylor} into \eqref{eq:BasicES} 
\begin{equation}
	\label{eq:ClassicESTaylor}
	\dot{x} = -\gamma h(x+\delta u) u = -\gamma (h(x)+R(x,\delta u)) u
\end{equation}
and investigate the following support system
\begin{equation}
	\label{eq:ClassicESTaylorApprox}
	\dot{x}_1 =-\gamma h(x_1)  u\qquad x(0) = x_{10}
\end{equation}
conceivable as approximation of \eqref{eq:ClassicESTaylor} for $|h(x)| \gg L_r\delta$. Let $H(x):= \int h(x)^{-1} dx$ and solve \eqref{eq:ClassicESTaylorApprox} by parts as
\begin{equation}
	x_1(t) = H^{-1}(H(x_{10}) -\gamma (\cos(t)-1) )
\end{equation}
which is a pure oscillation whose amplitude is proportional to $|H(x_{10})|${, evident in subplot (c) of Figure \ref{fig:Mr}.}
\begin{figure}[t!]
	\centering
		\begin{subfigure}{.48\columnwidth}
		\includegraphics[width=\textwidth]{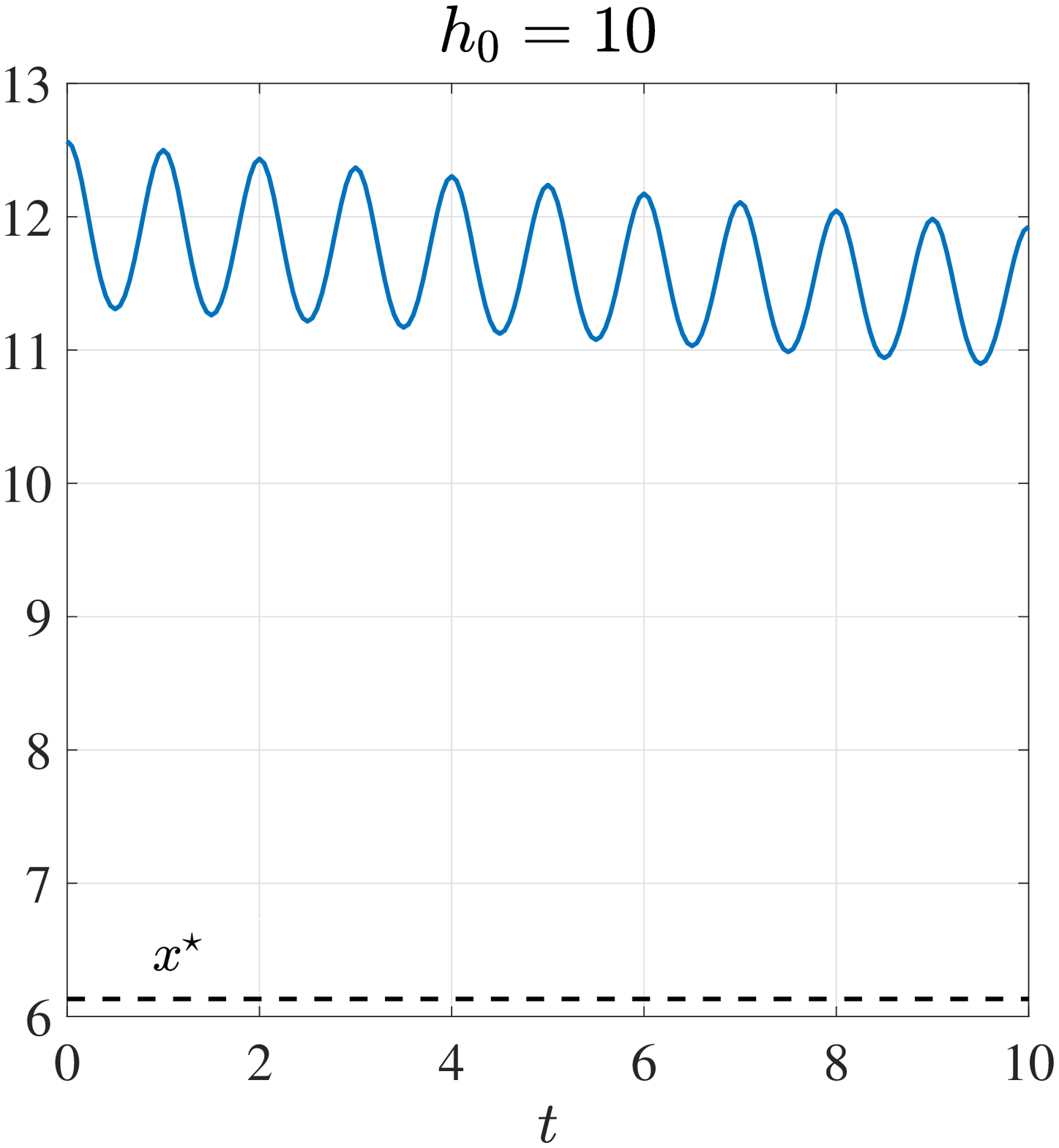}
		\caption{} 
	\end{subfigure}
		\begin{subfigure}{.48\columnwidth}
	\includegraphics[width=\textwidth]{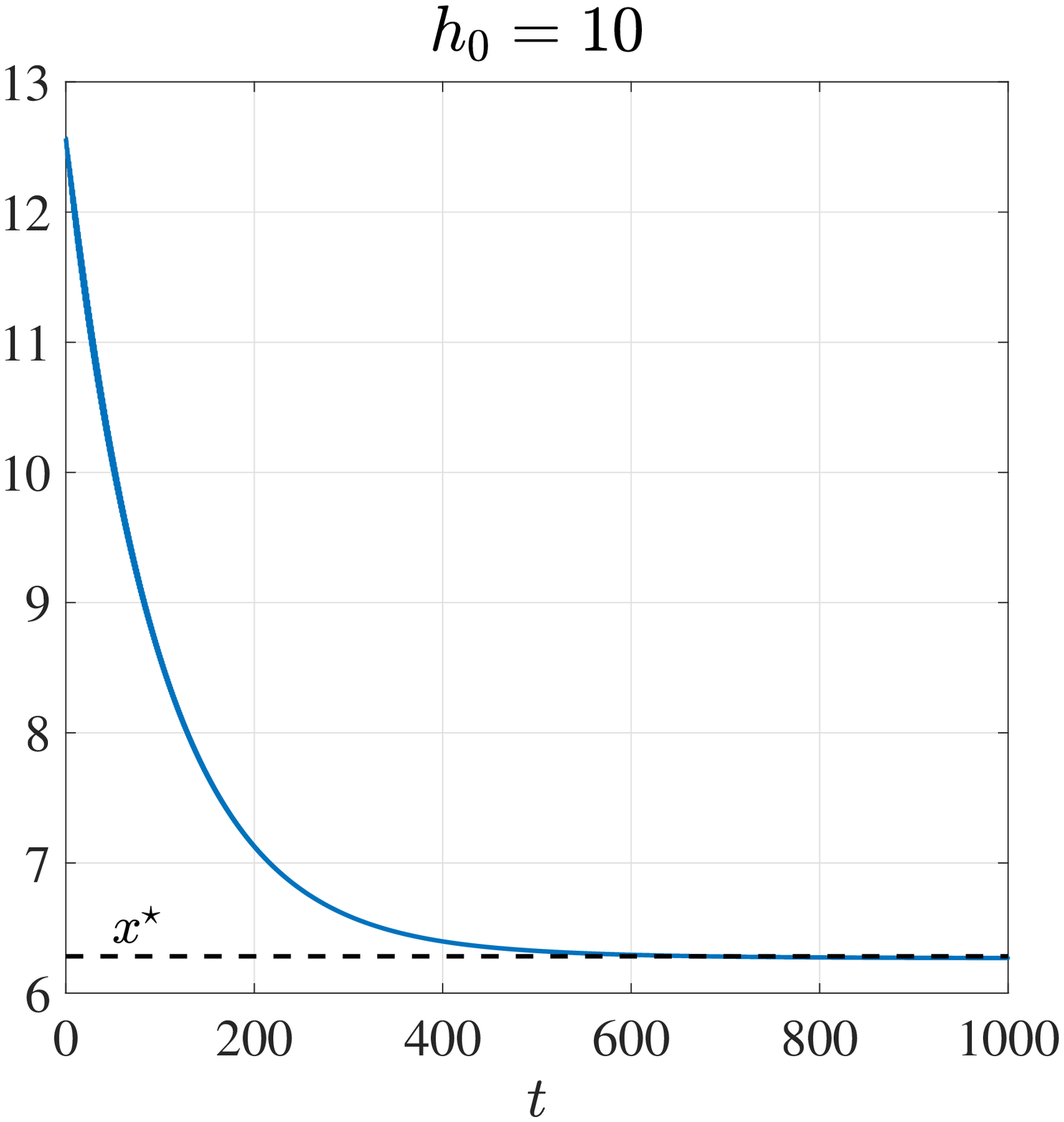}
	\caption{} 
\end{subfigure}
		\begin{subfigure}{.48\columnwidth}
	\includegraphics[width=\textwidth]{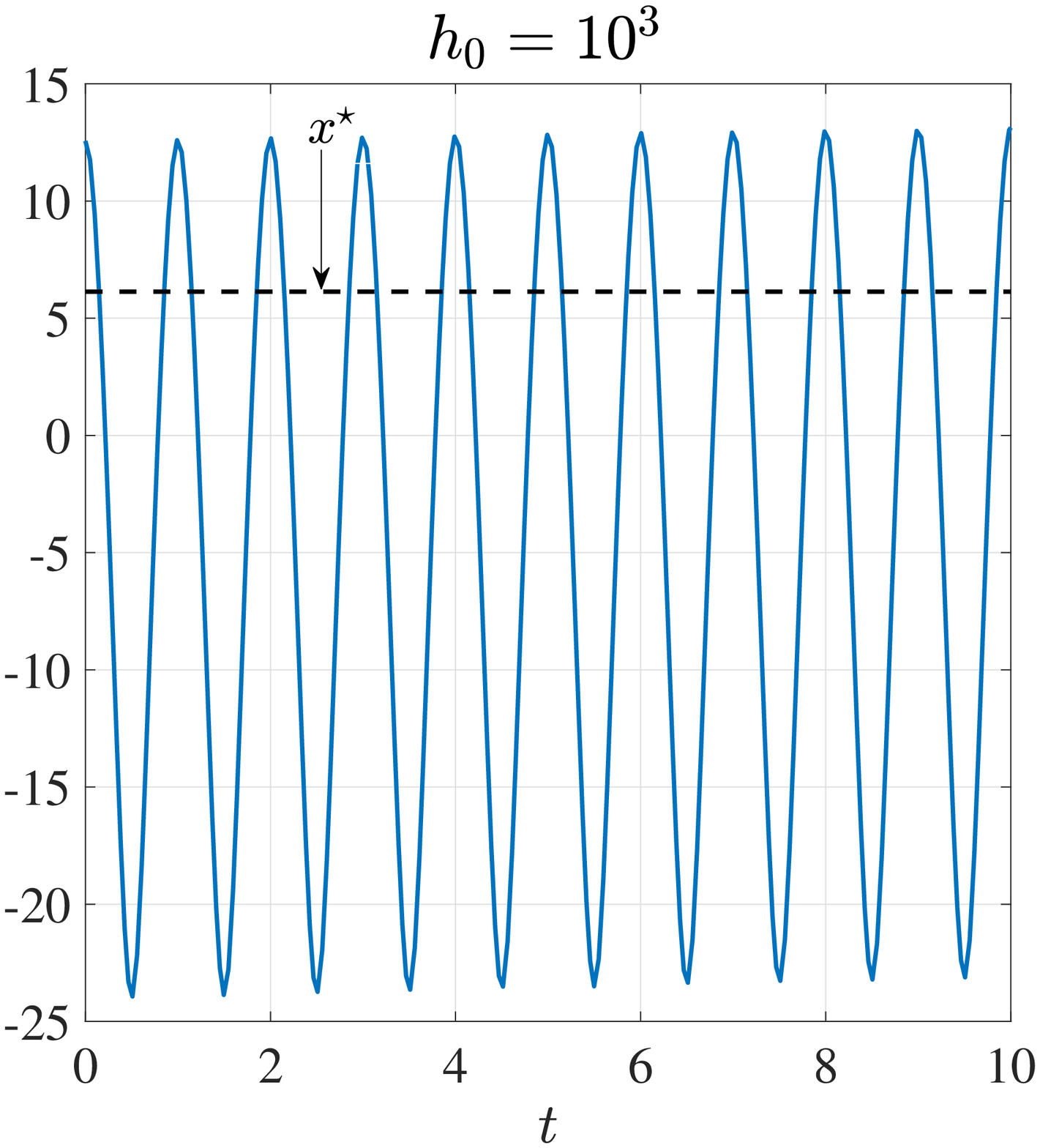}
	\caption{} 
\end{subfigure}
		\begin{subfigure}{.48\columnwidth}
	\includegraphics[width=\textwidth]{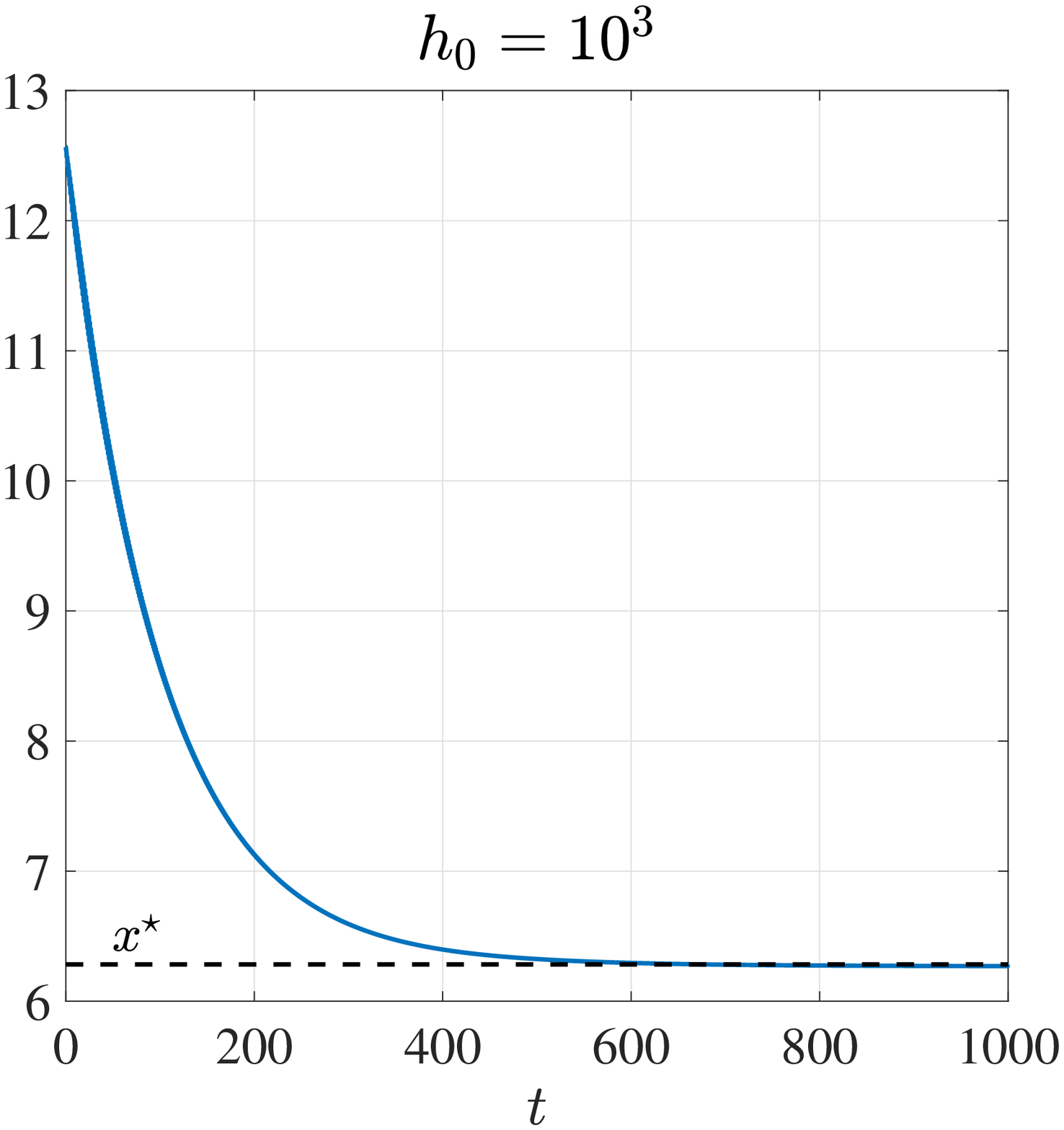}
	\caption{} 
\end{subfigure}
	\caption{Comparison of the behaviours of the classic ES (left) and that improved through the high-pass filter (right). In these simulations, we set $A = 0$, $x_0 = 4\pi$, and we increased the value of $h_0$ from $10$ to $10^3$. Subplots (a) and (c) demonstrate that, keeping fixed $\gamma$, the classic ES (blue) becomes oscillatory for large values of the cost function. On the opposite, subplots (b) and (d) show that the ES improved with the high-pass filter demonstrates a converge rate, uniform in $M_r$.	These simulations are performed with $\gamma = \delta = 0.1$.}
	\label{fig:Mr}
\end{figure}

The performance of the ES improved with the high-pass filter  \eqref{eq:GlobalESGeneric} is tested increasing the cost function.  Figure \ref{fig:Mr} shows that the convergence rate of \eqref{eq:GlobalESGeneric} is uniform with the amplitude of the cost. To conclude, the stability properties of the dynamics of the high-pass filter are tested in the simulations of Figure \ref{fig:NovelES} for the same test conditions of Figure \ref{fig:Mr}. These tests show that the average \eqref{eq:DotBarYAV} (red) accurately tracks the exact dynamics \eqref{eq:bary} (blue).

\begin{figure}[h!]
	\centering
	\includegraphics[width=0.48\columnwidth]{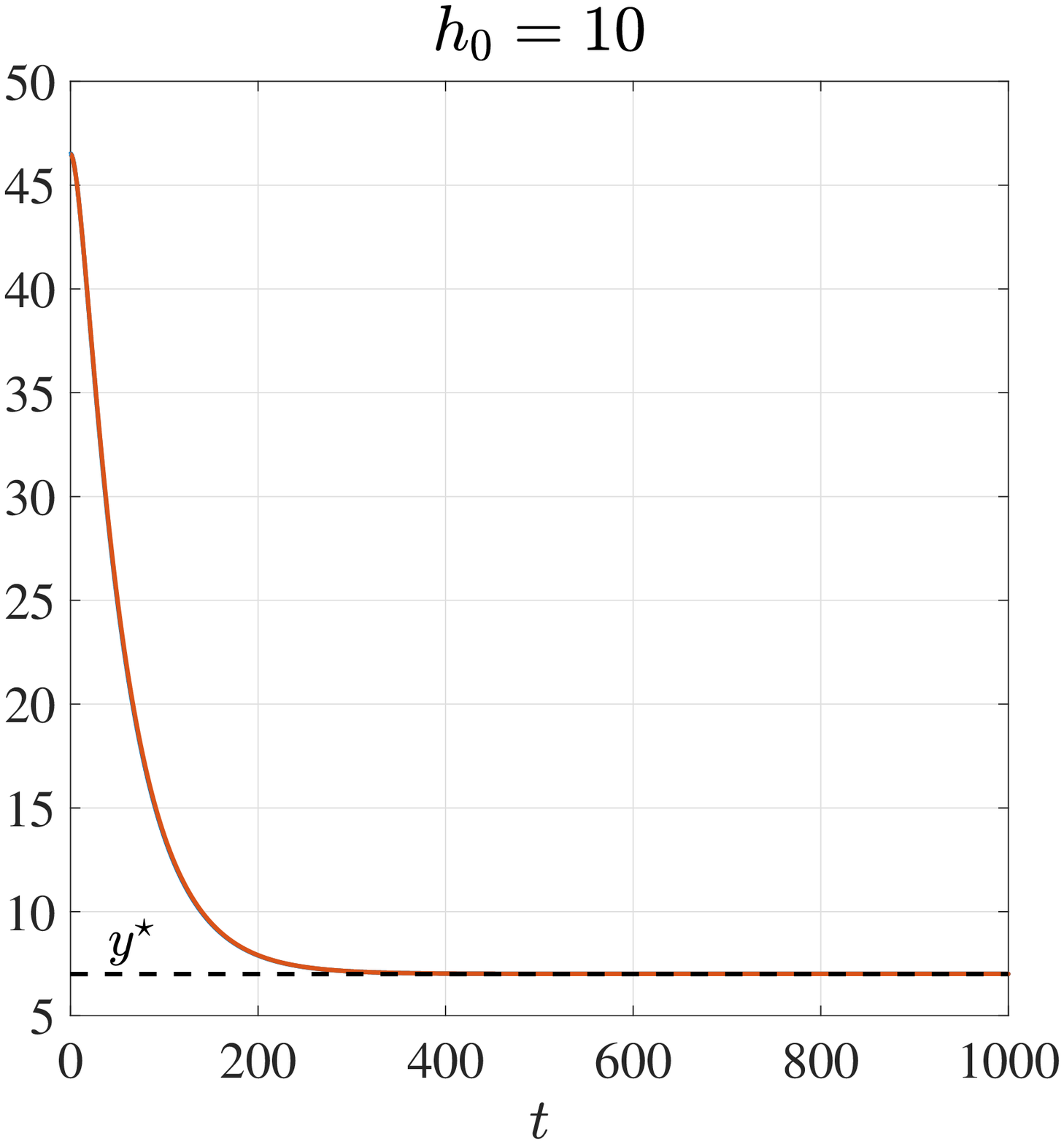}\llap{\raisebox{0.9cm}{\includegraphics[height=2.7cm]{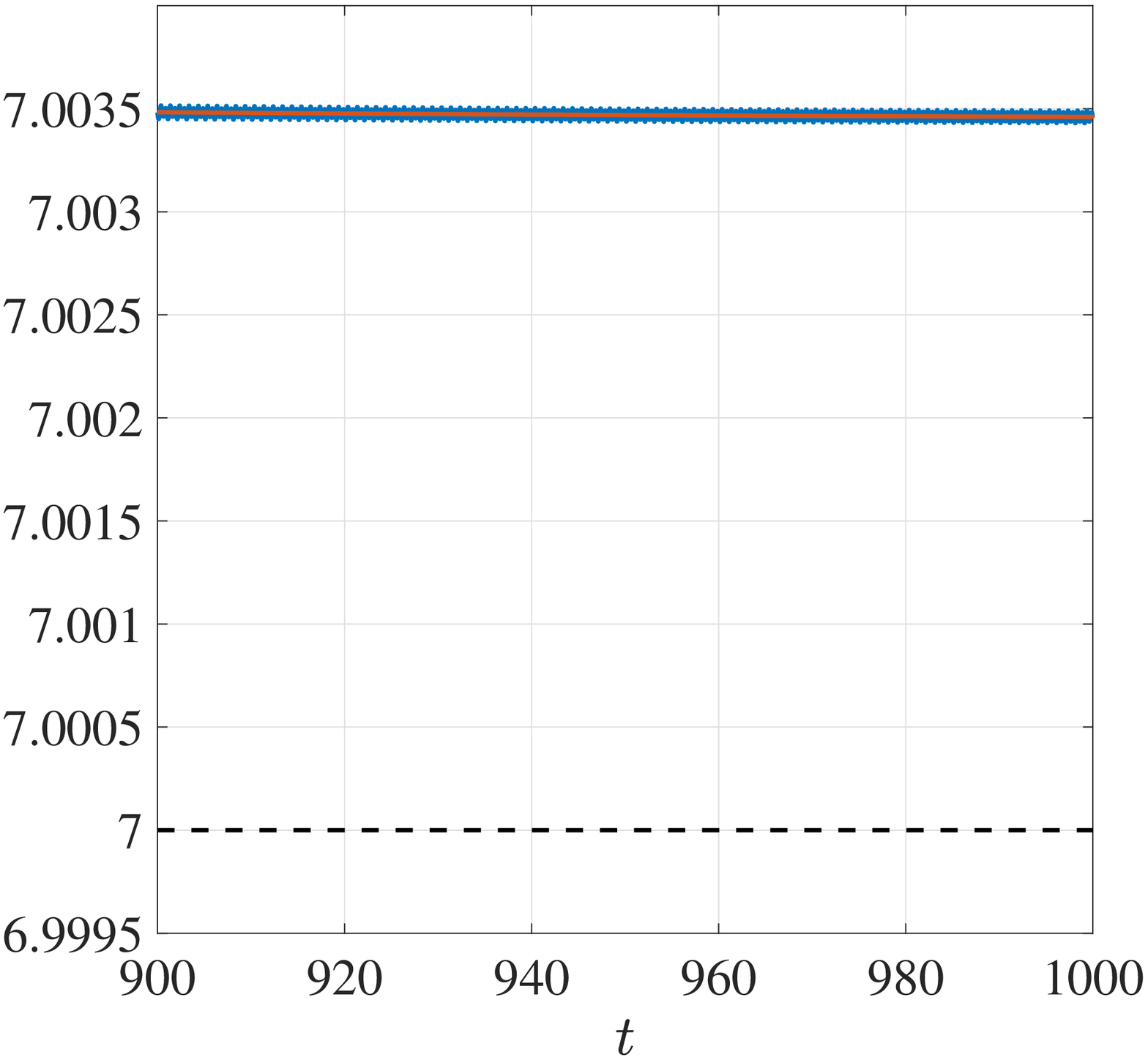}}}\llap{\raisebox{1.55cm}{\includegraphics[height=1.5cm]{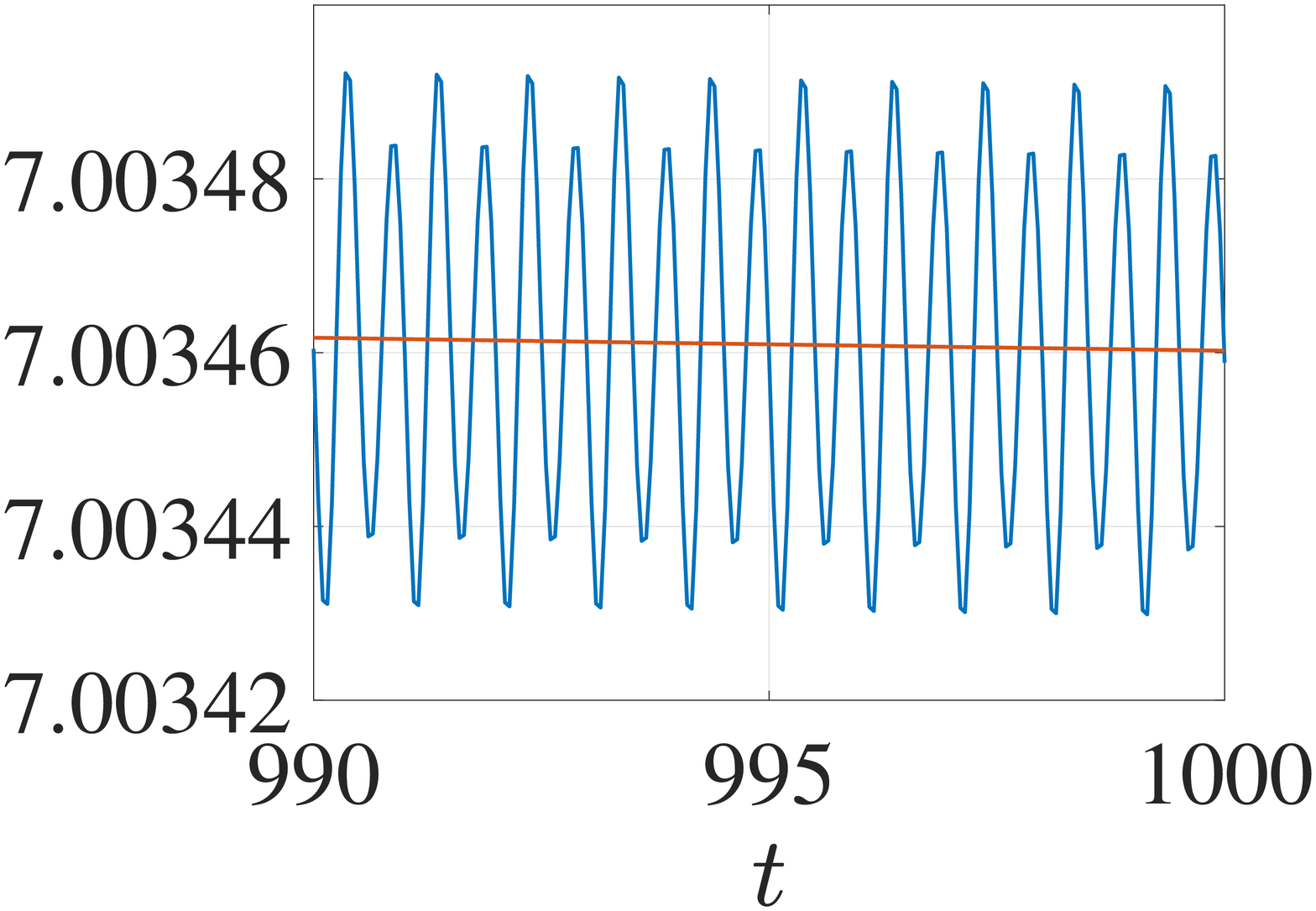}}}
	\includegraphics[width=0.48\columnwidth]{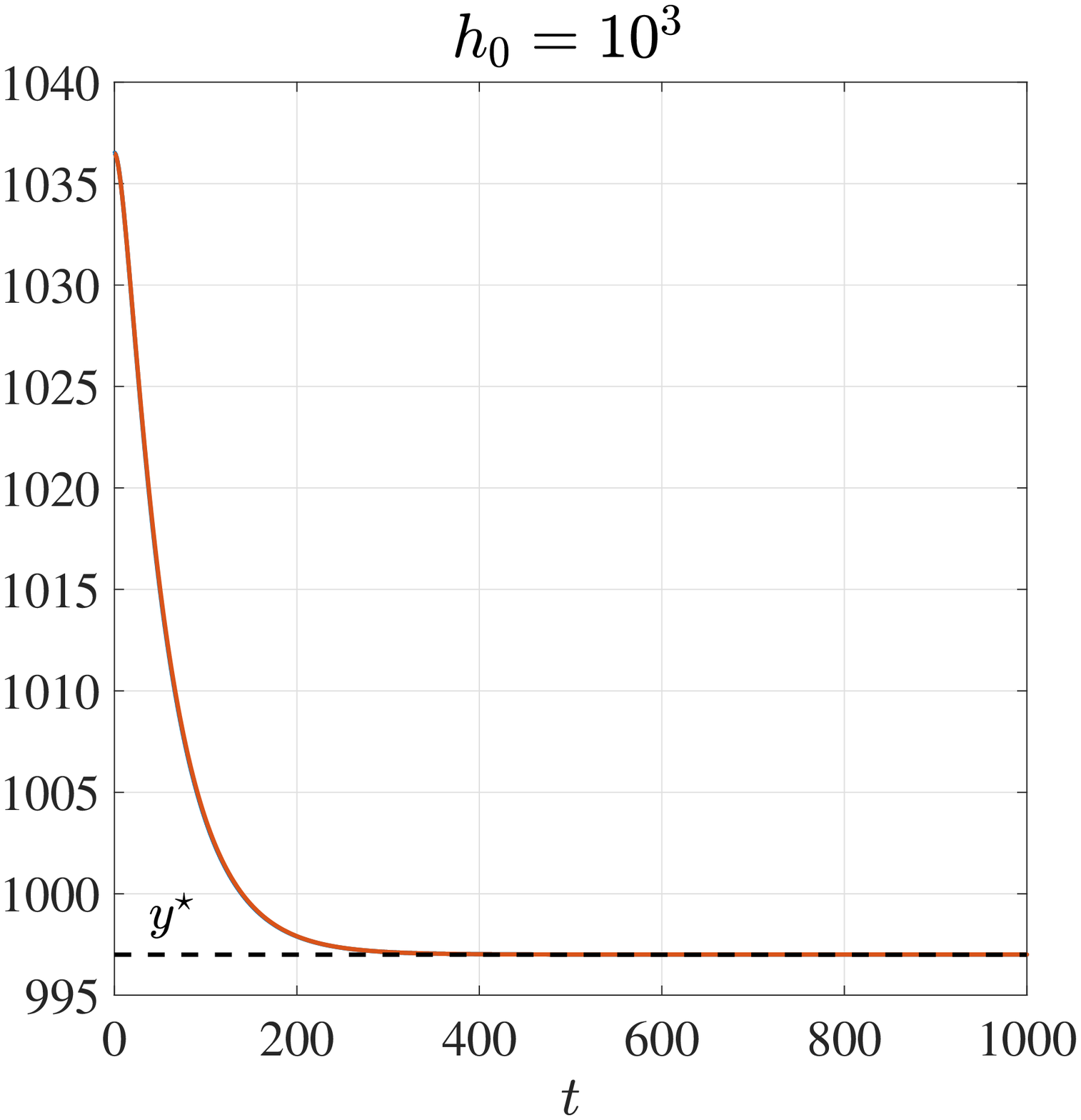}\llap{\raisebox{0.9cm}{\includegraphics[height=2.7cm]{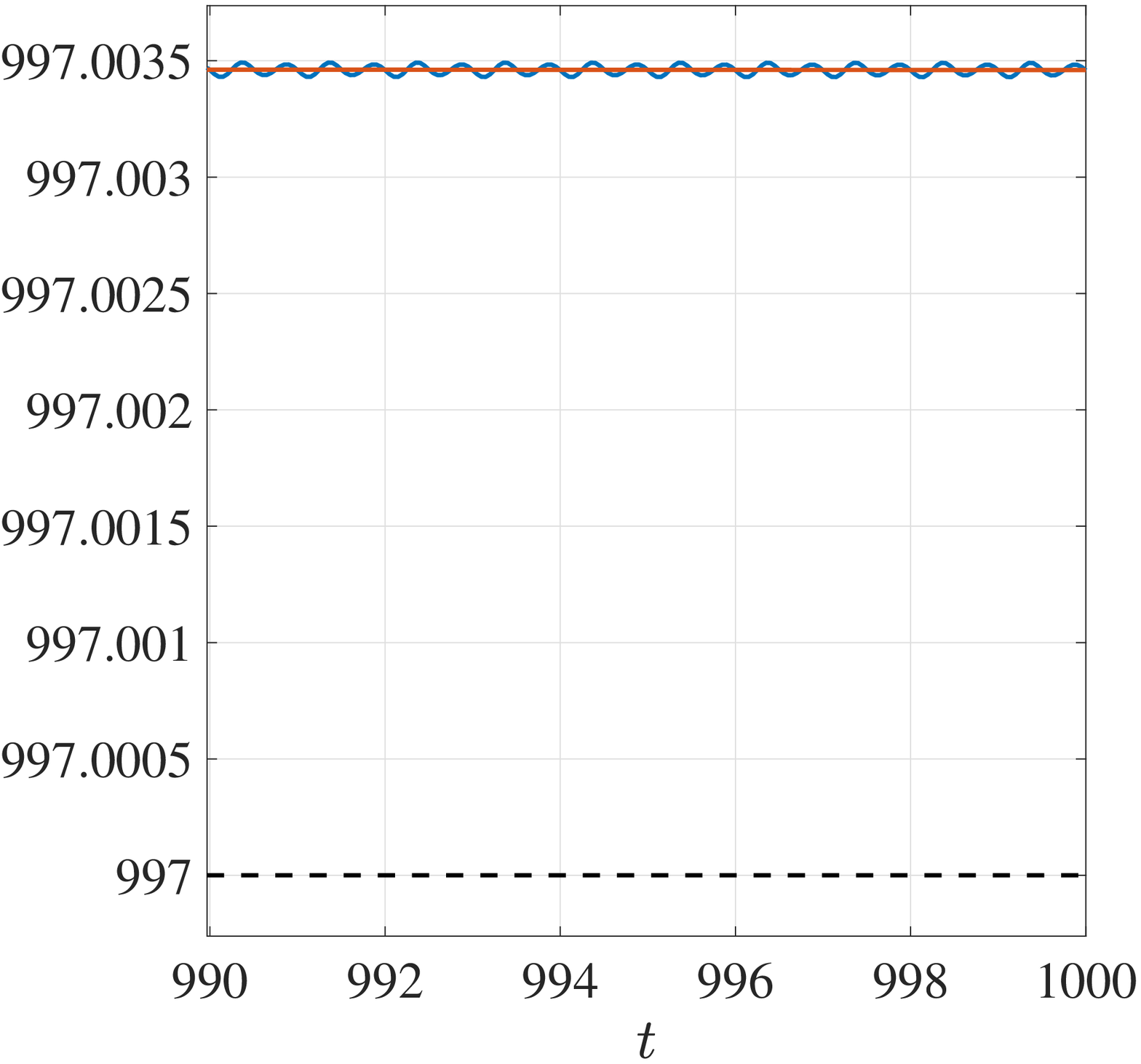}}}\llap{\raisebox{1.45cm}{\includegraphics[height=1.5cm]{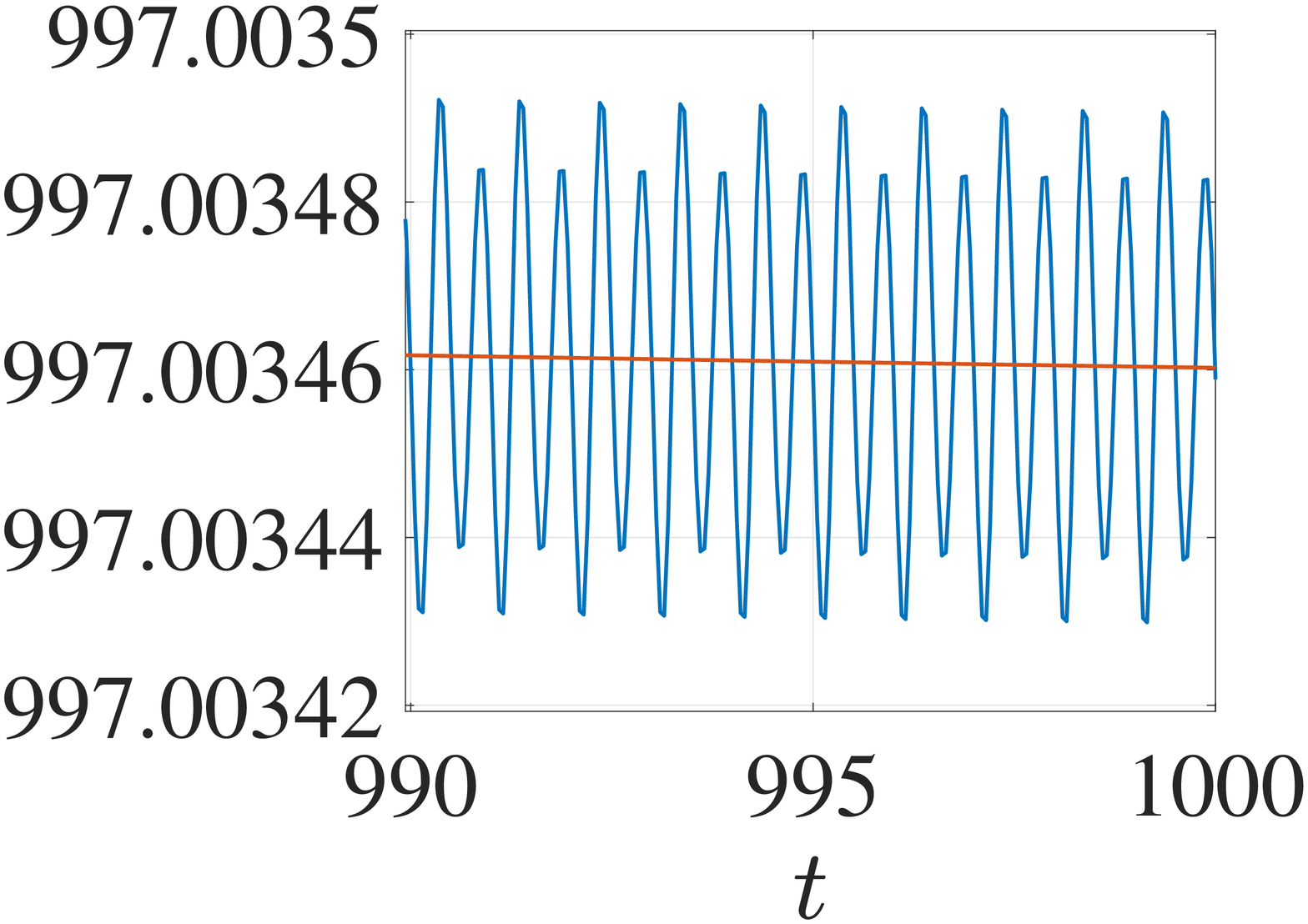}}}
	\caption{The behaviour of the HPF-ES does not change for increasing $M_r$. In these simulations the value of $h_0$ is increased from $10$ to $10^3$ demonstrating that, keeping fixed $\gamma$, the average \eqref{eq:DotBarYAV} (red) accurately tracks the exact dynamics \eqref{eq:bary} (blue). These simulations are performed with $\gamma = \delta = 0.1$.}
	\label{fig:NovelES}
\end{figure}

Through the {proposed} simulations we confirm three results: a) the classic ES can deal with cost functions with local \ifthenelse{\boolean{CONF}}{{saddle points}}{minima} where, on the opposite, the average of the classic ES obtained through the Taylor expansion gets stacked (Figure \ref{fig:2ndMap_b} and Figure \ref{fig:LocalMinima}); b) the classic ES and its average via Taylor expansion do not converge to the same equilibrium point (Figure \ref{fig:2ndMap_c});  c) in case of large cost functions the classic ES behaves as an oscillator whose amplitude is proportional to the value of the cost function (Figure \ref{fig:Mr}c). Vice versa, the adoption of the high-pass filter makes the basic ES convergence rate uniform with respect to the amplitude of the cost function (Figures \ref{fig:Mr}b and \ref{fig:Mr}d).  

\ifthenelse{\boolean{CONF}}{}{
\subsection{Solar Panel Optimisation}

As demonstrated in \cite{Li2016Detection}, the ES algorithm can be efficiently applied to optimise the performance of solar panels. In this section, we compare the performance of \eqref{eq:BasicES} and \eqref{eq:GlobalESGeneric} when applied to the photovoltaic array model presented in \cite{Li2016Detection}, and briefly recalled hereafter.

Let $\zeta \in \mathbb{R}^2$ be the state of the photovoltaic array modelling the dynamics of the output voltage $x \in\mathbb{R}$, forced by the input voltage $z \in \mathbb{R}$. Then, the plant is modelled as a linear time-invariant system
\begin{equation}
	\label{eq:PCModel}
\begin{aligned}
	\dot{\zeta} = &\, A \zeta + B z\\
	x =&\, C\zeta
\end{aligned}
\end{equation}
with $A$ Hurwitz. Since $A$ is Hurwitz, there exists a linear map $L\,:\mathbb{R}\to \mathbb{R}$ such that the set $\Omega := \{x = L z,\,z\in\mathbb{R}\}$ is globally exponentially stable. 
On the other hand, let $w \in \mathbb{R}$ be the solar panel output current, then there exists a non-linear smooth map $q\,:\,\mathbb{R}\to \mathbb{R}$ such that $w = q(x)$. The cost function (to be maximised) is represented by the output power $y:=x\,w$ which, constrained on $\Omega$, reads as $y|_\Omega = h(z):= Lz\,q(Lz)$. Figure \ref{fig:PVCells} (top-left) graphically represents  $h(z)$. Use \eqref{eq:BasicES} and \eqref{eq:PCModel}, and let $(x_{b}(t),y_{b}(t))$ be provided by the solutions of 
\begin{equation*}
	\begin{aligned}
		\dot{z}_b = &\, -\gamma y_{b}  u(t)&&z(0)=z_0\\
		\dot{\zeta}_b = &\, A \zeta_b + B z_b&&\zeta_b(0)=\zeta_0\\
		x_{b} =&\, C\zeta_b\\
		y_{b}  =&\, x_{b}\,q(x_{b}) + \nu(t)
	\end{aligned}
\end{equation*}
where $\nu(t) \in \mathbb{R}$ represents the watt-meter noise. 
Moreover, use \eqref{eq:GlobalESGeneric} and \eqref{eq:PCModel}, and let  $(x_{h}(t),y_{h}(t))$ be provided by the solutions of 
\begin{equation*}
	\begin{aligned}
		\dot{z}_h = &\, -\gamma (y_{h}-\bar{y}) u(t)&&z(0)=z_0\\
		\dot{\bar{y}} =&\, \gamma(y_{h}-\bar{y})&&\zeta_h(0)=\zeta_0 \\
		\dot{\zeta}_h = &\, A \zeta_h + B z_h &&\bar{y}(0)=\bar{y}_0\\
		x_{h} =&\, C\zeta_h\\
		y_{h} =&\, x_{h}\,q(x_{h})+ \nu(t).
	\end{aligned}
\end{equation*}
Then, the performance indices of these two schemes are compared in Figure \ref{fig:PVCells} in which the following parameters have been adopted: $\omega = 115.2$ rad/s, $\gamma = 0.5$ J$^{-1}$,  $\delta = 0.05$ V, and $\|\nu(t)\|_\infty = 1$ W. As previously described, the convergence speed of the basic ES \eqref{eq:BasicES} is deeply influenced by the magnitude of the cost function whereas \eqref{eq:GlobalESGeneric} is not. Indeed, as can be seen in the bottom-right plot of Figure \ref{fig:PVCells}, the average speed of the ES with the high-pass filter is nearly constant and motivated by the quasi-linearity of the map $h(x)$ for $x \in [0,\, 20]$ V. At the opposite, and accordingly to the study of the toy-example of Figure \ref{fig:Mr}, as soon as the cost increases the zero-mean oscillations of \eqref{eq:BasicES} dominate the slope-related contents thus stopping the seeking machinery. The final result is that \eqref{eq:GlobalESGeneric} performs much better than \eqref{eq:BasicES}.
\begin{figure}
	\centering
			\begin{subfigure}{.48\columnwidth}
		\includegraphics[width=\textwidth]{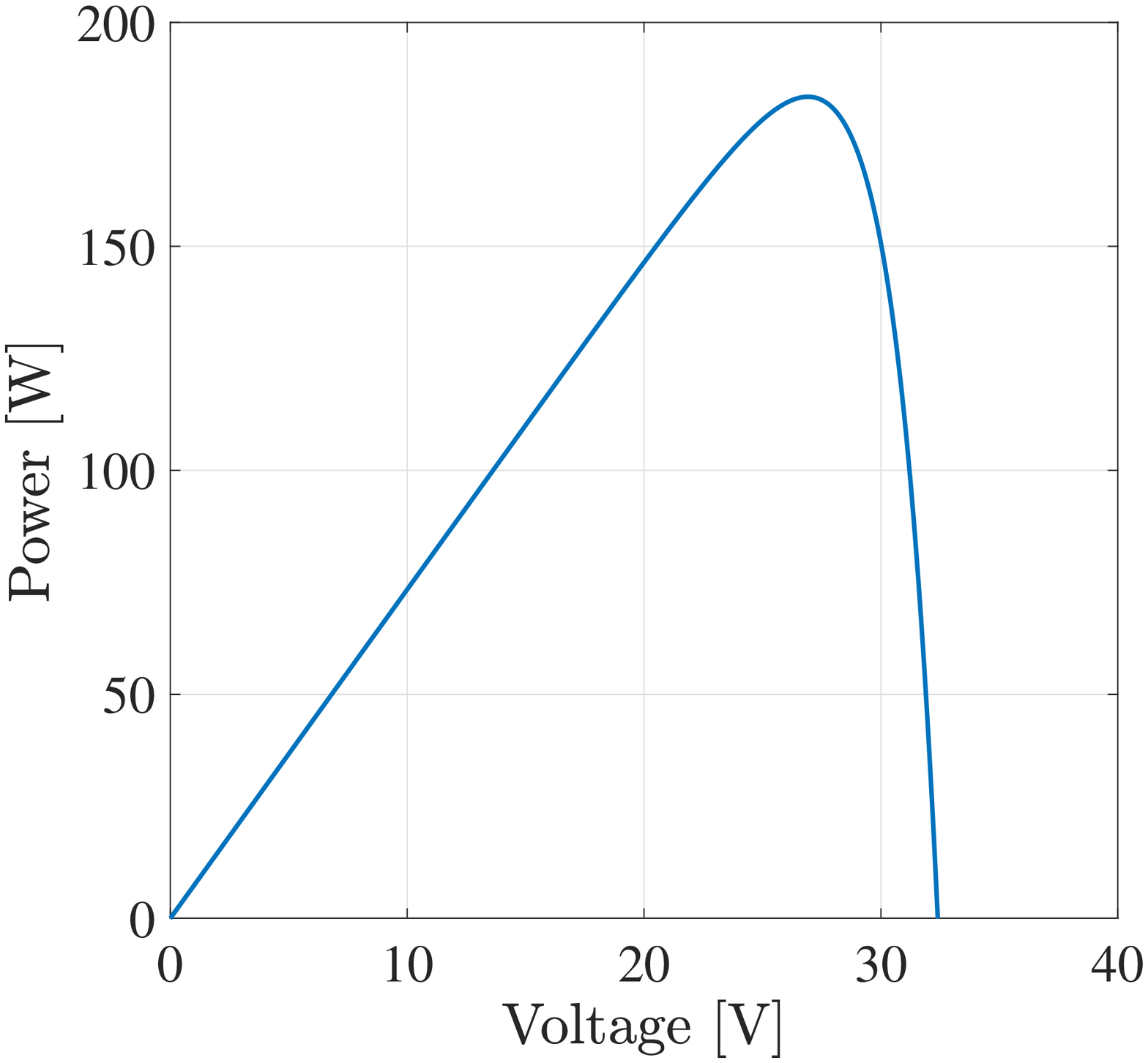}
		\caption{} 
	\end{subfigure}
			\begin{subfigure}{.48\columnwidth}
	\includegraphics[width=\textwidth]{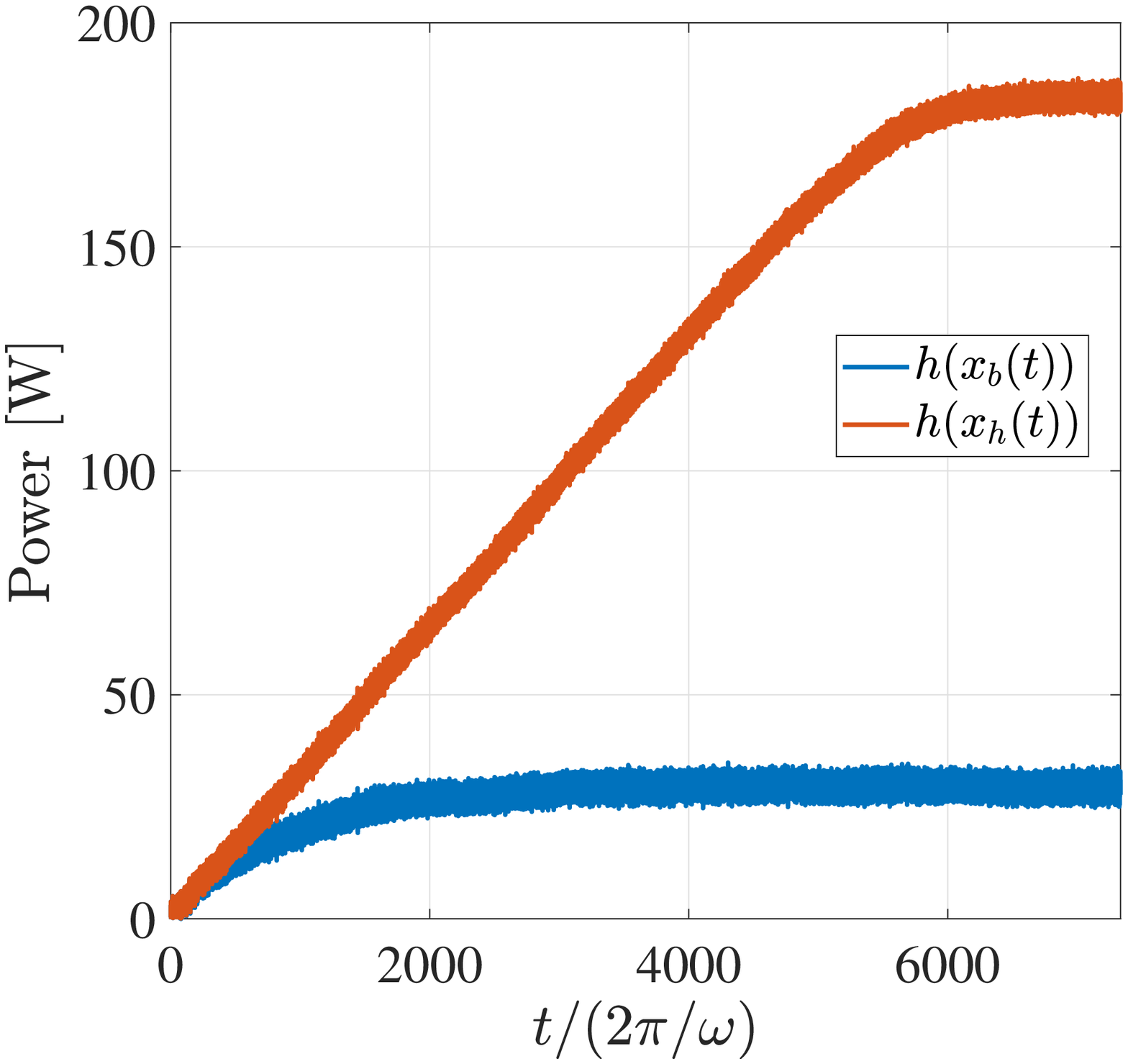}
	\caption{} 
\end{subfigure}
			\begin{subfigure}{.48\columnwidth}
	\includegraphics[width=\textwidth]{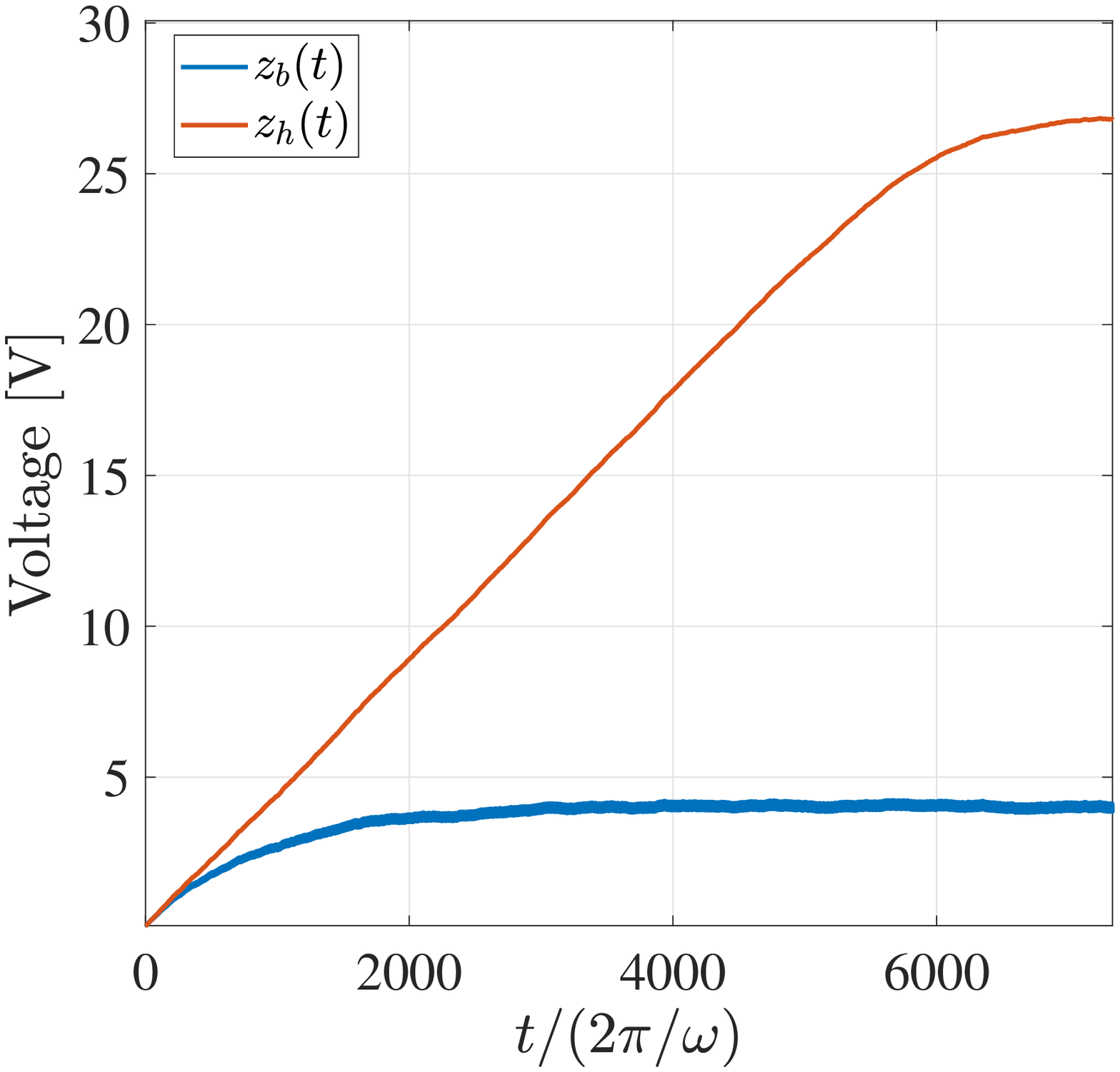}
	\caption{} 
\end{subfigure}
			\begin{subfigure}{.48\columnwidth}
	\includegraphics[width=\textwidth]{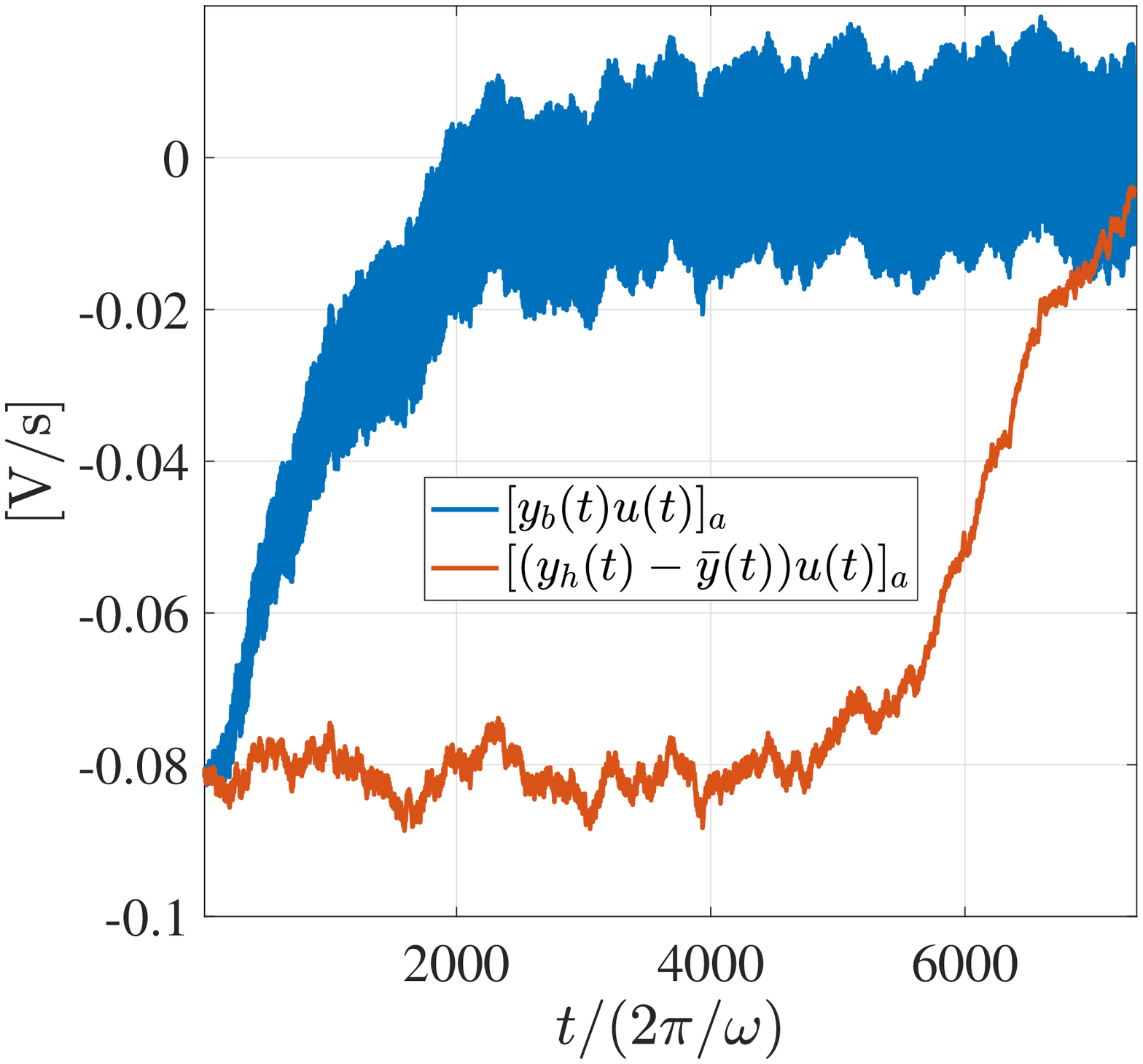}
	\caption{} 
\end{subfigure}
	\caption{(a) Representation of the static map $y|_\Omega=h(z)$. (b) Comparison of the output powers provided by the adoption of the base ES \eqref{eq:BasicES} and the advanced \eqref{eq:GlobalESGeneric}. (c) Comparison of the maximiser voltage determined by the application of \eqref{eq:BasicES} and \eqref{eq:GlobalESGeneric}. (d) Crude (numerical) averaging of $y_{b}(t) u(t)$ and $(y_{h}(t)-\bar{y}(t)) u(t)$.}
	\label{fig:PVCells}
\end{figure}
}

\section{Conclusions}
\label{sec:Concl}

This paper {deals with} two well-known extremum seeking schemes {to show that they work under less restrictive assumptions}. Relying on averaging and Fourier-series arguments, it is demonstrated that these schemes can deal with \ifthenelse{\boolean{CONF}}{{strictly quasi}}{non}-convex cost functions making the global minimiser (assumed to be unique) semi-global practically stable. Moreover, it is shown that the presence of a high-pass filter elaborating the cost function makes the tuning of the parameters independent of the cost magnitude. \ifthenelse{\boolean{CONF}}{}{This feature is then exploited to demonstrate that the global minimiser has a global domain of attraction if the cost function is globally Lipschitz and certain regularity conditions are verified by the average system.}
\vspace{5mm}

\bibliographystyle{ieeetr}
\bibliography{PaperBib}


\ifthenelse{\boolean{CONF}}{

\clearpage
\newpage

\twocolumn[
\begin{@twocolumnfalse}
	{
		\centering
		\LARGE
		Supplement to the paper:\\ Uniform quasi-convex optimisation via Extremum Seeking\\[1em]
	}
	{
		\centering
		Nicola Mimmo, Lorenzo Marconi, and Giuseppe Notarstefano\\[3em]
	}
\end{@twocolumnfalse}
]
}{}

\begin{appendices}

\section{Proof of Lemma \ref{lemma:Reduced1}}
\label{sec:ProofLemmaReduced1}

\ifthenelse{\boolean{CONF}}{Using the fact that
	\[
	\int_{0}^{1} h(s)\sin(2\pi t)\,dt = 0
	\]
we observe{, from \eqref{eq:FourierCoeff},} that 
\[
\dfrac{b_{1,\delta}(x)}{2} =\int_0^1 [h(x+\delta \sin(2\pi t))-h(x)]\sin(2\pi t)\,dt.
\]
This latter term can be split as the sum of the following two terms 
\[
\begin{aligned}
	\int_{0}^{1/2}\left[h(x +\delta\sin(2\pi t))- h(x)\right]\sin(2\pi t)\,dt   \\
	\int_{0}^{1/2}\left[h(x)-h( x-\delta\sin(2\pi t)) \right]\sin(2\pi t)\,dt.
\end{aligned}
\]
By using the first part of Assumption \ref{hyp:Unicity}, with $x_2=x +\delta\sin(2\pi t)$ and $x_1=x$ in the first term and with $x_2 = x$ and $x_1 =x -\delta\sin(2\pi t)$ in the second one,  it's immediately seen that
\[
\int_0^1 [h(x+\delta \sin(2\pi t))-h(x)]\sin(2\pi t)\,dt  \geq \underline{\delta}^\star (\delta) 
\]  	 
for all $x$ such that $x \geq x^\star + \delta$. Similarly, using the second part of  Assumption \ref{hyp:Unicity} with the same choices of $x_1$ and $x_2$, we observe that  
\[
\int_0^1 [h(x+\delta \sin(2\pi t))-h(x)]\sin(2\pi t)\,dt  \leq -\underline{\delta}^\star (\delta)  
\]  	 
for all $x$ such that $x \leq x^\star - \delta$.  Overall, we thus obtain that 
\begin{equation} \label{barb}
	\begin{aligned}
		\dfrac{b_{1,\delta}(x)}{2} &\geq \underline{\delta}^\star (\delta) \qquad &\forall \, x \geq x^\star + \delta\\
		\dfrac{b_{1,\delta}(x)}{2} &\leq  - \underline{\delta}^\star (\delta) \qquad &\forall \, x \leq x^\star - \delta\,.
	\end{aligned}  
\end{equation}
}{
  By letting $\tilde{h}(x):= h(x)-m(x)$ and using the fact that
\[
\int_{0}^{1} m(s)\sin(2\pi t)\,dt = 0
\]
we observe{, from \eqref{eq:FourierCoeff},}  that 
\[
\begin{aligned}
&\dfrac{b_{1,\delta}(x)}{2} =\, \int_0^1 h(x+\delta \sin(2\pi t))\sin(2\pi t)\,dt\\
	=&\, \int_0^1 [\tilde{h}(x+\delta \sin(2\pi t))+m(x+\delta \sin(2\pi t))]\sin(2\pi t)\,dt\\
	=&\, \int_0^1 \tilde{h}(x+\delta \sin(2\pi t))\sin(2\pi t)\,dt\\
		&\, +\int_0^1 [m(x+\delta \sin(2\pi t))-m(x)]\sin(2\pi t)\,dt.
\end{aligned}
\]

As for the first term, note that {$\forall\,x\,\in\mathbb{R}$}
\[
\left|\int_0^1 \tilde{h}(x+\delta \sin(2\pi t))\sin(2\pi t)\,dt\right| \le A\,.
\]
As for the second term, we observe that it 
can be split as the sum of the following two terms 
\[
\begin{aligned}
	 \int_{0}^{1/2}\left[m(x +\delta\sin(2\pi t))- m(x)\right]\sin(2\pi t)\,dt   \\
	 \int_{0}^{1/2}\left[m(x)-m( x-\delta\sin(2\pi t)) \right]\sin(2\pi t)\,dt.
\end{aligned}
\]
By using {point 2)} of Assumption \ref{hyp:Unicity}, with $x_2=x +\delta\sin(2\pi t)$ and $x_1=x$ in the first term, and with $x_2 = x$ and $x_1 =x -\delta\sin(2\pi t)$ in the second one,  {it holds that}
\[
\int_0^1 [m(x+\delta \sin(2\pi t))-m(x)]\sin(2\pi t)\,dt  \geq \underline{\delta}^\star (\delta), 
\]  	 
{with $\underline{\delta}^\star (\delta)$ is defined in \eqref{eq:deltastar}, }for all $x$ such that $x \geq x^\star + \delta$. Similarly, using the second part of  Assumption \ref{hyp:Unicity} with the same choices of $x_1$ and $x_2$, we observe that  
\[
\int_0^1 [m(x+\delta \sin(2\pi t))-m(x)]\sin(2\pi t)\,dt  \leq -\underline{\delta}^\star (\delta)  
\]  	 
for all $x$ such that $x \leq x^\star - \delta$. Overall, we thus obtain that 
\[
\begin{aligned}
	\dfrac{b_{1,\delta}(x)}{2} &\geq -A +  \underline{\delta}^\star (\delta) \qquad &\forall \, x \geq x^\star + \delta\\
	\dfrac{b_{1,\delta}(x)}{2} &\leq A -  \underline{\delta}^\star (\delta) \qquad &\forall \, x \leq x^\star - \delta\,,
\end{aligned}  
\]
namely, using the constraint on $A$ and $\delta$ specified in the claim,
\begin{equation} \label{barb}
	\begin{aligned}
		\dfrac{b_{1,\delta}(x)}{2} &\geq \bar b \qquad &\forall \, x \geq x^\star + \delta\\
		\dfrac{b_{1,\delta}(x)}{2} &\leq  -  \bar b \qquad &\forall \, x \leq x^\star - \delta\,.
	\end{aligned}  
\end{equation}}

Standard arguments can be then used to claim  that the trajectories of (\ref{eq:average-Fourier}) are ultimately bounded and enter in finite time (dependent on $\gamma$)  the set $\bar \Delta=[x^\star-\delta, x^\star+\delta]$. This implies that the omega limit set of the set $\bar \Delta$ of the system (\ref{eq:average-Fourier}), denoted by $\omega(\bar \Delta)$ (see \cite{hale2006dynamics}), is non-empty, it is locally asymptotically stable for the trajectories of (\ref{eq:average-Fourier}) and $\omega(\bar \Delta) \subseteq \bar \Delta$. Furthermore, the arguments in Lemma 1 in \cite{Marconi2010Robust} (see also \cite{marconi2008essential}) show that there exists a set ${\cal A}_\delta \subseteq \bar \Delta$, $\omega(\bar \Delta) \subseteq {\cal A}_\delta $,  which is locally exponentially stable for (\ref{eq:average-Fourier}). This proves item a). 

To prove item b), we exploit the fact that $b_{1,\delta}(s)$ is a continuous function fulfilling (\ref{barb}) and we prove that $b_{1,\delta}(s)$ is a strictly monotonic function for $s \in (x^\star-\delta, x^\star+\delta)$.  

To this end, compute
\[
\dfrac{\partial b_{1,\delta}(x)}{\partial x} = \int_{0}^{1} \left.\dfrac{\partial h(s)}{\partial s}\right|_{x+\delta\sin(2\pi t)}\sin(2\pi t) dt
\]
By Taylor expansion of ${\partial h(s)}/{\partial s}$ around $x$ we obtain  
\[
\begin{aligned}
	\dfrac{\partial h(s)}{\partial s} &= \left.\dfrac{\partial h(s)}{\partial s}\right|_x + \left.\dfrac{\partial^2 h(s)}{\partial s^2}\right|_x(s-x)+ O((s-x)^2).\\
\end{aligned}
\]
By evaluating the previous expression with $s=x+\delta \sin (2\pi t)$,  it turns out that
\[
\begin{aligned}
	\left .\dfrac{\partial h(s)}{\partial s}\right |_{s=x+\delta \sin (2\pi t)} =&\, 
	\left.\dfrac{\partial h(s)}{\partial s}\right|_x \\
	&+
	\left.\dfrac{\partial^2 h(s)}{\partial 		s^2}\right|_{x} \delta \sin (2\pi t)+ O(\delta^2)\\
\end{aligned}
\]

Hence 	
\[
\begin{array}{rcl}
	\dfrac{\partial b_{1,\delta}(x)}{\partial x} &=&
	\displaystyle \left.\dfrac{\partial h(s)}{\partial s}\right|_x   \int_{0}^{1} \sin(2\pi t) +\\
	&& \displaystyle \delta \left.\dfrac{\partial^2 h(s)}{\partial s^2}\right|_{x} \int_{0}^{1} \sin^2(2\pi t) dt + O(\delta^2)\\[3mm]
	&=& \delta \pi  \left.\dfrac{\partial^2 h(s)}{\partial s^2}\right|_{x} + O(\delta^2)\,.
\end{array}
\]

By Assumption 2, there exists a constant $\bar c>0$ such that $\left.{\partial^2 h(s)}/{\partial s^2}\right|_{x^\star} =\bar c$. Furthermore,  continuity of ${\partial^2 h(s)}/{\partial s^2}$  implies that there exists an interval $\Delta$, independent of $\delta$ and with $x^\star \in \Delta$,  such that $\dfrac{\partial^2 h(s)}{\partial s^2} \geq \bar c/2$ for all $s \in \Delta$. From this, it follows that
\[
\dfrac{\partial b_{1,\delta}(x)}{\partial x} \geq \delta \dfrac{\bar c}{4} + O(\delta^2)
\]
for all $x \in \Delta$. The result claimed in item b then follows by taking $\delta^\star$ sufficiently small so that $[x^\star-\delta^\star, x^\star+\delta^\star] \subseteq \Delta$ and  $|O(\delta^2)| < \delta {\bar c}/{4}$. 
To prove the last point of item b) we evaluate 
	\[
	\begin{aligned}
		b_{1,\delta}(x^\star) =& \, 2\int_{0}^{1}h(x^\star+\delta\sin(2\pi t))\sin(2\pi t)\,dt\\
		=&\, 2\int_{0}^{1/2}h(x^\star+\delta\sin(2\pi t))\sin(2\pi t)\,dt  \\
		&- 2\int_{0}^{1/2} h(x^\star-\delta\sin(2\pi t))\sin(2\pi t)\,dt 
	\end{aligned}
	\]
	in which we have exploited $\sin(2\pi t) = -\sin(2\pi t + \pi)$ for all $t \in \mathbb{R}$. Then, assuming $\delta \le \delta^\star$,  and $h(x^\star-s) = h(x^\star+s)$ for any $s \in [0,\,\delta]$ it follows that 
	\[
	\begin{aligned}
		b_{1,\delta}(x^\star) =&\, 2\int_{0}^{1/2}h(x^\star+\delta\sin(2\pi t))\sin(2\pi t)\,dt\\ 
		&- 2\int_{0}^{1/2}h(x^\star+\delta\sin(2\pi t))\sin(2\pi t)\,dt =0.
	\end{aligned}
	\]

\section{Proof of Proposition \ref{prop:BasicES}}
\label{app:Detailed_BasicESNotGlobal_R}
The proof of Proposition \ref{prop:BasicES} follows by standard averaging results as C1) and C2) of Lemma \ref{lemma:Averaging} (or, for instance,  Theorem 2.6.1 and Theorem 4.2.1 in \cite{Sanders1985Averaging}) and it is thus just sketched with a particular eye in showing the dependence of certain key quantities from $M_r$.   

In the first part of the proof  we show that trajectories of \eqref{eq:BasicES} and \eqref{eq:average-Fourier} originating from $x(0) = x_a(0)$ are arbitrary closed for an arbitrary large finite timespan if $\gamma$ is taken sufficiently small. By using the definition of $r_0$, the fact that ${\cal A}_\delta \subseteq [x^\star-	\delta, x^\star+\delta]$ and $| x_a(t)|_{{\cal A}_\delta} \leq \beta(| x_a(0)|_{{\cal A}_\delta},0)$ for all $t\geq 0$, it turns out that  
\[
\begin{aligned}
	|x_a(t)-x^\star| & \leq \beta(|x_a(0)|_{{\cal A}_\delta},0) + \delta \leq \beta(|x_a(0) - x^\star|,0) +\delta\\
	& \leq
	\beta(r_0,0) + \delta = r-d-\delta 
\end{aligned}
\]
for all $t \geq 0$. This implies that $|x_a(t) + \delta u(t) - x^\star| \leq r - d < r$ for all $t\geq 0$. Furthermore, for all $x(t)$ such that $|x(t) - x_a(t)| \leq d$ we have  $|x(t) + \delta u(t) - x^\star| \leq r$. These facts will used later in conjunction with Assumption \ref{hyp:Unicity}.

Let
\begin{equation}
	\label{eq:epsilonBasicES0}
	\epsilon(x_a,t):= \int_0^t h(x_a+\delta u(\tau))u(\tau)-\dfrac{b_{1,\delta}(x_a)}{2} \,d\tau
\end{equation}
and note that $\epsilon(x_a,\cdot)$ is $1$-periodic with $\epsilon(x_a,n)=0$ for all $n \in \mathbb{N}$. 
\ifthenelse{\boolean{LONG}}{
	Pick $N \in \mathbb{N}$ as the largest integer lower than $t$ and decompose the domain $[0,\,t]$ as the union of $[n-1, n)$ and $[N,\,t]$, with $n = 1, \cdots, N$. Then
	\begin{equation}
		\epsilon(x_a,t)	=\int_N^t h(x_a+\delta u(\tau))u(\tau)-\dfrac{b_{1,\delta}(x_a)}{2} \,d\tau \label{eq:epsilonBasicES}.
	\end{equation}
	Now, add and subtract $h(x_a)\,u(\tau)$ into the integral \eqref{eq:epsilonBasicES}, and note that
	\begin{equation}
		\begin{aligned}
			\left|(h(x_a+\delta u(\tau))\pm h(x_a))u(\tau)-\dfrac{b_{1,\delta}(x_a)}{2}\right| \le\\
			\left|(h(x_a+\delta u(\tau))- h(x_a))u(\tau)\right|+\left|\dfrac{b_{1,\delta}(x_a)}{2}\right|\\
			+|h(x_a)||u(\tau)|.
		\end{aligned}
	\end{equation}
	By Assumption \ref{hyp:Existence}, the first term in the previous relation can be bounded as 
	\begin{equation}
		\label{eq:Lipschitz}
		\left|(h(x_a+\delta u(\tau))- h(x_a))u(\tau)\right| \le L_r  \delta.
	\end{equation}
	Similarly, by definition \eqref{eq:FourierCoeff} and Assumption \ref{hyp:Existence}, the second term can be bounded as 
	\begin{equation}
		\label{remark1}
		\begin{aligned}
			\dfrac{|b_{1,\delta}(x_a)|}{2} =&\, \left|\int_0^1 h(x_a+\delta u(t)) u(t)\,dt\right| \\
			=&\,\left|\int_0^1 (h(x_a+\delta u(t))-h(x_a))u(t)\,dt\right| \\
			\le&\,\int_0^1 \left|h(x_a+\delta u(t))-h(x_a)\right|\,dt \le L_r \delta. 
		\end{aligned}
	\end{equation}
	As for the third term, we  have $|h(x_a)||u(\cdot)| < M_r$.	Overall,
	$|\epsilon(x_a,t)|$ cen be bounded as
}{Add and subtract $h(x_a)u(\tau)$ into the integral of \eqref{eq:epsilonBasicES0}, use the triangle inequality and Assumption \ref{hyp:Unicity} to bound $|\epsilon(x_a,t)|$ as}
\begin{align}
	&|\epsilon(x_a,t)|\le  \bar \epsilon(L_r,M_r, \delta):=2 L_r \delta + M_r  \label{eq:epsilonBasicESBound}\,.		
\end{align}

Define $z(t) := x_a(t)-\gamma\epsilon(x_a(t),t)$, and note that,  using \eqref{eq:epsilonBasicESBound},
\begin{equation} \label{x-xa}
	\begin{aligned}
		|x(t)-x_a(t)|&=|x(t)-z(t)-x_a(t)+z(t)| \\ & \leq  |x(t)-z(t)| +\gamma \bar \epsilon(L_r,M_r, \delta).
	\end{aligned}
\end{equation}
We consider now the relation
\[
x(t)-z(t) = x(0)-z(0)+\int_0^t \dot{x}(\tau)-\dot{z}(\tau) \,d\tau
\]
and we observe that 
\begin{equation}
	\label{eq:xzDiff}
	\begin{aligned}
		\dot{x}(t)&-\dot{z}(t) =\\ &-\gamma\left(h(x(t)+\delta u(t))u(t)-\dfrac{b_{1,\delta}(x_a(t))}{2}\right)\\
		&-\left.\gamma^2\dfrac{\partial\epsilon(x,t)}{\partial x}\right|_{x = x_a(t)} \dfrac{b_{1,\delta}(x_a(t))}{2}\\
		& +\gamma\left.\dfrac{\partial \epsilon(x,t) }{\partial t}\right|_{x = x_a(t)}. 
	\end{aligned}
\end{equation}
\ifthenelse{\boolean{LONG}}{
	By adding and subtracting the term
	\begin{equation}
		\gamma \,(\,h(z(t)+\delta u(t)) + h(x_a(t)+\delta u(t)) \,)\,u(t)
	\end{equation}
	to \eqref{eq:xzDiff} and exploiting \eqref{eq:epsilonBasicES0}, we obtain
	\begin{equation}
		\label{eq:xzDiff2}
		\begin{aligned}
			\dot{x}(t)&-\dot{z}(t) =\\
			& -\gamma(h(x(t)+\delta u(t))-h(z(t)+\delta u(t)))u(t) \\
			&-\gamma(h(z(t)+\delta u(t))-h(x_a(t)+\delta u(t)))u(t)\\ 
			&-\left.\gamma^2\dfrac{\partial\epsilon(x,t)}{\partial x}\right|_{x = x_a(t)} \dfrac{b_{1,\delta}(x_a(t))}{2}.
		\end{aligned}
	\end{equation}
	By using Assumption \ref{hyp:Existence} we observe that for all $x(t)$  satisfying $| x(t) - x_a(t)| \leq d$ and for all $z(t)$ satisfying 
	\begin{equation} 
		\label{eq:z+d}
		|z(t) + \delta u(t) - x^\star| \le r
	\end{equation}
	we have
	\begin{equation}
		\label{eq:SomeBounds}
		\begin{aligned}
			|h(x(t)+\delta u(t))-h(z(t)+\delta u(t))| \le &\, L_r |x(t)-z(t)|\\
			|h(z(t)+\delta u(t))-h(x_a(t)+\delta u(t))|\le & \,\gamma L_r |\epsilon(x_a(t),t)|.
		\end{aligned}
	\end{equation}
	Relation (\ref{eq:z+d}) is fulfilled if $|z(t) - x_a(t)| \leq d$, which, by recalling the definition of $z$ and the bound (\ref{eq:epsilonBasicESBound}), is true if
}{Let} $\gamma \leq \gamma^\star_1$ with 
\begin{equation}
	\label{eq:gamma1star}
	\gamma_1^\star := \dfrac{\bar c}{ \bar \epsilon(L_r,M_r, \delta)}
\end{equation}
in which $\bar c$ is any positive number with $\bar c < d$. \ifthenelse{\boolean{LONG}}{
	Now use \eqref{eq:FourierCoeff} into \eqref{eq:epsilonBasicES} and use $|u(\cdot)|\le 1$  to obtain
	\begin{align}
		\left|\left.\dfrac{\partial\epsilon(x,t)}{\partial x}\right|_{x = x_a}\right| \le &\int_N^t \left|\left.\dfrac{\partial h(x)}{\partial x}\right|_{x=x_a+\delta u(\tau)}u(\tau)\right|\,d\tau\nonumber\\
		&+\dfrac{1}{2}\int_N^t \left|\left.\dfrac{\partial b_{1,\delta}(x)}{\partial x}\right|_{x=x_a}  \right|\,d\tau\nonumber\\
		\le &\int_N^t \left|\left.\dfrac{\partial h(x)}{\partial x}\right|_{x=x_a+\delta u(\tau)}\right| \,d\tau\nonumber\\
		&+\int_N^t  \int_{0}^{1} \left|\left.\dfrac{\partial h(x)}{\partial x} \right|_{x=x_a+\delta u(s)}\right| \,ds \,d\tau\nonumber\\			
		\le &\, 2L_r. \label{eq:SomeBounds2}
	\end{align}
	Overall, by using \eqref{remark1}, \eqref{eq:epsilonBasicESBound}, \eqref{eq:SomeBounds} and \eqref{eq:SomeBounds2}, 
}{Then,}
it turns out that 
\ifthenelse{\boolean{LONG}}{
	\[
	\begin{aligned}
		|x(t)-z&(t)| \le |x(0)-z(0)|+\int_0^t \left| \dot{x}(\tau)-\dot{z}(\tau) \right| \,d\tau\\
		\le &\, |x(0)-z(0)|+ \gamma L_r \int_0^t  |x(\tau)-z(\tau)| \,d\tau\\
		&+ \gamma^2  \bar k(L_r,M_r, \delta)  t	
	\end{aligned}
	\]
	where $\bar k(L_r,M_r, \delta) := L_r  \bar \epsilon(L_r,M_r, \delta) + 2 L_r^2 \delta$, namely,  by the  Gronwall Lemma [\cite{Sanders1985Averaging}, \S 1.3] to obtain
}{}
\begin{equation}
	\label{eq:BoundNormXmZGronwall}
	\begin{aligned}
		|x(t)-z(t)| \le & |x(0)-z(0)| e^{\gamma L_r t}\\
		&+\left(e^{\gamma L_r t} -1\right) \dfrac{\gamma \bar k(L_r,M_r, \delta)}{L_r }.
	\end{aligned}
\end{equation}
Now let  $\bar{t} > 0$ and $\bar c \in (0, d)$ be arbitrary numbers, and let $\gamma^\star(\bar t, \bar c, L_r, M_r, \delta)$ be defined as $\gamma^\star := \min\{\gamma_2^\star,\,\gamma_3^\star,\gamma_4^\star\}$ in which 
\begin{equation}\nonumber
	\begin{aligned}
		\gamma_2^\star := \dfrac{\bar c}{3 
			e^{ L_r \bar{t}}}\quad 
		\gamma_3^\star := \dfrac{\bar c L_r }{3 \bar k(L_r,M_r, \delta) (e^{ L_r \bar{t}}-1)}
	\end{aligned}
\end{equation}
and $\gamma_4^\star := \gamma_1^\star / 3$.
Using (\ref{x-xa}) and (\ref{eq:BoundNormXmZGronwall}), simple computations show that  for any $\gamma\in (0,\, \gamma^\star]$ 
\[
|x(t)-x_a(t)| \le  \bar c \quad \mbox{for all } t \in [0, \bar{t}/ \gamma]\,.			  
\]

We use now the local exponential stability of the attractor ${\cal A}_\delta$ to extend the previous bound for any $t \geq \bar{t}/ \gamma$. The assumption in question implies the existence of $\bar c_0>0$, $\kappa >0$, $\lambda>0$ such that the trajectories of \eqref{eq:average-Fourier} starting at $|x_a(0)|_{\mathcal{A}_\delta} \le \bar c_0$ are bounded by
\[
|x_a(t)|_{\mathcal{A}_\delta} \le \kappa |x_a(0)|_{\mathcal{A}_\delta} e^{-\gamma \lambda t}\,.
\]
Without loss of generality let  $\bar c_0 < {d / 2 \kappa}$ and fix $\bar c \in (0, \bar c_0)$. Furthermore. let $\bar{t}_1,\bar{t}_2 > 0$ be such that 
\[
\beta(r, \bar{t}_1) < \bar c_0-\bar c \qquad 
\bar{t}_2  > \dfrac{1}{\lambda} \ln\left(\dfrac{\kappa\, \bar c_0}{\bar c_0-\bar c}\right)\,,
\]
fix $\bar{t} = \max\{\bar{t}_1,\bar{t}_2\}$ and fix once for all $\gamma \in (0,\min\{\gamma^\star,1\})$, with $\gamma^\star(\bar t, \bar c, L_r, M_r, \delta)$ defined before. 
Now divide the time axis into sub-intervals of the form   
\[
I_n:=\left[n\dfrac{\bar{t}}{\gamma},\,(n+1)\dfrac{\bar{t}}{\gamma}\right) \quad n \in {\mathbb{N}}
\]
and, with $x_a(t,x_{a0})$ the trajectory of the average system at time $t$ with initial condition $x_{a0}$ at time $t=0$, let
\[
x_n(t) :=x_a(t-n{\bar{t}}/{\gamma},x(n{\bar{t}}/{\gamma})) \qquad \forall \, t \in I_n.
\]
Because of the first part of the proof and by the definition of $\bar c$, it turns out that  $|x(\bar{t}/\gamma)|_{\mathcal{A}_\delta} < \bar c_0$ and thus $|x_1(t)|_{\mathcal{A}_\delta} < \kappa \bar c_0 e^{- \gamma \lambda (t-\bar{t}/\gamma)}$ for all $t \in I_1$. The same arguments used in the first part of the proof show that   
\[	
|x(t)-x_1(t)| \leq \bar c \qquad \forall\,t \in I_1
\]
Moreover,
\[
\begin{aligned}
	|x(t)|_{\mathcal{A}_\delta} &\le |x_1(t)|_{\mathcal{A}_\delta} + |x(t)-x_1(t)| \\
	&\le \kappa \bar c_0 e^{-\gamma \lambda (t-\bar{t}/\gamma)} + \bar c \\&\le \kappa \bar c_0+\bar c \leq d && \forall \, t \in I_1
\end{aligned}
\]
and
\[
\begin{aligned}
	|x(2\bar{t}/\gamma)|_{\mathcal{A}_\delta} &\le \kappa \bar c_0 e^{-\gamma \lambda \bar{t}/\gamma} + \bar c\\
	&\le \kappa \bar c_0 e^{-\lambda \bar{t}_2} + \bar c <  \bar c_0.
\end{aligned}
\]
The last steps can be then iterated for all $n=2, 3, ...$ and this proves the Proposition.

\section{Proof of Lemma \ref{lemma:Reduced2}}
\label{sec:ProofLemmaReduced2}
This proof follows from standard cascade arguments in view of the result in Lemma \ref{lemma:Reduced1} and from the fact that $\gamma$ is a positive parameter. We only present some detail behind the construction of the (non-unique) function $\tau(\cdot)$.
 By (\ref{barb}) in Lemma \ref{lemma:Reduced1}   we have that the set $\bar \Delta =[x^\star- \delta, x^\star+\delta]$ is forward invariant and reached in a finite time $t^\star(\gamma, x_{a0})$ for the trajectories $x_a(t,x_{a0})$ of  (\ref{eq:DotTildeXAv_2}) with initial condition $x_{a0}$. Let $\hat \Delta$ a superset of $\bar \Delta$ and let $\hat{b}_{1,\delta}\,:\,\mathbb{R}\to \mathbb{R}$ a smooth  function defined as
\[
\hat{b}_{1,\delta}(x) =\left\{\begin{array}{cc}
b_{1,\delta}(x) & x \in \bar \Delta\\
0 & x \not\in \hat \Delta\,.
\end{array}
 \right.
\]
Denote with $\hat{x}_a(t,x_{a0})$ the flow of $\dot{\hat{x}}_a = -\gamma \hat{b}_{1,\delta}(\hat{x}_a)/2$ at time $t$ with initial condition $x_{a0}$ and 
note that for all $x_{a0} \in \bar \Delta$ we have $\hat{x}_a(t, x_{a0}) = {x}_a(t, x_{a0})$. Define the map $\tau\,:\,\mathbb{R} \to \mathbb{R}$ with
\[
\tau(x) = 	\dfrac{\gamma}{2} \int_{-\infty}^0 e^{\gamma\tau} a_{0,\delta}(\hat{x}_a(\tau,x))\,d\tau.
\]
Let $(x_a(t, x_{a0}),\bar{y}_a(t,\bar{y}_{a0},x_{a0}))$ be the solution of \eqref{eq:TildeAV_2} with initial conditions $(x_{a0},\bar{y}_{a0})$ and note that 
for any $x_{a0} \in \mathcal{A}_\delta $ the following  holds
\[
\begin{aligned}
	\bar{y}_a(t,\tau(x_{a0}),x_{a0}) =&\, \tau(x_{a0}) e^{-\gamma t}\\
	&\, + \dfrac{\gamma}{2}\int_{0}^t e^{-\gamma (t-\tau)} a_{0,\delta}(x_a(\tau,x_{a0}))\,d\tau\\
	=&\,   \dfrac{\gamma}{2} \int_{-\infty}^t e^{-\gamma (t-\tau)} a_{0,\delta}(\hat{x}_a(\tau,x_{a0}))\,d\tau \\
	=&\,   \dfrac{\gamma}{2} \int_{-\infty}^0 e^{\gamma \tau} a_{0,\delta}(\hat{x}_a(\tau+t,x_{a0}))\,d\tau \\
		=&\,   \dfrac{\gamma}{2} \int_{-\infty}^0 e^{\gamma \tau} a_{0,\delta}(\hat{x}_a(\tau,\hat{x}_a(t,x_{a0})))\,d\tau \\
		=&\, \tau(\hat{x}_a(t,x_{a0}))\\
		=&\, \tau(x_a(t,x_{a0}))\,.
\end{aligned}
\]
Since $\mathcal{A}_\delta$ is forward invariant for \eqref{eq:DotTildeXAv_2},  this implies that $\mathtt{graph}(\tau|_{\mathcal{A}_\delta})$ is forward invariant for \eqref{eq:TildeAV_2}. Point a) 
 then follows by using, for instance, the arguments in \cite{marconi2007output}. 
As for point b), adopt the same arguments to prove point b) of Lemma \ref{lemma:Reduced1}, use the cascade connection, and the forward invariance of $\mathtt{graph}(\tau|_{\mathcal{A}_\delta})$.

\section{Proof of Theorem \ref{prop:HPFES}}
\label{sec:DetailedProofHPFES}
First, the  trajectories of \eqref{eq:GlobalESGeneric}  are demonstrated to remain close to those of \eqref{eq:TildeAV_2} for a finite timespan. Second, exploiting the local exponential stability of $\mbox{\rm graph} \left. \tau \right |_{{\cal A}_\delta}$, the trajectories of \eqref{eq:GlobalESGeneric} are demonstrated to remain bounded and close to this set for an infinite time horizon.

Define $\eta = (x,\bar{y})$ and $\eta_a = (x_a,\bar{y}_a)$, and 
let
\[
\begin{aligned}
	F(\eta,t) & := \left(\begin{array}{c}
		- \varepsilon_\delta({x},\bar{y},t)   u(t) \\
		\varepsilon_\delta({x},\bar{y},t)
	\end{array}\right)
\end{aligned}
\]
and 
\[
\begin{aligned}
	F_a(\eta_a) & := \left(\begin{array}{c}
		-\dfrac{b_{1,\delta}(x_a)}{2} \\
		-\bar{y}_a + \dfrac{a_{0,\delta}(x_a)}{2}
	\end{array}\right).
\end{aligned}
\]

Moreover, let
\[
\epsilon(\eta_a,t) := \int_0^t F(\eta_a,\tau)-F_a(\eta_a) \,d \tau
\]
and note that $\epsilon(\eta_a,t)$ is $1$-periodic and $\epsilon(\eta_a,n)=0$ for each $n \in\mathbb{N}
$ and for any $\eta_a \in \mathbb{R}^2$. \ifthenelse{\boolean{LONG}}{
	Let $N \in \mathbb{N}$ be the largest integer such that $N \le t$, then
	\[
	\begin{aligned}
		\epsilon(\eta_a,t) =&  \sum_{n = 1}^N\int_{n-1}^{n} F(\eta_a,\tau)-F_a(\eta_a)
		\,d \tau \\
		& +\int_{N}^{t} 
		F(\eta_a,\tau)-F_a(\eta_a)
		\,d \tau\\
		=& \int_{N}^{t} 
		F(\eta_a,\tau)-F_a(\eta_a)
		\,d \tau.
	\end{aligned}
	\] 
}{}
Let  $\epsilon_1(\eta_a,t)$ and $\epsilon_2(\eta_a,t)$ be the first and second entries of $\epsilon(\eta_a,t)$. \ifthenelse{\boolean{LONG}}{
	We now find a bound for these two quantities.
	
	As a starting point, let $e_{y_a}:= a_{0,\delta}(x_a)/2-\bar{y}_a$ and compute
	\begin{equation}
		\label{eq:DynEYA}
			\dot{e}_{y_a}= -\gamma e_{y_a}  -\dfrac{\gamma}{4}\dfrac{\partial a_{0,\delta}(x_a)}{\partial x_a}b_{1,\delta}(x_a) \quad
			e_{y_a}(0) = 0
	\end{equation}
	in which, without loss of generality, we set $\bar{y}_{a0} = a_{0,\delta}(x_a(0))/2$. 
	Now, assume $|x_a|_{\mathcal{A}_\delta} < r$ and use the definition of the Lipschitz constant of $h(\cdot)$ to bound
	\[
	\left|\dfrac{\partial a_{0,\delta}(x)}{\partial x}\right|_{x = x_a}
	\le 2  \int_0^1 \left|\left.\dfrac{\partial h(x)}{\partial x} \right|_{x = x_a+\delta u(t)} \right|\, dt  \le 2  L_r. 
	\]
	Solve \eqref{eq:DynEYA}  and use \eqref{remark1} to obtain
	\begin{equation}
		\label{rem:BoundEYA}
		|{e}_{y_a}(t)| \le 
		\delta  L_r^2\qquad \forall\,t\ge 0.
	\end{equation}
	As for
	\begin{equation}
		\label{eq:epsilon_1}
		\epsilon_1(\eta_a,t)= \int_{N}^{t} 
		\varepsilon_\delta({x}_a,\bar{y}_a,\tau)u(\tau)-b_{1,\delta}(x_a)/2 
		\,d \tau
	\end{equation}
	we add and subtract $h(x_a) u$ and $a_{0,\delta}(x_a) u /2$ inside the integral and we use the triangle inequality to  
	bound
	\begin{equation}
		\label{eq:epsilon0}
		\begin{aligned}
			|\epsilon_1&(\eta_a,t)| \le\\
			& \int_{N}^{t} 		\left|\varepsilon_\delta({x}_a,\bar{y}_a,\tau)u(\tau)-b_{1,\delta}(x_a)/2 \right|
			\,d \tau \le\\
			& \int_{N}^{t} 		
			|(h(x_a+\delta u)-h(x_a))u|
			\,d \tau \\
			& + \int_{N}^{t} 		
			|b_{1,\delta}(x_a)/2|
			\,d \tau \\
			& + \int_{N}^{t} 		
			|h(x_a)-a_{0,\delta}(x_a)/2|\,d \tau \\
			& + \int_{N}^{t} 		
			|a_{0,\delta}(x_a)/2-\bar{y}_a|\,d \tau. 
		\end{aligned}
	\end{equation}
	Then, in this order, for the terms appearing in \eqref{eq:epsilon0} we use the definition of the Lipschitz constant of $h(\cdot)$,  \eqref{remark1} , the Mean Value Theorem, and \eqref{rem:BoundEYA} to obtain
	\begin{align}
		\label{eq:epsilon}
		|\epsilon_1(\eta_a,t)| \le  \delta  L_r ( 3 +  L_r) .
	\end{align}
}{
	To bound $|\epsilon_1(\eta_a,t)|$ add and subtract $h(x_a)$ and $a_{0,\delta}(x_a)u/2$ into the integral of $\epsilon_1(\eta_a,t)$, use the triangle inequality, the Lipschitz constant of $h(\cdot)$, and the Mean Value Theorem. To bound $|\epsilon_2(\eta_a,t)|$ use the Mean Value Theorem. Then we demonstrate that
	\begin{align}
		\label{eq:epsilon}
		|\epsilon_1(\eta_a,t)| \le&\, \delta  L_r ( 3 +  L_r)\\
		\label{eq:BoundEspilon2}
		|\epsilon_2(\eta_a,t)| 
		\le &\, 
		L_r 2\delta.
	\end{align}
}

The benefit introduced by the presence of $\bar{y}$ is evident when \eqref{eq:epsilon} is compared to \eqref{eq:epsilonBasicESBound}. Indeed, in \eqref{eq:epsilon} only the Lipschitz constant of $h(\cdot)$ appears whereas  \eqref{eq:epsilonBasicESBound} is bounded by the supremum of $|h(\cdot)|$ which is radially unbounded. 
\ifthenelse{\boolean{LONG}}{
	As for 
	\begin{equation}
		\label{eq:Epsilon2}
		\epsilon_2(\eta_a,t) = \int_N^t h(x_a+\delta u(\tau))-\dfrac{a_{0,\delta}(x_a)}{2}\,d\tau
	\end{equation}
	thanks to the Mean Value Theorem we note that there exists $\bar{x}_1 \in [x_a-\delta,\, x_a+\delta]$ such that $h(\bar{x}_1)=\dfrac{a_{0,\delta}(x_a)}{2}$. Then
	\begin{equation}
		\label{eq:BoundEspilon2}
		\begin{aligned}
			|\epsilon_2(\eta_a,t)| 
			\le \int_N^t \left|h(x_a+\delta u(\tau))-h(\bar{x}_1)\right| \,d\tau 
			\le L_r 2\delta.
		\end{aligned}
	\end{equation}
}{}
Define $z(t) = \eta_a(t) - \gamma \epsilon(\eta_a(t),t)$, exploit the triangle inequality, and use \eqref{eq:epsilon} and \eqref{eq:BoundEspilon2} to write
\begin{equation}
	\label{eq:ineq}
	\begin{aligned}
		\|\eta(t)&-\eta_a(t)\| = \|\eta(t)-\eta_a(t) \pm z(t)\| \\
		\le  &\|\eta(t)-z(t)\|+\|z(t)-\eta_a(t)\|\\
		\le & \|\eta(t)-z(t)\| +\gamma  \|\epsilon(\eta_a(t),t)\|\\
		\le & \|\eta(t)-z(t)\| +\gamma  \bar{k}_3(L_r,\delta).
	\end{aligned}
\end{equation}
with $\bar{k}_3(L_r,\delta):= \delta L_r \sqrt{4+    (3 +  L_r)^2}$. 
Denote with $z_1$ and $z_2$ the first and second entry of $z$.  
In the next steps we find a bound for $x-z_1$ and $\bar{y}-z_2$ which represent the first and second entry of $\eta-z$. \ifthenelse{\boolean{LONG}}{
	Let us compute 
	\begin{equation}
		\label{eq:XmZ1}
		\begin{aligned}
			x(t)-z_1(t) =& \,x(0)-z_1(0) + \int_0^t\dot{x}(\tau)-\dot{x}_a(\tau)  d\tau\\
			&- \gamma   \int_0^t \dfrac{\partial  \epsilon_1(\eta,t)}{\partial \eta}\Big|_{\eta = \eta_a(\tau)}\dot{\eta}_a(\tau)  d\tau\\
			&- \gamma   \int_0^t  \dfrac{\partial }{\partial \tau} \epsilon_1(\eta_a(\tau),\tau) d\tau.
		\end{aligned}
	\end{equation}
	Then, we use \eqref{eq:GlobalESGeneric_dotx}, \eqref{eq:DotTildeXAv_2}, and \eqref{eq:epsilon_1} to rewrite \eqref{eq:XmZ1} as  
	\begin{equation}
		\label{eq:XmZ1_v2}
		\begin{aligned}
			&x(t)-z_1(t) 	= \,x(0)-z_1(0) \\
			&+ \gamma  \int_0^t
			\varepsilon_\delta(x(\tau),\bar{y}(\tau),\tau)u(\tau)
			-b_{1,\delta}(x_a(\tau))/2 
			d\tau\\
			&- \gamma    \int_0^t \dfrac{\partial  \epsilon_1(\eta,t)}{\partial \eta}\Big|_{\eta = \eta_a(\tau)}\dot{\eta}_a(\tau) d\tau\\
			&- \gamma   \int_0^t
			\varepsilon_\delta(x_a(\tau),\bar{y}_a(\tau),\tau)u(\tau)-b_{1,\delta}(x_a(\tau))/2
			d\tau.	
		\end{aligned}
	\end{equation}
	Add and subtract  the term
	\[
	\varepsilon_\delta(z_1(\tau),\bar{y}(\tau),\tau)u(\tau)
	\]
	inside the first integral of the right member of \eqref{eq:XmZ1_v2} and rearrange this latter as follows
	\begin{align}
		x(t)-&z_1(t) = \,x(0)-z_1(0)\nonumber \\
		+ \gamma  \int_0^t&
		\varepsilon_\delta(x(\tau),\bar{y}(\tau),\tau)u(\tau)
		-\varepsilon_\delta(z_1(\tau),\bar{y}(\tau),\tau)u(\tau)
		d\tau\nonumber\\
		+ \gamma  \int_0^t	&
		\varepsilon_\delta(z_1(\tau),\bar{y}(\tau),\tau)u(\tau)
		-\varepsilon_\delta(x_a(\tau),\bar{y}_a(\tau),\tau)u(\tau)
		d\tau\nonumber\\
		& -\gamma   \int_0^t \left.\dfrac{\partial \epsilon_1(\eta,\tau) }{\partial \eta}\right|_{\eta = \eta_a(\tau)}\dot{\eta}_a(\tau) d\tau.	 \label{eq:XmZ1_v3}
	\end{align}
	In analogy with \eqref{eq:z+d}-\eqref{eq:gamma1star}, let}{Let}  $\bar{c} \in (0, d)$ and 
define
\begin{equation}
	\gamma_0^\star = \dfrac{1}{2 L_r}\left( \sqrt{9+4 \bar{c}/\delta} -3\right).
\end{equation}
Then, \ifthenelse{\boolean{LONG}}{in agreement with the definition of $z_1$, $|z_1(t) - x_a(t)| \leq d$ for all $\gamma \in (0, \gamma_0^\star)$.
}{for all $\gamma \in (0, \gamma_0^\star)$}

\ifthenelse{\boolean{LONG}}{As consequence, 
	\begin{equation} 
		|z_1(t) + \delta u(t) - x^\star| \le r
	\end{equation}
	and so for all $x(t)$ satisfying $| x(t) - x_a(t)| \leq d$ we can exploit the Lipschitz constant of $h(\cdot)$ to derive
	\begin{equation}
		\label{eq:Ineq1}
		|\varepsilon_\delta(x,\bar{y},t)-\varepsilon_\delta(z_1,\bar{y},t)| \le L_r |x-z_1|.
	\end{equation}
	On the other hand, the use of the triangle inequality, the definition of $z_2$, and \eqref{eq:BoundEspilon2} lead to
	\[
	\begin{aligned}
		|\bar{y}-\bar{y}_a| =& \,|\bar{y}-\bar{y}_a\pm z_2| \le |\bar{y}-z_2|+|z_2-\bar{y}_a|\\
		\le &\, |\bar{y}-z_2|+\gamma|\epsilon_2(\eta_a,t)| \\
		\le &\, |\bar{y}-z_2|+\gamma L_r 2\delta 
	\end{aligned}
	\]
	which, with the use of the definitions of $\varepsilon_\delta$ and $z_1$, jointly with \eqref{eq:epsilon}, implies
	\begin{equation}
		\label{eq:Ineq2}
		\begin{aligned}
			|\varepsilon_\delta(z_1,\bar{y}&,t)-
			\varepsilon_\delta(x_a,\bar{y}_a,t)| \\
			\le&L_r \gamma |\epsilon_1(\eta_a,t)|+ |\bar{y}(t)-\bar{y}_a(t)| \\
			\le&   \gamma \delta L_r \left(2   +  L_r\ ( 3 + L_r)\right) + |\bar{y}-z_2|.
		\end{aligned}
	\end{equation}
	
	Moreover, through the triangle inequality we get
	\begin{equation}
		\label{eq:BoundNormPartialEpsilon1}
		\begin{aligned}
			\Bigg|&\dfrac{\partial \epsilon_1(\eta,t)}{\partial \eta}\Bigg|_{\eta = \eta_a}\dot{\eta}_a\Bigg| \le 
			\,
			\gamma
			|a_{0,\delta}(x_a)/2-\bar{y}_a|
			\\
			&+\gamma \left|\left.\dfrac{\partial\varepsilon_\delta(x,\bar{y}_a,t)-b_{1,\delta}(x)/2}{\partial x}\right|_{x = x_a}\right|\dfrac{|b_{1,\delta}(x_a)|}{2}.
		\end{aligned}
	\end{equation}
	Now, since $\partial \varepsilon_\delta(x,\bar{y},t)/\partial x = \partial y_\delta(x)/\partial x$ and  \eqref{eq:average-Fourier} coincides with \eqref{eq:DotTildeXAv_2},  we can exploit \eqref{eq:SomeBounds2} to bound
	\begin{equation}\label{eq:BoundNormPartialEpsilon1x}
		\left|\left.\dfrac{\partial\varepsilon_\delta(x,\bar{y}_a,t)-b_{1,\delta}(x)/2}{\partial x}\right|_{x = x_a}\right| \le 
		2 L_r .
	\end{equation}
	Then, substitute \eqref{eq:BoundNormPartialEpsilon1x} into \eqref{eq:BoundNormPartialEpsilon1} and use \eqref{remark1}  and \eqref{rem:BoundEYA} to bound 
	\begin{equation}
		\label{eq:Ineq3}
		\Bigg|\dfrac{\partial \epsilon_1(\eta,t)}{\partial \eta}\Bigg|_{\eta = \eta_a}\dot{\eta}_a \Bigg| \le 
		\gamma \delta    L_r^2 3
	\end{equation}
	and use \eqref{eq:epsilon}, \eqref{eq:Ineq1}, \eqref{eq:Ineq2} and \eqref{eq:Ineq3} to bound \eqref{eq:XmZ1_v3} as
}{}
\begin{equation}
	\label{eq:XZ1}
	\begin{aligned}
		|x&(t)-z_1(t)| \le  \,|x(0)-z_1(0)| + t\gamma^2   \bar{k}_3(L_r,\delta) \\
		&+ \gamma   \int_0^t
		L_r|x(\tau)-z_1(\tau)| + |\bar{y}(\tau)-z_2(\tau)|
		d\tau	
	\end{aligned}
\end{equation}
where $\bar{k}_1(L_r,\delta) :=  \delta  L_r \left( 
2 +   L_r ( 3 + L_r) + 
3L_r  \right)$.
\ifthenelse{\boolean{LONG}}{
	We now use the same conceptual steps \eqref{eq:XmZ1}-\eqref{eq:XZ1} to find a bound for $|\bar{y}(\tau)-z_2(\tau)|$. Compute
	\begin{equation}
		\label{eq:YZ2}
		\begin{aligned}
			\bar{y}(t)&-z_2(t)  = \bar{y}(0)-z_2(0) + \int_0^t \dot{y}(\tau)-\dot{z}_2(\tau) \,d\tau\\
			=& \, \bar{y}(0)-z_2(0) \\
			& + \gamma \int_0^t h(x(\tau)+\delta u(\tau)) - \dfrac{a_{0,\delta}(x_a(\tau))}{2} \,d\tau \\
			&-\gamma \int_0^t \bar{y}(\tau) -\bar{y}_a(\tau)\,d\tau\\
			&+ \gamma \int_0^t \left.\dfrac{\partial  \epsilon_2(\eta,\tau)}{\partial \eta}\right|_{\eta = \eta_a(\tau)} \dot{\eta}_a(\tau)\,d\tau\\
			&+ \gamma \int_0^t \dfrac{\partial \epsilon_2(\eta_a(\tau),\tau) }{\partial \tau}\,d\tau .
		\end{aligned}
	\end{equation}
	Add and subtract to the first and the second integral
	$h(x_a(\tau)+\delta u(\tau))$ and $z_2(\tau)$ respectively, use \eqref{eq:Epsilon2} and rearrange \eqref{eq:YZ2} as follows
	\begin{equation}
		\label{eq:YZ2_2}
		\begin{aligned}
			\bar{y}(t)&-z_2(t)  =  \bar{y}(0)-z_2(0) \\
			& + \gamma \int_0^t h(x(\tau)+\delta u(\tau)) - h(x_a(\tau)+\delta u(\tau)) \,d\tau \\
			&-\gamma \int_0^t \bar{y}(\tau)-z_2(\tau) \,d\tau-\gamma^2 \int_0^t \epsilon_2(\eta_2(\tau),\tau) \,d\tau\\	
			&+ \gamma \int_0^t \left.\dfrac{\partial  \epsilon_2(\eta,\tau)}{\partial \eta}\right|_{\eta = \eta_a(\tau)} \dot{\eta}_a(\tau)\,d\tau.
		\end{aligned}
	\end{equation}
	Use the defiition of the Lipschitz constant of $h$, the definition of $z_1$, the triangle inequality, \eqref{rem:BoundEYA},  and \eqref{eq:epsilon} to bound
	\begin{equation}
		\label{eq:Bound4}
		\begin{aligned}
			|h(x+&\delta u(\tau)) - h(x_a+\delta u(\tau))| \le L_r |x-x_a| \\
			&\le L_r(|x-z_1|+\gamma |\epsilon_1(\eta_a,\tau)|) \\
			&\le L_r(|x-z_1|+\gamma \delta  L_r ( 3 + L_r)).
		\end{aligned}
	\end{equation}
	On the other hand, we use \eqref{eq:FourierCoeff} into \eqref{eq:Epsilon2}, the Lipschitz constant of $h(\cdot)$, and  \eqref{remark1}  to write
	\begin{equation}
		\label{eq:Bound5}
		\begin{aligned}
			\Bigg|&\left.\dfrac{\partial  \epsilon_2(\eta,\tau)}{\partial \eta}\right|_{\eta = \eta_a(\tau)} \dot{\eta}_a(\tau)\Bigg| \le \\
			&	\gamma\Bigg|\left.\dfrac{\partial  \epsilon_2(\eta,\tau)}{\partial x_a}\right|_{\eta = \eta_a(\tau)}\Bigg| \dfrac{|b_{1,\delta}(x_a)|}{2}  \le \gamma 2 L_r^2 \delta^2. 
		\end{aligned}
	\end{equation}
	Use \eqref{eq:BoundEspilon2}, \eqref{eq:Bound4}, \eqref{eq:Bound5}, to bound \eqref{eq:YZ2_2} as
}{and}
\begin{equation}
	\label{eq:YZ2_3}
	\begin{aligned}
		|\bar{y}(t)&-z_2(t)|  \le | \bar{y}(0)-z_2(0)|  + t\gamma^2  \bar{k}_2(L_r,\delta)\\
		& + \gamma \int_0^t L_r|x-z_1| + |\bar{y}(\tau)-z_2(\tau)|\,d\tau 
	\end{aligned}
\end{equation}
where $\bar{k}_2(L_r,\delta) : =\delta L_r \left(2 +     L_r (2\delta + 3 +  L_r)  \right) $.

Now define $f(t) = |x(t)-z_1(t)|+
|\bar{y}(t)-z_2(t)|$ and $A = \max\{1,L_r\}$, sum \eqref{eq:XZ1} and \eqref{eq:YZ2_3} to obtain
\[
\begin{aligned}
	f(t) \le&\,  f(0)+ t \gamma^2\left(\bar{k}_3(L_r,\delta)+
	\bar{k}_2(L_r,\delta) \right) \\
	&+ \gamma A\int_0^t f(\tau)\,d\tau,
\end{aligned}
\]
and use the specific Gronwall Lemma [\cite{Sanders1985Averaging}, \S 1.3] to determine
\[
f(t) \le \gamma \dfrac{\bar{k}_3(L_r,\delta)+\bar{k}_2(L_r,\delta)}{A} \left(e^{ \gamma A t}-1\right)+  f(0) e^{\gamma A t}
\]
where, since $\epsilon(\eta_a(0),0) = 0$, we have
\[
\begin{aligned}
	f(0) =&\, |x(0)-z_1(0)|+ |\bar{y}(0)-z_2(0)|\\
	=& \, |x(0)-x_a(0)+\gamma \epsilon_1(\eta_a(0),0)|\\
	&\,+ |\bar{y}(0)-\bar{y}_a(0)+\gamma \epsilon_2(\eta_a(0),0)|\\
	=& \, |x(0)-x_a(0)|+ |\bar{y}(0)-\bar{y}_a(0)|.
\end{aligned}
\]
Use \eqref{eq:ineq} and $\|\cdot\|_2 \le \|\cdot\|_1$ to bound 
\[
\begin{aligned}
	\|\eta(t)-\eta_a(t)\|_2 
	\le &\, \|\eta(t)-z(t)\|_1 +\gamma  \bar{k}_3(L_r,\delta) \\
	= &\, f(t) +\gamma  \bar{k}_3(L_r,\delta).
\end{aligned}
\]
Finally, let $\bar{t}> 0$ and define
\[
\begin{aligned}
	\gamma_1^\star &:= \dfrac{\bar{c}}{3}\dfrac{A}{\bar{k}_1(L_r,\delta,\bar{c}/3)+\bar{k}_2(L_r,\delta,\bar{c}/3)} \dfrac{1}{e^{A \bar{t}}-1} \\
	\gamma_2^\star &:= \dfrac{\bar{c}}{3} \dfrac{1}{ e^{A \bar{t}}}\\
	\gamma_3^\star &:= \dfrac{\bar{c}}{3} \dfrac{1}{\bar{k}_3(L_r,\delta)}\\
	\gamma^\star &:=\min\{\gamma_0^\star,	\gamma_1^\star ,\,\gamma_2^\star,\,\gamma_3^\star\}
\end{aligned}
\]
then for any $t \in [0, \bar{t}/\gamma]$, for any $\|\eta(0)-\eta_a(0)\|_1 \le \gamma^\star$, and for any $\gamma \in (0, \gamma^\star)$
\[
\begin{aligned}
	\|\eta(t)-\eta_a(t)\|_2 \le&\, 
	f(t) +\gamma  \bar{k}_3(L_r,\delta)\\
	\le&\,  \gamma \dfrac{\bar{k}_3(L_r,\delta)+\bar{k}_2(L_r,\delta)}{A} \left(e^{ \gamma A t}-1\right)\\
	&\,+  f(0) e^{\gamma A t} +\gamma  \bar{k}_3(L_r,\delta)\\
	\le&\,  \gamma_1^\star \dfrac{\bar{k}_3(L_r,\delta)+\bar{k}_2(L_r,\delta)}{A} \left(e^{ \gamma A t}-1\right)\\
	&\,+  \gamma_2^\star e^{\gamma A t} +\gamma_3^\star  \bar{k}_3(L_r,\delta)\\
	\le &\, \dfrac{\bar{c}}{3}\dfrac{\bar{k}_3(L_r,\delta)+\bar{k}_2(L_r,\delta)}{\bar{k}_1(L_r,\delta,\bar{c}/3)+\bar{k}_2(L_r,\delta,\bar{c}/3)} \dfrac{e^{ \gamma A t}-1}{e^{A \bar{t}}-1} \\
	&\,+  \dfrac{\bar{c}}{3} \dfrac{e^{\gamma A t}}{ e^{A \bar{t}}} + \dfrac{\bar{c}}{3}\dfrac{\bar{k}_3(L_r,\delta)}{\bar{k}_3(L_r,\delta)} \le 
	\bar{c}.
\end{aligned}
\]
It is worth noting that $A$, $\bar{k}_1$, $\bar{k}_2$, and $\bar{k}_3$ are not dependent on the cost function magnitude $M_r$. To stress this point let $\gamma^\star(\bar t, \bar c, L_r, \delta) := \min\{\gamma_0^\star,	\gamma_1^\star ,\,\gamma_2^\star,\,\gamma_3^\star\}$.

Finally, we use the same arguments adopted at the end of the proof of Proposition \ref{prop:BasicES} to exploit  the local exponential stability of the attractor $\mbox{\rm graph} \left. \tau \right |_{{\cal A}_\delta}$ to extend the previous bound for any $t \geq \bar{t}/ \gamma$.

\section{Results on Averaging}
\label{sec:Averaging}
Let $f(\cdot,\cdot) \,:\, \mathbb{R}^n\times [0,\infty) \to \mathbb{R}$, let $\gamma > 0$ and consider the system
\begin{subequations}
	\label{eq:AveragingProblem}
\begin{equation}
\dot{x} =  \gamma f(x,t)\qquad x(0) = x_0.
\end{equation}
Let $T > 0$ and assume $f$ to be $T$-periodic in $t$, then the associated \textit{average} system is defined as
\begin{equation}
\dot{x}_a =  \gamma f_a(x_a)\qquad x_a(0) = x_0
\end{equation}
with
\begin{equation}
f_a(x) := \dfrac{1}{T} \int_{0}^{T} f(x,t)\,dt.
\end{equation}
\end{subequations}
The results provided in Lemma \ref{lemma:Averaging} are grounded on the following assumptions:
\begin{itemize}
	\item[A1)] $f(x,t)$ is Lipschitz continuous and bounded on any compact set $\mathcal{D}\subset \mathbb{R}^n$, uniformly in $t$, \textit{i.e.} there exist (finite) $L_\mathcal{D},\,M_\mathcal{D} > 0$ such that
	\begin{itemize}
		\item $
		\|f(x_1,t)-f(x_2,t)\| \le L_{\mathcal{D}}\|x_1-x_2\|
		$
		for all $x_1,x_2 \in \mathcal{D}$ and for all $ t \ge 0$
		\item $\|f(x,t)\| \le M_\mathcal{D}$ for all $x \in \mathcal{D}$ and for all $ t \ge 0$
	\end{itemize}
	\item[A2)]  There exists  $\mathcal{A} \subset \mathbb{R}^n$, and $\beta(\cdot,\cdot) \in \mathcal{KL}$ such that 
	\[
	\|x_a(t)\|_{\mathcal{A}} \le \beta(\|x_a(0)\|_{\mathcal{A}},\gamma t)\qquad \forall\, x_a(0) \in \mathbb{R}^n.
	\]
	\ifthenelse{\boolean{CONF}}{}{
	\item[A3)] Let A1) be verified for $\mathcal{D} = \mathbb{R}^n$. Let A2) and assume that for any $\rho,d > 0$ with $\rho > d$ there exists (finite) $\bar{t}^\star > 0$ such that 
	\[
	\beta(\rho,\bar{t}) < \rho-d \qquad \forall 	\, \bar{t} > \bar{t}^\star.
	\]
}
\end{itemize}

\begin{lemma}
	\label{lemma:Averaging}
	Let \eqref{eq:AveragingProblem},  then the following claims are true:
	\begin{enumerate}
		\item[C1)] {\bf Closeness.} Let A1) be verified and pick $d >0$ and $\underline{\mathcal{D}} \subset \mathbb{R}^n$. Define $\mathcal{D}$ as the smallest set such that if $\|x\|_{\underline{\mathcal{D}}}\le d$ then $x \in \mathcal{D}$. Then for any $\bar{t} > 0$ there exists $\gamma^\star(d,\bar{t},M_\mathcal{D},L_\mathcal{D}) > 0$ such that for any $\gamma\in (0,\gamma^\star)$
		\[
		 \|x(t)-x_a(t)\| \le d 
		\]
		for all $t \in [0,\,\bar{t}/\gamma]$ such that $x(t),x_a(t) \in \underline{\mathcal{D}}$;
		\item[C2)] {\bf Semi-Global Practical Stability.} Let A1) and A2) be verified. Then, for any $d,\rho >0$ there exists $\bar{t} > 0$ and $\gamma^\star(d,\bar{t},M_\mathcal{D},L_\mathcal{D})$ such that for any $x_0\in \mathbb{R}^n\,:\,\|x_0\|_\mathcal{A}\le \rho$ and for any $\gamma \in (0, \gamma^\star)$ the trajectories $t \mapsto x(t)$ are bounded and
\[
\|x(t)\|_{\mathcal{A}} \le d \qquad \forall\, t \ge \dfrac{\bar{t}}{\gamma}.
\]
\ifthenelse{\boolean{CONF}}{}{
\item[C3)] \label{lemma:Global} {\bf Global Practical Stability.} Let A3) be verified. Then, there exists $\gamma^\star > 0$ such that for any $\gamma\in(0,\gamma^\star)$, and for any $x_0\in \mathbb{R}^n$
the trajectories $t \mapsto x(t)$ are bounded and
\[
\limsup_{t \to \infty} \|x(t)\|_{\mathcal{A}} \le d .
\]
}
	\end{enumerate}
\end{lemma}

\begin{proof} Let A1) be verified. As a starting point, we note that $\|f_a(x)\|\le M_\mathcal{D}$ for all $x \in \mathcal{D}$ thanks to the $t$-uniform boundedness of $f(x,t)$ on $\mathcal{D}$ and thanks to the definition of $f_a(x)$. Let 
\[
\epsilon(x,t) := \int_0^t f(x,\tau) - f_a(x) \,d\tau,
\]
and define $N \in \mathbb{N}$ as the largest integer such that $NT \le t$. Then, thanks to the $T$-periodicity of $f(x,t)$ we have
\[
\epsilon(x,t) = \int_0^t f(x,\tau) - f_a(x) \,d\tau = \int_{NT}^t f(x,\tau) - f_a(x) \,d\tau,
\]
from which, using the triangle inequality,
\[
\epsilon(x,t) \le \int_{NT}^t \|f(x,\tau)\| + \|f_a(x)\| \,d\tau \le 2 M_\mathcal{D} T.
\]
Exploit $z(t) := x_a(t)+\gamma \epsilon(x_a(t),t)$  to bound
\[
\begin{aligned}
\|x(t)-x_a(t)\| &\,\le \|x(t)-z(t)\| + \|z(t)-x_a(t)\| \\
&\, \le \|x(t)-z(t)\|+\gamma \|\epsilon(x_a(t),t)\|  \\
&\, \le \|x(t)-z(t)\|+\gamma 2M_\mathcal{D} T.
\end{aligned}
\]
On the other hand
\[
\begin{aligned}
	x(t)-z(t) = &\, x(0)-z(0) + \int_0^t \dot{x}(\tau)-\dot{z}(\tau)\,d\tau \\
	= &\, x(0)-z(0) 		\\
	&\, + \gamma \int_0^t f(x(\tau),\tau)\,d\tau \\
	&\, - \int_0^t \dot{x}_a(\tau)\,d\tau \\
		&\, - \gamma \int_0^t  \left.\dfrac{\partial \epsilon(s,\tau)}{\partial s}\right|_{s = x_a(\tau)}\dot{x}_a(\tau)\,d\tau \\
			&\, -\gamma \int_0^t \left.\dfrac{\partial \epsilon(x_a(\tau),s)}{\partial s}\right|_{s = \tau}\,d\tau \\
	= &\, x(0)-z(0) 		\\
&\, + \gamma \int_0^t f(x(\tau),\tau)-f(x_a(\tau),\tau)\,d\tau \\
&\, - \gamma^2 \int_0^t  \left.\dfrac{\partial \epsilon(s,\tau)}{\partial s}\right|_{s = x_a(\tau)}f_a(x_a(\tau))\,d\tau \\
\end{aligned}
\]

Add and subtract within the first integral $f(z(\tau),\tau)$, let $\bar{c} \in (0, \,d)$, and $\gamma_1^\star := \bar{c}/(2M_\mathcal{D}T)$. Define $L_\mathcal{D}$ as the Lipschitz constant of $f(x,t)$ for $x \in \mathcal{D}$.  Then, 
for any $\gamma \in (0, \gamma_1^\star)$ and for any $x_a(t) \in \underline{\mathcal{D}}$ we have $z(t) \in \mathcal{D}$ and 
\[
\begin{aligned}
\|f(z(\tau),\tau)-f(x_a(\tau),\tau)\| \le&\, L_\mathcal{D}\|z(\tau)-x_a(\tau)\| \\
\le &\,\gamma 2 L_\mathcal{D} M_\mathcal{D} T\\
\|f(x(\tau),\tau)-f(z(\tau),\tau)\| \le&\, L_\mathcal{D}\|x(\tau)-z(\tau)\|.
\end{aligned}
\] 
Moreover, note that
\[
\begin{aligned}
	\left\|\dfrac{\partial \epsilon(s,t)}{\partial s}\right\| \le  \int_{NT}^t \left\|\dfrac{\partial f(s,\tau)}{\partial s}\right\|+\left\|\dfrac{\partial f_a(s)}{\partial s}\right\| \,d\tau \le 2 L_\mathcal{D} T,
\end{aligned}
\]
then
\[
\begin{aligned}
	\|x(t)-z(t)\| \le &\, \|x(0)-z(0) \|		\\
	&\, + \gamma \int_0^t \|f(x(\tau),\tau)-f(z(\tau),\tau)\|\,d\tau \\
	&\, + \gamma \int_0^t \|f(z(\tau),\tau)-f(x_a(\tau),\tau)\|\,d\tau \\
	&\, + \gamma^2 \int_0^t  \left\|\left.\dfrac{\partial \epsilon(s,\tau)}{\partial s}\right|_{s = x_a(\tau)}\right\|\|f_a(x_a(\tau))\|\,d\tau \\
	\le &\, \|x(0)-z(0) \|		\\
	&\, + \gamma L_\mathcal{D} \int_0^t \|x(\tau)-z(\tau)\|\,d\tau \\
	&\, + \gamma^2 3 L_\mathcal{D} M_\mathcal{D} T  t.
\end{aligned}
\]
Apply the Gronwall Lemma to compute
\[
\begin{aligned}
&\|x(t)-z(t)\| \le  \|x(0)-z(0) \|	+ \gamma^2 3 L_\mathcal{D} M_\mathcal{D} T  t	\\
	& + \gamma L_\mathcal{D} \int_0^t (\|x(0)-z(0) \|	+ \gamma^2 3 L_\mathcal{D} M_\mathcal{D} T  \tau)\\
	& \qquad \cdot e^{\gamma L_\mathcal{D} (t-\tau)}\,d\tau\\
	 & =  \|x(0)-z(0) \| e^{\gamma L_\mathcal{D}t}  + \gamma 3  M_\mathcal{D} T  \left(e^{\gamma L_\mathcal{D}t} -1\right)
\end{aligned}
\]
from which
\[
\begin{aligned}
\|x(t)-x_a(t)\| \le&\, \|x(0)-z(0) \| e^{\gamma L_\mathcal{D}t}  \\
&\,+ \gamma 3  M_\mathcal{D} T  \left(e^{\gamma L_\mathcal{D}t} -1\right)\\
&\,+\gamma 2M_\mathcal{D} T.
\end{aligned}
\]
Then, let $ \bar{t} > 0$, 
\[
\begin{aligned}
	\gamma_2^\star := &\, \dfrac{d}{6 M_\mathcal{D}T e^{L_\mathcal{D}\bar{t}}}\\
	\gamma_3^\star := &\, \dfrac{d}{9  M_\mathcal{D} T  \left(e^{ L_\mathcal{D}\bar{t}} -1\right)}\\
	\gamma_4^\star := &\,\dfrac{d}{6 M_\mathcal{D} T},
\end{aligned}
\]
and define $\gamma^\star(d,\bar{t},M_\mathcal{D},L_\mathcal{D}) =\min\{1,\,\gamma_1^\star,\gamma_2^\star,\gamma_3^\star,\gamma_4^\star\}$. Without loss of generality pick $x_a(0)=x_0$, then for any $\gamma \in(0,\, \gamma^\star)$  
\[
\begin{aligned}
	\|x(t)-x_a(t)\| \le&\, \gamma 2 M_\mathcal{D}T e^{\gamma L_\mathcal{D}t}  \\
	&\,+ \gamma 3  M_\mathcal{D} T  \left(e^{\gamma L_\mathcal{D}t} -1\right)\\
	&\,+\gamma 2M_\mathcal{D} T\\
	\le&\, \gamma_1^\star 2 M_\mathcal{D}T e^{ \gamma L_\mathcal{D}t}  \\
	&\,+ \gamma_2^\star 3  M_\mathcal{D} T  \left(e^{ \gamma L_\mathcal{D}t} -1\right)\\
	&\,+\gamma_3^\star 2M_\mathcal{D} T\\
\le&\, \dfrac{d}{3}\left(\dfrac{e^{ \gamma L_\mathcal{D}t}}{  e^{L_\mathcal{D}\bar{t}}}    
+ \dfrac{e^{\gamma L_\mathcal{D}t} -1}{e^{ L_\mathcal{D}\bar{t}} -1}  +1\right) \le d
\end{aligned}
\]
for all $t \in [0, \bar{t}/\gamma]$ as long as $x(t),x_a(t) \in \underline{\mathcal{D}}$. This proves C1). 

To prove C2) let $$\mathcal{D}:=\{x\in\mathbb{R}^n\,:\,\|x\|_\mathcal{A} \le \beta(\rho,0) + d\}.$$
Let   $\underline{\rho}$ be the largest real such that 
\[
\beta(\underline{\rho},0) < d/2.
\]
Define $\tilde{d} \in (0, \underline{\rho})$ and find $\bar{t}_1^\star > 0$ such that
\[
\beta(\rho,\tau) + \tilde{d} \le \underline{\rho} \qquad \forall\,\tau \ge \bar{t}_1^\star.
\]
Let $\bar{t} > \bar{t}_1^\star$, then for any  $\gamma \in (0,\,\gamma^\star(\tilde{d},\bar{t},M_\mathcal{D},L_\mathcal{D}))$ and any $x_0 \in \mathbb{R}^n$ such that $\|x_0\|_\mathcal{A} \le \rho$ we have
\[
\begin{aligned}
	\|x(t)\|_{\mathcal{A}} & \le \|x(t)-x_a(t)\| + \|x_a(t)\|_\mathcal{A}\\
						   & \le \tilde{d} + \beta(\|x_0\|_{\mathcal{A}},\gamma t) && \forall \, t \in [0, \bar{t}/\gamma]
\end{aligned}
\]
and
\[
	\|x(\bar{t}/\gamma)\|_{\mathcal{A}}  \le  \underline{\rho}. 
\]
Now divide the time axis into sub-intervals of the form   
\[
I_n:=\left[n\dfrac{\bar{t}}{\gamma},\,(n+1)\dfrac{\bar{t}}{\gamma}\right) \quad n \in {\mathbb{N}}
\]
and, with $x_a(t,x_{a0})$ the trajectory of the average system at time $t$ with initial condition $x_{a0}$ at time $t=0$, let
\[
x_n(t) :=x_a(t-n{\bar{t}}/{\gamma},x(n{\bar{t}}/{\gamma})) \qquad \forall \, t \in I_n.
\]
 The same arguments used in the first part of the proof show that   
\[	
\|x(t)-x_n(t)\| \leq \tilde{d} \qquad \forall\,t \in I_n,\,\forall \, n\in\mathbb{N}
\]
Moreover,
\[
\begin{aligned}
\|x(t)\|_{\mathcal{A}} &\le \|x_{n}(t)\|_{\mathcal{A}} + \|x(t)-x_{n}(t)\| \\
	&\le \beta(\|x(n\bar{t}/\gamma)\|_{\mathcal{A}},\gamma t-n\bar{t}) + \tilde{d}\\
	&\le \beta(\|x(n\bar{t}/\gamma)\|_{\mathcal{A}},0) + \tilde{d}&& \forall\,t \in I_n,\,\forall \, n\in\mathbb{N}
\end{aligned}
\]
and
\[
\begin{aligned}
	\|x((n+1)\bar{t}/\gamma)\|_{\mathcal{A}} &\le \beta(\|x(n\bar{t}/\gamma)\|_{\mathcal{A}},\bar{t}) + \tilde{d} && \forall \, n\in\mathbb{N}.
\end{aligned}
\]
Finally, since for $n=1$
\[
\begin{aligned}
	\|x(t)\|_{\mathcal{A}} & \le \beta(\|x(\bar{t}/\gamma)\|_{\mathcal{A}},0) + \tilde{d}\\
						   & \le \beta(\underline{\rho},0) + \tilde{d}\\
						   & \le d && \forall\, t \in I_1
\end{aligned}
\]
and
\[
\begin{aligned}
	\|x(2 \bar{t}/\gamma)\|_{\mathcal{A}} &\le \beta(\|x(\bar{t}/\gamma)\|_{\mathcal{A}},\bar{t}) + \tilde{d} \\
	&\le \beta(\underline{\rho},\bar{t}) + \tilde{d} \\
	&\le \beta(\rho,\bar{t}) + \tilde{d} \\
	& \le \underline{\rho}
\end{aligned}
\]
the arguments can be iterated for any $n > 1$. This proves C2).

To prove C3), define $\underline{\rho}$, $\tilde{d}$ and $\bar{t}_1^\star$  by following the same steps used before to prove C2). Using A3) define $\bar{t}_2^\star > 0$ such that
\[
\beta(r,\bar{t}) < r-\underline{\rho} \qquad \forall \, r > \underline{\rho}, \,\forall	\, \bar{t} > \bar{t}_2^\star.
\]
Define
\[
\bar{t} > \max\{\bar{t}_1^\star,\,\bar{t}_2^\star\}
\]
and pick $\gamma \in (0,\,\gamma^\star(\tilde{d},\bar{t},M_{\mathbb{R}^n},L_{\mathbb{R}^n}))$. As done to prove C2) divide the time axis into sub-intervals $I_n$ and on each interval define $x_n(t)$. As consequence, for any $n \in \mathbb{N}$ and any $x_0 \in \mathbb{R}^n$, we have that $x_{a0} = x_0$ implies
\[
\|x(t)-x_n(t)\| \le \tilde{d} \qquad \forall \, t \in I_n.
\]
Then, as far as $\|x((n \bar{t})/\gamma )\|_{\mathcal{A}} > \underline{\rho}$
\[
\begin{aligned}
	\|x((n+1)\bar{t}/\gamma)\|_{\mathcal{A}} \le &\, \|x_n((n+1)\bar{t}/\gamma)\|_{\mathcal{A}}\\
	&\, +\|x((n+1)\bar{t}/\gamma)-x_n((n+1)\bar{t}/\gamma)\| \\
	\le &\,  \|x_a(\bar{t}/\gamma,x(n \bar{t}/\gamma )\|_{\mathcal{A}}+\tilde{d}\\
	\le &\,  \beta(\|x(n \bar{t}/\gamma )\|_{\mathcal{A}}, \bar{t}/\gamma)+\tilde{d}\\
	\le &\,  \beta(\|x(n \bar{t}/\gamma )\|_{\mathcal{A}}, \bar{t}_2^\star/\gamma)+\tilde{d}\\
	< &\, \|x(n \bar{t}/\gamma )\|_{\mathcal{A}}.
\end{aligned}
\] 
Then we can conclude that as far as $\|x(n \bar{t}/\gamma )\|_{\mathcal{A}} > \underline{\rho}$ the sequence
\[
\{\|x(n\bar{t}/\gamma)\|_{\mathcal{A}}\}_{n\in \mathbb{N}}
\]
is strictly decreasing and
\[
\limsup_{n \to \infty} \|x(n \bar{t}/\gamma )\|_{\mathcal{A}} \le \underline{\rho}.
\]
\ifthenelse{\boolean{CONF}}{}{
	The proof of C3) can be then obtained by following the same final steps used to prove C2).}
\end{proof}
\end{appendices}

\end{document}